\documentclass[12pt,reqno]{amsart}
\usepackage{amsmath}
\usepackage{amssymb}
\usepackage{verbatim}
\usepackage{mathrsfs}
\usepackage{graphicx}
\usepackage{caption}
\usepackage{subcaption}
\usepackage{epstopdf}
\usepackage{makecell}
\usepackage{tabularx}
\usepackage{longtable}
\usepackage{tikz}
\usepackage{pgfplots}
\usepackage{xcolor}
\usepackage{pgfplots}
\usepackage{pgfplotstable}
\usepackage[ruled,vlined]{algorithm2e}
\usepackage[noend]{algpseudocode}
\usepackage[toc]{appendix}
\usepackage[capitalise]{cleveref}
\usepackage[noadjust]{cite}
\usepackage{color}

\usepackage[top=0.6in,bottom=0.5in,left=0.7in,right=0.7in]{geometry}

\newtheorem{lemma}{Lemma}

\newtheorem{theorem}[lemma]{Theorem}
\newtheorem{remark}[lemma]{Remark}

\newtheorem{example}{Example}

\numberwithin{equation}{section}
\numberwithin{lemma}{section}

\newcommand{\N}{\mathbb{N}}    
\newcommand{\NN}{\mathbb{N}_0} 
\newcommand{\R}{\mathbb{R}}    

\newcommand{\nv}{\vec{n}}
\newcommand{\bo}{\mathcal{O}}
\newcommand{\Op}{\Omega^+}
\newcommand{\Om}{\Omega^-}

\newcommand{\be}{ \begin{equation} }
\newcommand{\ee}{ \end{equation} }
\newcommand{\odd}{\operatorname{odd}}
\newcommand{\ind}{\Lambda}

\begin{document}
		
\title[Sixth Order Compact Finite Difference Scheme for Poisson Interface Problem]{Sixth Order Compact Finite Difference Scheme for Poisson Interface Problem with Singular Sources}

\author{Qiwei Feng, Bin Han and Peter Minev}

\thanks{
Research supported in part by
Natural Sciences and Engineering Research Council (NSERC) of Canada}

\address{Department of Mathematical and Statistical Sciences,
University of Alberta, Edmonton,\quad Alberta, Canada T6G 2G1.
\quad {\tt qfeng@ualberta.ca }\qquad {\tt bhan@ualberta.ca}
\quad{\tt minev@ualberta.ca}
}

\begin{abstract}
Let $\Gamma$ be a smooth curve inside
a two-dimensional rectangular region $\Omega$.
In this paper, we consider the Poisson interface problem $-\nabla^2 u=f$ in $\Omega\setminus \Gamma$ with Dirichlet boundary condition such that $f$ is smooth in $\Omega\setminus \Gamma$ and the jump functions $[u]$ and $[\nabla u\cdot \nv]$ across $\Gamma$ are smooth along $\Gamma$.
This Poisson interface problem includes
the weak solution of $-\nabla^2 u=f+g\delta_\Gamma$ in $\Omega$ as a special case.
Because the source term $f$ is possibly discontinuous across the interface curve $\Gamma$ and contains a delta function singularity along the curve $\Gamma$,
both the solution $u$ of the Poisson interface problem and its flux $\nabla u\cdot \nv$ are often discontinuous across the interface. To solve the Poisson interface problem with singular sources, in this paper we propose a sixth order compact finite difference scheme on uniform Cartesian grids.
Our proposed compact finite difference scheme with explicitly given stencils extends the immersed interface method (IIM) to the highest possible accuracy order six for compact finite difference schemes on uniform Cartesian grids, but without the need to change coordinates into the local coordinates as in most papers on IIM in the literature.
Also in contrast with most published papers on IIM, we explicitly provide the formulas
for all involved stencils and therefore, our proposed scheme can be easily implemented and is of interest to practitioners dealing with Poisson interface problems.
Note that the curve $\Gamma$ splits
$\Omega$ into two disjoint subregions $\Op$ and $\Om$.
The coefficient matrix $A$ in the resulting linear system $Ax=b$, following from the proposed scheme, is independent of any source term $f$, jump condition $g\delta_\Gamma$, interface curve $\Gamma$ and Dirichlet boundary conditions, while only $b$ depends on these factors and is explicitly given, according to the configuration of the nine stencil points in $\Op$ or $\Om$.
The constant coefficient matrix $A$ facilitates the parallel implementation of the algorithm in case of a large size matrix and only requires the update of the right hand side vector $b$ for different Poisson interface problems. The scheme can be readily extended to three space dimensions as well.
Due to the flexibility and explicitness of the proposed scheme, it can be generalized to obtain the highest order compact finite difference scheme for non-uniform grids as well. Our numerical experiments confirm the sixth accuracy order of the proposed compact finite difference scheme on uniform meshes for the Poisson interface problems with various singular sources.
\end{abstract}


\maketitle

\maketitle

\pagenumbering{arabic}

\section{Introduction and problem formulation}

Interface problems are common in many practical problems
such as modeling flows in composite materials, oil reservoir simulations, nuclear waste
disposal, and other flows in porous media \cite{HW97}. For example, in groundwater or oil reservoir modelling
the permeability of the porous medium can change drastically across the interfaces between
various geological layers and this can significantly affect the transport process
\cite{PE11}.  The coefficients of heterogeneous and anisotropic
diffusion problems may also be highly oscillatory, and may contain a wide range of
various spatial scales, so
 very fine meshes are required in any standard finite difference/element discretization in order to capture the small scale features.
Thus, speed and storage are two important criteria in choosing suitable
algorithms for solving such problems.

High jumps in the coefficient functions can cause severe
singularities in the exact solutions of the equations \cite{CaKi01,Kell75,KiKo06,Blumen85,Petz01,Petz02,Nic90,KCPK07,NS94,FGW73,Kell71,Kell72}.
In general, the solutions of such problems have limited smoothness and the error
analysis of their approximations by finite elements, \cite{Guermond}, and finite differences, \cite{Samarskii}, for problems with weak solutions could be used.
However, such error estimates show convergence rates that are lower than the observed in the computational practice
for interface problems. Thus, an accurate tailored approximation and an error analysis which takes into account the specificity of such problem
is  an important and challenging task.

One possible approach to the resolution of this issue is provided by the so called
immersed interface methods (IIM) proposed by LeVeque and  Li (see \cite{LiIto06,LeLi94,Li98,Li96} and the references therein).
The main idea behind this approach is to adjust the finite difference approximation of the differential operators in the vicinity of the interface
using Taylor expansions, so that the approximation order remains similar to the order of the approximation in the regions where no singularities are present, thus avoiding the
need of a local grid refinement.
The explicit-jump immersed interface method, introduced in \cite{WB00}, is based on the same idea, however, instead of modification of the discrete operators, it modifies explicitly
the right hand side of the problem, and derives a second order finite difference scheme for problems with discontinuous,
piecewise constant coefficients.  In fact this approach is quite similar to the famous immersed boundary method (IBM) of Peskin \cite{P02}.
Exploiting the idea of the IIM, in \cite{FLQ11} the authors construct a fourth order compact finite difference
method for the Helmholtz equation with discontinuous coefficients across straight vertical line interfaces.
\cite{CFL19} derives a compact finite difference method for piecewise smooth coefficients, so that the solution and its gradient can both achieve a second order of accuracy.
\cite{DFL20} considers anisotropic elliptic interface problems whose coefficient matrix is symmetric semi-positive definite and derives a hybrid discretization involving finite elements away
of the interfaces, and an immersed interface finite difference approximation near or at the interfaces.  The error in the maximum norm is order $O(h^2\log \frac{1}{h})$.
In addition to the treatment of interface problems, Taylor expansions can be used to derive high order compact finite difference schemes for regular elliptic problems.
A family of fourth and sixth order compact finite difference methods for the three-dimensional Poisson equation are derived in
\cite{ZFH13}. \cite{SDKS13} concludes that the highest order for a compact finite difference method for the two-dimensional Poisson's equation on uniform grids is sixth.

Another source of singularity in the solution of elliptic problems is the presence of singularities in the source term.
One possible regularization of Dirac delta functions is analyzed in \cite{TE04}, and \cite{ZZFW06} introduces the
high order matched interface and boundary (MIB) method for  elliptic equations with singular sources.
Similarly, the finite difference discretization of such problems are considered by \cite{Towers10}.
In \cite{BG15} the idea of the IBM is combined with a Discontinuous Galerkin (DG) spatial discretization to derive a high-order method for elliptic problems
with discontinuous coefficients and singular sources.
In \cite{JCL18}, the authors combine the idea of the IIM with a continuous finite element discretization to derive a high order finite element
method for elliptic problems with jumps in the solution and its flux across smooth interfaces.
Elliptic problems with point-located Dirac delta source terms are considered in \cite{EG17}.  A second order approximation to the singular source is combined with
a second order finite difference approximation of the operator on Cartesian grids with hanging nodes, that allow for local refinements around the singular points.
Another finite difference version of the immersed interface method is used to solve the heat diffusion with singular sources in \cite{KV07}.

Singular solutions, induced by discontinuous coefficients of singular sources can also be approximated using a continuous finite element approximation,
 by enriching the basis with  singular functions located in the proper spatial locations, as considered in \cite{CaKi01, Kell75, KiKo06, KCPK07, FGW73, BAGMO96, GO94}.
Alternatively, a posteriori error estimates can be used to devise grid refinement algorithms, as demonstrated for example in \cite{Petz02}, where such estimates are provided in case of interface problems with discontinuous coefficients. 

This short literature review is far from being complete, however, one can conclude that there are not many available discretizations of elliptic problems with singularities, that can achieve higher-than-second order
of accuracy, and not rely on a local grid refinement.  It is particularly challenging to derive a compact finite difference scheme that can achieve the highest possible order of accuracy i.e. a compact scheme of order six.
The derivation of such a scheme is the main goal of the present paper.
In particular, we consider Poisson's equations with discontinuous or singular source terms, such that the solution and the normal component of gradient can be discontinuous across an interface curve.

To fix the ideas, let $\Omega=(l_1, l_2)\times(l_3, l_4)$ be a two-dimensional rectangular region.
Let also $\psi$ be a smooth two-dimensional function and consider a smooth curve $\Gamma:=\{(x,y)\in \Omega: \psi(x,y)=0\}$, which partitions $\Omega$ into two subregions:
$\Op:=\{(x,y)\in \Omega\; :\; \psi(x,y)>0\}$ and $\Om:=\{(x,y)\in \Omega\; : \; \psi(x,y)<0\}$.
We define $f_{\pm}:=f \chi_{\Omega_{\pm}}$ and $u_{\pm}:=u \chi_{\Omega_{\pm}}$.
The particular examples for $\psi(x,y)=x^2+y^2-2$ and $\psi(x,y)=y-\cos(x)$ are illustrated
in \cref{fig:figure0}.
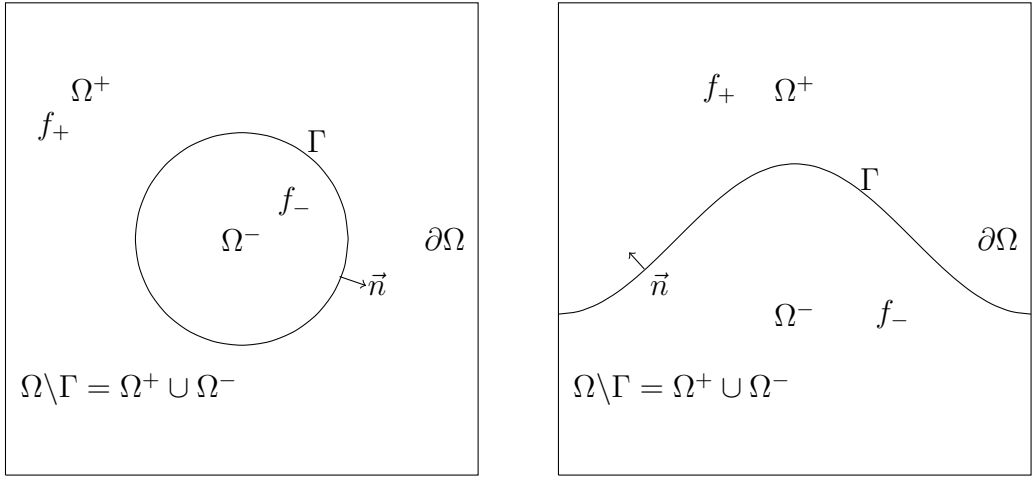
\begin{figure}[htbp]
	\hspace{6mm}
	\vspace{-0.3mm}
	\centering	
\begin{subfigure}[b]{0.4\textwidth}
\begin{tikzpicture}
	\draw[domain =0:360,smooth]
	 plot({sqrt(2)*cos(\x)}, {sqrt(2)*sin(\x)});
	\draw
(-pi, -pi) -- (-pi, pi) -- (pi, pi) -- (pi, -pi) --(-pi,-pi);
	\node (A) at (0,0) {$\Omega^{-}$};
	\node (A) at (-2,2) {$\Omega^{+}$};
	\node (A) at (0.7,0.5) {$f_{-}$};
    \node (A) at (-2.5,1.5) {$f_{+}$}; 	
	\node (A) at (1,1.3) {$\Gamma$}; 	
	\node (A) at (-1.5,-2) {$\Omega\backslash \Gamma={\Omega}^{+} \cup {\Omega}^{-}$};
	\node (A) at (2.7,0) {$\partial\Omega$};
	\node (A) at (1.8,-0.6) {$\nv$};
	\draw[->] (1.3,-0.5)--(1.65,-0.62);
\end{tikzpicture}
\end{subfigure}
\begin{subfigure}[b]{0.4\textwidth}
	\begin{tikzpicture}
	\draw[domain=-pi:pi,smooth] plot(\x,{cos(\x r)});
	\draw
(-pi, -pi) -- (-pi, pi) -- (pi, pi) -- (pi, -pi) --(-pi,-pi);	
    \node (A) at (0,-1) {$\Omega^{-}$};
	\node (A) at (0,2) {$\Omega^{+}$};
	\node (A) at (1.3,-1) {$f_{-}$};
	\node (A) at (-1,2) {$f_{+}$}; 	
	\node (A) at (1,0.8) {$\Gamma$}; 	
	\node (A) at (-1.5,-2) {$\Omega\backslash \Gamma={\Omega}^{+} \cup {\Omega}^{-}$};
	\node (A) at (2.7,0) {$\partial\Omega$};
	\node (A) at (-1.8,-0.6) {$\nv$};
	\draw[->] (-2,-0.4)--(-2.2,-0.18);
	\end{tikzpicture}
\end{subfigure}
	\caption
	{The problem region $\Omega=(-\pi,\pi)^2$ and
the two subregions $\Omega^+=\{(x,y)\in \Omega\; :\; \psi(x,y)>0\}$ and $\Omega^-=\{(x,y)\in \Omega\; :\; \psi(x,y)<0\}$ partitioned by the interface curve $\Gamma=\{(x,y)\in \Omega\; :\; \psi(x,y)=0\}$ with the functions
$\psi(x,y)=x^2+y^2-2$ (left) and $\psi(x,y)=y-\cos(x)$ (right). Note that $\Omega\backslash \Gamma=\Omega^+\cup \Omega^-$.}
\label{fig:figure0}
\vspace{-3.9mm}
\end{figure}
We now state the Poisson interface problem with singular sources as follows:
\begin{equation} \label{Qeques1}
\begin{cases}
-\nabla ^2 u=f &\text{in $\Omega \setminus \Gamma$},\\	
u=g_0 &\text{on $\partial\Omega$},\\
\left[u\right]=g_1 &\text{on $\Gamma$},\\
\left[\nabla  u \cdot \nv \right]=g  &\text{on $\Gamma$},
\end{cases}
\end{equation}
which, if $g_1=0$, can be equivalently rewritten as
\begin{equation} \label{Qeques2}
\begin{cases}
-\nabla^2 u=f-g\delta_{\Gamma}
&\text{in $\Omega$},\\
u=g_0 &\text{on $\partial\Omega$}.
\end{cases}
\end{equation}
Here $\nv$ is the unit normal vector of $\Gamma$ pointing towards $\Op$, and for a point $(x_0,y_0)\in \Gamma$,
\begin{align}
&\hspace{12mm}[u](x_0,y_0):=\lim_{(x,y)\in \Op, (x,y) \to (x_0,y_0) }u(x,y)- \lim_{(x,y)\in \Om, (x,y) \to (x_0,y_0) }u(x,y),\label{jumpCD1}\\
&[ \nabla  u \cdot \nv](x_0,y_0):=  \lim_{(x,y)\in \Op, (x,y) \to (x_0,y_0) } \nabla  u(x,y) \cdot \nv- \lim_{(x,y)\in \Om, (x,y) \to (x_0,y_0) } \nabla  u(x,y) \cdot \nv. \label{jumpCD2}
\end{align}
The conditions in \eqref{jumpCD1} and \eqref{jumpCD2} are called jump conditions for interface problems.

The Poisson interface problem in \eqref{Qeques1} with singular sources arise in many applications. In chemical reaction-diffusion processes, the solution $u$ represents the chemical concentration \cite{KV07,CY93}.
In case of catalytic reactions the catalyst is usually distributed in a very thin layer over an interface $\Gamma$, and therefore the reaction can be considered as occurring on a $d-1$-dimensional manifold in a $d$-dimensional space.  Such reactions result in a continuous chemical concentration $u$,
but a discontinuous gradient $\nabla  u$ across the interface $\Gamma$, i.e., $g_1(x,y)=0$ and $g(x,y)\ne 0$ for all $(x,y)\in \Gamma$.

Problems with discontinuous solutions and/or discontinuous fluxes appear also in multicomponent incompressible flows with or without interfacial tension.  As discussed by \cite{LeLi97},
if surface tension is present at the interface the incompressibility constraint, applied to the momentum equation yields a pressure Poisson equation with a dipole source (the gradient of
the delta function representing the interfacial tension alongside the fluid-fluid interface).  Since such a source function is difficult to approximate, its effect can be modeled via interface jump
conditions for the pressure and its gradient.  The solution for the velocity is always continuous across the interface, however, If the viscosities of the fluids on both sides of the interface differ,
its flux has a jump there.  So, both the velocity and the pressure can be subject to elliptic problems with jumps of the solution or its flux across fluid-fluid or fluid-elastic-structure interfaces.

Similar jumps appear in the solution for the pressure $u$ and velocity $ \vec{v}$, of the set of the Darcy's and continuity equations:
\begin{align}
& \vec{v} =-\frac{k}{\mu}\nabla u, \\
& \nabla \cdot \vec{v} = f,
\end{align}
in case of a discontinuous mobility $\frac{k}{\mu}$ and/or singular source $f$:


In this paper we consider the Poisson interface problem  in \eqref{Qeques1} under the following assumptions:

\begin{itemize}
\item $f$ is smooth in each of $\Op$ and $\Om$, and $f$ may be discontinuous across the interface curve $\Gamma$.

\item All functions $\psi$, $g$, $g_0$  and $g_1$ are smooth.

\item The exact solution $u$ is piecewise smooth in the sense that $u$ has uniformly continuous partial derivatives of (total) order up to seven in each of the subregions $\Op$ and $\Om$.
\end{itemize}

The remainder of the paper is organized as follows. In \cref{sec:sixord}, we  derive the sixth order compact finite difference scheme at regular and irregular points, and discuss their consistency  in \cref{thm:regular}.
and  \cref{fluxtm2}, correspondingly.
Here the center of a stencil is called a regular point if it,  together with all other eight points in the stencil are completely inside $\Op$ or are completely outside $\Op$. Otherwise, it is called an irregular point. We also give a simple proof for the maximum order of compact schemes at regular points. In \cref{thm:interface} we provide an expression for the jump of certain derivatives of the solution, due to the interface conditions.
In 2D there are $72$ different configurations for the stencil, depending on how the interface curve partitions the nine points in it. Up to a symmetry and rigid motion, all configurations at an irregular point can be classified into nine typical cases, see \cref{fig:9Exam1,fig:9Exam2,fig:9Exam3} for a graphical representation of these configurations.
In \cref{sec:Numeri}, we provide numerical experiments to check the convergence rate measured in $l_2$ and $l_{\infty}$ norms. We test the numerical examples in four cases: (1) the exact solution is known and $\Gamma$ does not intersect $\partial \Omega$; (2) the exact solution is known and $\Gamma$ intersects $\partial \Omega$; (3) the exact solution is unknown and $\Gamma$ does not intersect $\partial \Omega$; and (4) the exact solution is unknown and $\Gamma$  intersects $\partial \Omega$. In \cref{sec:Conclu}, we summarize the main contributions of this paper. Finally, in \cref{sec:proof} we shall provide the detailed proof for \cref{thm:interface}, which plays a key role in our development of the compact stencils at irregular points in \cref{sec:sixord}.

\section{Stencils for sixth order compact finite difference schemes using uniform Cartesian grids}
\label{sec:sixord}

Recall that the problem domain is $\Omega=(l_1,l_2)\times (l_3,l_4)$.
Because we shall use uniform Cartesian meshes, we require that the longer side of $\Omega$ is a multiple of the shorter side of $\Omega$. Without loss of generality, we can assume $l_4-l_3=N_0 (l_2-l_1)$ for some positive integer $N_0$.
For any positive integer $N_1\in \mathcal{N}$, we define $N_2:=N_0 N_1$ and then the grid size is  $h:=(l_2-l_1)/N_1=(l_4-l_3)/N_2$.

Let $x_i=l_1+i h$ and $y_j=l_3+j h$ for $i=1,\ldots,N_1-1$ and $j=1,\ldots,N_2-1$.
Because in this paper we are only interested in compact finite difference schemes on uniform Cartesian grids, for a compact stencil centered at the center point $(x_i,y_j)$, the compact stencil involves nine points $(x_i+kh, y_j+lh)$ for $k,l\in \{-1,0,1\}$.
It is convenient to use a level set function $\psi$, which is a two-dimensional smooth function, to describe a given smooth interface curve $\Gamma$ through
\[
\Gamma:=\{(x,y)\in \Omega \; : \; \psi(x,y)=0\}.
\]
Then the interface curve $\Gamma$ splits the problem domain $\Omega$ into two subregions:
$\Op:=\{ (x,y)\in \Omega \; : \; \psi(x,y)>0\}$ and
$\Om:=\{ (x,y)\in \Omega \; : \; \psi(x,y)<0\}$.
Now the interface curve $\Gamma$ splits these nine points into two groups depending on whether these points lie inside $\Op$ or $\Om$.
If a grid point lies on the curve $\Gamma$, then the grid point lies on the boundaries of both $\Op$ and $\Om$. For simplicity we may assume that the grid point belongs to $\Op$ and we can use the interface conditions to handle such a grid point.
Therefore, we naturally define
\[
d_{i,j}^+:=\{(k,\ell) \; : \;
k,\ell\in \{-1,0,1\}, \psi(x_i+kh, y_j+\ell h)\ge 0\}
\]
and
\[
d_{i,j}^-:=\{(k,\ell) \; : \;
k,\ell\in \{-1,0,1\}, \psi(x_i+kh, y_j+\ell h)<0\}.
\]
That is, the interface curve $\Gamma$ splits the nine points in a compact stencil into two disjoint sets  $\{(x_{i+k}, y_{j+\ell})\; : \; (k,\ell)\in d_{i,j}^+\} \subseteq \Op$ and
$\{(x_{i+k}, y_{j+\ell})\; : \; (k,\ell)\in d_{i,j}^-\} \subseteq \Om$.
We say that a grid/center point $(x_i,y_j)$ is
\emph{a regular point} if  $d_{i,j}^+=\emptyset$ or $d_{i,j}^-=\emptyset$.
That is, the center point $(x_i,y_j)$ of a stencil is regular if all its nine points are completely inside $\Op$ (hence $d_{i,j}^-=\emptyset$) or inside $\Om$ (i.e., $d_{i,j}^+=\emptyset$).
See \cref{fig:stenc1} for an example of regular points.
Otherwise, the center point $(x_i,y_j)$ of a stencil is called \emph{an irregular point}
if $d_{i,j}^+\ne \emptyset$ and $d_{i,j}^-\ne \emptyset$. That is, the interface curve $\Gamma$ splits the nine points into two disjoint nonempty sets. As explained before,
up to a symmetry and rigid motion, all the compact stencils at an irregular point can be classified into nine typical cases, see \cref{fig:9Exam1,fig:9Exam2,fig:9Exam3} for these nine typical cases.

\begin{figure}[htbp]
	\centering
	\begin{tikzpicture}[scale = 3]
		\draw[help lines,step = 0.25]
		(-2,-2) grid (0,0);
\draw[line width=2pt, domain=180:270,smooth, variable=\x, red] plot ({sqrt(1.5)*cos(\x)}, {sqrt(1.5)*sin(\x)});
\node at (-1.25-0.25,-1.25-0.25)[circle,fill,inner sep=3pt,color=black]{};
\node at (-1.25-0.25,-1.25)[circle,fill,inner sep=3pt,color=black]{};
\node at (-1.25-0.25,-1.25+0.25)[circle,fill,inner sep=3pt,color=black]{};
\node at (-1.25,-1.25-0.25)[circle,fill,inner sep=3pt,color=black]{};
\node at (-1.25,-1.25)[circle,fill,inner sep=3pt,color=black]{};
\node at (-1.25,-1.25+0.25)[circle,fill,inner sep=3pt,color=black]{};
\node at (-1.25+0.25,-1.25-0.25)[circle,fill,inner sep=3pt,color=black]{};
\node at (-1.25+0.25,-1.25)[circle,fill,inner sep=3pt,color=black]{};
\node at (-1.25+0.25,-1.25+0.25)[circle,fill,inner sep=3pt,color=black]{};
\node at (-0.5-0.25,-0.5-0.25)[circle,fill,inner sep=3pt,color=blue]{};
\node at (-0.5-0.25,-0.5)[circle,fill,inner sep=3pt,color=blue]{};
\node at (-0.5-0.25,-0.5+0.25)[circle,fill,inner sep=3pt,color=blue]{};
\node at (-0.5,-0.5-0.25)[circle,fill,inner sep=3pt,color=blue]{};
\node at (-0.5,-0.5)[circle,fill,inner sep=3pt,color=blue]{};
\node at (-0.5,-0.5+0.25)[circle,fill,inner sep=3pt,color=blue]{};
\node at (-0.5+0.25,-0.5-0.25)[circle,fill,inner sep=3pt,color=blue]{};
\node at (-0.5+0.25,-0.5)[circle,fill,inner sep=3pt,color=blue]{};
\node at (-0.5+0.25,-0.5+0.25)[circle,fill,inner sep=3pt,color=blue]{};
	\end{tikzpicture}
	\caption
	{An example of regular points. The curve in red color is the interface curve $\Gamma$.}
\label{fig:stenc1}
\end{figure}
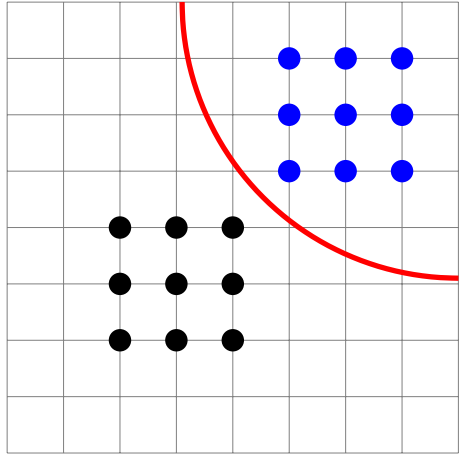

\begin{figure}[h]
	\centering
	\hspace{12mm}	
	 \begin{subfigure}[b]{0.3\textwidth}
	\begin{tikzpicture}[scale = 1]
	\draw[help lines,step = 1]
	(0,0) grid (4,4);
\node at (1,1)[circle,fill,inner sep=1.5pt,color=black]{};
\node at (1,2)[circle,fill,inner sep=1.5pt,color=black]{};
\node at (1,3)[circle,fill,inner sep=1.5pt,color=black]{};	
\node at (2,1)[circle,fill,inner sep=1.5pt,color=black]{};
\node at (2,2)[circle,fill,inner sep=1.5pt,color=black]{};
\node at (2,3)[circle,fill,inner sep=1.5pt,color=black]{};	
\node at (3,1)[circle,fill,inner sep=1.5pt,color=black]{};
\node at (3,2)[circle,fill,inner sep=1.5pt,color=black]{};
\node at (3,3)[circle,fill,inner sep=1.5pt,color=black]{};	
\draw[line width=1.5pt, red]  plot [smooth,tension=0.8]
coordinates {(0,1.8) (1,2.5) (1.4,4)};
    \end{tikzpicture}
	\end{subfigure}
\begin{subfigure}[b]{0.3\textwidth}
	\begin{tikzpicture}[scale = 1]
	\draw[help lines,step = 1]
	(0,0) grid (4,4);
	\node at (1,1)[circle,fill,inner sep=1.5pt,color=black]{};
	\node at (1,2)[circle,fill,inner sep=1.5pt,color=black]{};
	\node at (1,3)[circle,fill,inner sep=1.5pt,color=black]{};	
	\node at (2,1)[circle,fill,inner sep=1.5pt,color=black]{};
	\node at (2,2)[circle,fill,inner sep=1.5pt,color=black]{};
	\node at (2,3)[circle,fill,inner sep=1.5pt,color=black]{};	
	\node at (3,1)[circle,fill,inner sep=1.5pt,color=black]{};
	\node at (3,2)[circle,fill,inner sep=1.5pt,color=black]{};
	\node at (3,3)[circle,fill,inner sep=1.5pt,color=black]{};	
	\draw[line width=1.5pt, red] (1.2,4) parabola bend (2,2.5) (2.8,4);
\end{tikzpicture}
\end{subfigure}
\begin{subfigure}[b]{0.3\textwidth}
	\begin{tikzpicture}[scale = 1]
	\draw[help lines,step = 1]
	(0,0) grid (4,4);
	\node at (1,1)[circle,fill,inner sep=1.5pt,color=black]{};
	\node at (1,2)[circle,fill,inner sep=1.5pt,color=black]{};
	\node at (1,3)[circle,fill,inner sep=1.5pt,color=black]{};	
	\node at (2,1)[circle,fill,inner sep=1.5pt,color=black]{};
	\node at (2,2)[circle,fill,inner sep=1.5pt,color=black]{};
	\node at (2,3)[circle,fill,inner sep=1.5pt,color=black]{};	
	\node at (3,1)[circle,fill,inner sep=1.5pt,color=black]{};
	\node at (3,2)[circle,fill,inner sep=1.5pt,color=black]{};
	\node at (3,3)[circle,fill,inner sep=1.5pt,color=black]{};	
	\draw[line width=1.5pt, red] (0,2.1) .. controls (1,2.3) and (2,2.3) ..(3.37,4);
\end{tikzpicture}
\end{subfigure}	
	\caption
{Examples for irregular points. The curve in red color is the interface curve $\Gamma$.}
\label{fig:9Exam1}
\end{figure}
\begin{figure}[h]
	\centering
	\hspace{12mm}	
	 \begin{subfigure}[b]{0.3\textwidth}
		\begin{tikzpicture}[scale = 1]
			\draw[help lines,step = 1]
			(0,0) grid (4,4);
			\node at (1,1)[circle,fill,inner sep=1.5pt,color=black]{};
			\node at (1,2)[circle,fill,inner sep=1.5pt,color=black]{};
			\node at (1,3)[circle,fill,inner sep=1.5pt,color=black]{};	
			\node at (2,1)[circle,fill,inner sep=1.5pt,color=black]{};
			\node at (2,2)[circle,fill,inner sep=1.5pt,color=black]{};
			\node at (2,3)[circle,fill,inner sep=1.5pt,color=black]{};	
			\node at (3,1)[circle,fill,inner sep=1.5pt,color=black]{};
			\node at (3,2)[circle,fill,inner sep=1.5pt,color=black]{};
			\node at (3,3)[circle,fill,inner sep=1.5pt,color=black]{};	
			\draw[line width=1.5pt, red]  plot [smooth,tension=0.8]
			coordinates {(0,0.2) (1,1.4) (1.8,4)};
		\end{tikzpicture}
	\end{subfigure}
	 \begin{subfigure}[b]{0.3\textwidth}
		\begin{tikzpicture}[scale = 1]
			\draw[help lines,step = 1]
			(0,0) grid (4,4);
			\node at (1,1)[circle,fill,inner sep=1.5pt,color=black]{};
			\node at (1,2)[circle,fill,inner sep=1.5pt,color=black]{};
			\node at (1,3)[circle,fill,inner sep=1.5pt,color=black]{};	
			\node at (2,1)[circle,fill,inner sep=1.5pt,color=black]{};
			\node at (2,2)[circle,fill,inner sep=1.5pt,color=black]{};
			\node at (2,3)[circle,fill,inner sep=1.5pt,color=black]{};	
			\node at (3,1)[circle,fill,inner sep=1.5pt,color=black]{};
			\node at (3,2)[circle,fill,inner sep=1.5pt,color=black]{};
			\node at (3,3)[circle,fill,inner sep=1.5pt,color=black]{};	
			\draw[line width=1.5pt, red]  plot [smooth,tension=0.8]
			coordinates {(0,1.1) (1,1.4) (2,2.4) (2.7,4)};
		\end{tikzpicture}
	\end{subfigure}
	 \begin{subfigure}[b]{0.3\textwidth}
		\begin{tikzpicture}[scale = 1]
			\draw[help lines,step = 1]
			(0,0) grid (4,4);
			\node at (1,1)[circle,fill,inner sep=1.5pt,color=black]{};
			\node at (1,2)[circle,fill,inner sep=1.5pt,color=black]{};
			\node at (1,3)[circle,fill,inner sep=1.5pt,color=black]{};	
			\node at (2,1)[circle,fill,inner sep=1.5pt,color=black]{};
			\node at (2,2)[circle,fill,inner sep=1.5pt,color=black]{};
			\node at (2,3)[circle,fill,inner sep=1.5pt,color=black]{};	
			\node at (3,1)[circle,fill,inner sep=1.5pt,color=black]{};
			\node at (3,2)[circle,fill,inner sep=1.5pt,color=black]{};
			\node at (3,3)[circle,fill,inner sep=1.5pt,color=black]{};	
			\draw[line width=1.5pt, red]  plot [smooth,tension=0.8]
			coordinates {(0.8,0) (1.4,1.4) (1.7,4)};
		\end{tikzpicture}
	\end{subfigure}	
	\caption
{Examples for irregular points. The curve in red color is the interface curve $\Gamma$.}
\label{fig:9Exam2}
\end{figure}
\begin{figure}[h]
	\centering
	\hspace{12mm}	
	 \begin{subfigure}[b]{0.3\textwidth}
		\begin{tikzpicture}[scale = 1]
			\draw[help lines,step = 1]
			(0,0) grid (4,4);
			\node at (1,1)[circle,fill,inner sep=1.5pt,color=black]{};
			\node at (1,2)[circle,fill,inner sep=1.5pt,color=black]{};
			\node at (1,3)[circle,fill,inner sep=1.5pt,color=black]{};	
			\node at (2,1)[circle,fill,inner sep=1.5pt,color=black]{};
			\node at (2,2)[circle,fill,inner sep=1.5pt,color=black]{};
			\node at (2,3)[circle,fill,inner sep=1.5pt,color=black]{};	
			\node at (3,1)[circle,fill,inner sep=1.5pt,color=black]{};
			\node at (3,2)[circle,fill,inner sep=1.5pt,color=black]{};
			\node at (3,3)[circle,fill,inner sep=1.5pt,color=black]{};	
			\draw[line width=1.5pt, red] (0.5,0) .. controls (1.8,1.7) and (2,2.1) .. (2.4,4);
		\end{tikzpicture}
	\end{subfigure}
	 \begin{subfigure}[b]{0.3\textwidth}
		\begin{tikzpicture}[scale = 1]
			\draw[help lines,step = 1]
			(0,0) grid (4,4);
			\node at (1,1)[circle,fill,inner sep=1.5pt,color=black]{};
			\node at (1,2)[circle,fill,inner sep=1.5pt,color=black]{};
			\node at (1,3)[circle,fill,inner sep=1.5pt,color=black]{};	
			\node at (2,1)[circle,fill,inner sep=1.5pt,color=black]{};
			\node at (2,2)[circle,fill,inner sep=1.5pt,color=black]{};
			\node at (2,3)[circle,fill,inner sep=1.5pt,color=black]{};	
			\node at (3,1)[circle,fill,inner sep=1.5pt,color=black]{};
			\node at (3,2)[circle,fill,inner sep=1.5pt,color=black]{};
			\node at (3,3)[circle,fill,inner sep=1.5pt,color=black]{};	
			\draw[line width=1.5pt, red] (0,1.55) .. controls (1,1.5) and (2,2.3) ..(4,3.3);
		\end{tikzpicture}
	\end{subfigure}
	 \begin{subfigure}[b]{0.3\textwidth}
		\begin{tikzpicture}[scale = 1]
			\draw[help lines,step = 1]
			(0,0) grid (4,4);
			\node at (1,1)[circle,fill,inner sep=1.5pt,color=black]{};
			\node at (1,2)[circle,fill,inner sep=1.5pt,color=black]{};
			\node at (1,3)[circle,fill,inner sep=1.5pt,color=black]{};	
			\node at (2,1)[circle,fill,inner sep=1.5pt,color=black]{};
			\node at (2,2)[circle,fill,inner sep=1.5pt,color=black]{};
			\node at (2,3)[circle,fill,inner sep=1.5pt,color=black]{};	
			\node at (3,1)[circle,fill,inner sep=1.5pt,color=black]{};
			\node at (3,2)[circle,fill,inner sep=1.5pt,color=black]{};
			\node at (3,3)[circle,fill,inner sep=1.5pt,color=black]{};	
		   	\draw[line width=1.5pt, red] (0,1.2) .. controls (2,1.3) and (2.5,1.4) ..(2.7,4);
		
		\end{tikzpicture}
	\end{subfigure}	
	\caption
	{Examples for irregular points. The curve in red color is the interface curve $\Gamma$.}
	\label{fig:9Exam3}
\end{figure}

Let $M$ be a positive integer standing for a given desired accuracy order.
Before we discuss the compact  schemes at a regular or an irregular point $(x_i,y_j)$, let us introduce some notations and then outline the main ideas for finding the stencils at a general grid point $(x_i,y_j)$, achieving a given accuracy order $M$.

We first pick up and fix a base point $(x_i^*,y_j^*)$ inside the open square $(x_i-h,x_i+h)\times (y_j-h,y_j+h)$.
That is, we can write
\be \label{base:pt}
x_i^*=x_i-v_0  \quad \mbox{and}\quad y_j^*=y_j-w_0  \quad \mbox{with}\quad
-h<v_0, w_0<h.
\ee
For the sake of brevity, we shall use the following notions:
\be \label{ufmn}
u^{(m,n)}:=\frac{\partial^{m+n} u}{ \partial^m x \partial^n y}(x_i^*,y_j^*)
 \quad\mbox{and}\quad
f^{(m,n)}:=\frac{\partial^{m+n} f}{ \partial^m x \partial^n y}(x_i^*,y_j^*),
\ee
which are just their $(m,n)$th partial derivatives at the base point $(x_i^*,y_j^*)$.
Define $\NN:=\N\cup\{0\}$, the set of all nonnegative integers.
For a nonnegative integer $K\in \NN$, we define
\be \label{Sk}
\ind_{K}:=\{(m,n-m) \; : \; n=0,\ldots,K
\; \mbox{ and }\; m=0,\ldots,n\}, \qquad K\in \NN.
\ee
For a smooth function $u$, its values $u(x+x_i^*,y+y_j^*)$ for small $x,y$ can be well approximated through its Taylor polynomial below:
\be \label{u:approx}
u(x+x_i^*,y+y_j^*)=
\sum_{(m,n)\in \ind_{M+1}} \frac{u^{(m,n)}}{m!n!}x^m y^{n}
+\bo(h^{M+2}), \qquad x, y \in (-2h,2h).
\ee
In other words, in a neighborhood of the base point $(x_i^*,y_j^*)$,
the function $u$ is well approximated and completely determined by the partial derivatives of $u$ of total degree less than $M+2$ at the base point $(x_i^*,y_j^*)$, i.e.,  by the unknown quantities $u^{(m,n)}, (m,n)\in \ind_{M+1}$.
We can similarly approximate $f(x+x_i^*,y+y_j^*)$ for small $x,y$.
%
%
For $x\in \R$, the floor function $\lfloor x\rfloor$ is defined to be the largest integer less than or equal to $x$.
For an integer $m$, we define
\[
\odd(m):=\begin{cases}
0, &\text{if $m$ is even},\\
1, &\text{if $m$ is odd}.
\end{cases}
\]
That is, $\odd(m)=m-2\lfloor m/2\rfloor$ and $\lfloor m/2\rfloor=\frac{m-\odd(m)}{2}$.
Because the function $u$ is a solution to the partial differential equation in \eqref{Qeques1}, we shall see that all the quantities $u^{(m,n)}, (m,n)\in \ind_{M+1}$ are not independent of each other. In fact, we have:

\begin{lemma}\label{lem:uderiv}
Let $u$ be a function satisfying $-\nabla ^2 u=f$ in $\Omega\setminus \Gamma$.
If a point $(x_i^*,y_j^*)\in \Omega\setminus \Gamma$, then
\be \label{uderiv:relation}
u^{(m,n)}=(-1)^{\lfloor\frac{m}{2}\rfloor}
u^{(\odd(m),n+m-\odd(m))}+
\sum_{\ell=1}^{\lfloor m/2\rfloor}
(-1)^{\ell} f^{(m-2\ell,n+2\ell-2)},\qquad \forall\; (m,n)\in \ind_{M+1}^2,
\ee
where the subsets $\ind_{M+1}^1$ and $\ind_{M+1}^2$ of $\ind_{M+1}$ are defined by
\be \label{Sk1}
\ind_{M+1}^2:=\ind_{M+1}\setminus \ind_{M+1}^1\quad \mbox{with}\quad
\ind_{M+1}^1:=\{(\ell,k-\ell) \; :   k=\ell,\ldots, M+1-\ell\; \mbox{and} \;\ell=0,1\; \}.
\ee
\end{lemma}

\begin{proof} By our assumption, we have $u_{xx}+u_{yy}=-f$ in $\Omega\setminus \Gamma$. Therefore, we obtain
\be \label{umn:relation}
u^{(m+2,n)}+u^{(m,n+2)}=-f^{(m,n)},\qquad
\forall\; m,n\in \NN.
\ee
Hence, for $(m,n)\in \ind_{M+1}^2$, we have $m\ge 2$ and
\[
u^{(m,n)}=-f^{(m-2,n)}-u^{(m-2,n+2)},\qquad (m,n)\in \ind_{M+1}^2.
\]
Then we can recursively apply the above identity $\frac{m-\odd(m)}{2}-1$ times to get \eqref{uderiv:relation}.
\end{proof}

For the convenience of the reader,
see \cref{fig:umn1} for an illustration
of the quantities $u^{(m,n)}, (m,n)\in \ind_{M+1}^1$ and
$(m,n)\in \ind_{M+1}^2$ in \cref{lem:uderiv} with $M=6$.

\begin{figure}[h]
	\centering
	\hspace{12mm}	
	\begin{tikzpicture}[scale = 0.8]
		\node (A) at (0,7) {$u^{(0,0)}$};
		\node (A) at (0,6) {$u^{(0,1)}$};
		\node (A) at (0,5) {$u^{(0,2)}$};
		\node (A) at (0,4) {$u^{(0,3)}$};
		\node (A) at (0,3) {$u^{(0,4)}$};
		\node (A) at (0,2) {$u^{(0,5)}$};
		\node (A) at (0,1) {$u^{(0,6)}$};
		\node (A) at (0,0) {$u^{(0,7)}$};
		\node (A) at (2,7) {$u^{(1,0)}$};
		\node (A) at (2,6) {$u^{(1,1)}$};
		\node (A) at (2,5) {$u^{(1,2)}$};
		\node (A) at (2,4) {$u^{(1,3)}$};
		\node (A) at (2,3) {$u^{(1,4)}$};
		\node (A) at (2,2) {$u^{(1,5)}$};
		\node (A) at (2,1) {$u^{(1,6)}$};	 
		\node (A) at (4,7) {$u^{(2,0)}$};
		\node (A) at (4,6) {$u^{(2,1)}$};
		\node (A) at (4,5) {$u^{(2,2)}$};
		\node (A) at (4,4) {$u^{(2,3)}$};
		\node (A) at (4,3) {$u^{(2,4)}$};
		\node (A) at (4,2) {$u^{(2,5)}$};
		\node (A) at (6,7) {$u^{(3,0)}$};
		\node (A) at (6,6) {$u^{(3,1)}$};
		\node (A) at (6,5) {$u^{(3,2)}$};
		\node (A) at (6,4) {$u^{(3,3)}$};
		\node (A) at (6,3) {$u^{(3,4)}$};
		\node (A) at (8,7) {$u^{(4,0)}$};
		\node (A) at (8,6) {$u^{(4,1)}$};
		\node (A) at (8,5) {$u^{(4,2)}$};
		\node (A) at (8,4) {$u^{(4,3)}$};
		\node (A) at (10,7) {$u^{(5,0)}$};
		\node (A) at (10,6) {$u^{(5,1)}$};
		\node (A) at (10,5) {$u^{(5,2)}$};
		\node (A) at (12,7) {$u^{(6,0)}$};
		\node (A) at (12,6) {$u^{(6,1)}$};
		\node (A) at (14,7) {$u^{(7,0)}$};      	 
		\draw[line width=1.5pt, red]  plot [tension=0.8]
		coordinates {(-0.6,7.5) (-0.6,-0.6) (2.6,0.6) (2.6,7.5) (-0.6,7.5)};	
		\draw[line width=1.5pt, blue]  plot [tension=0.8]
		coordinates {(3,7.5)  (3,1.0) (14.6,6.8) (14.6,7.5) (3,7.5)};
		 \draw[decorate,decoration={brace,mirror,amplitude=4mm},xshift=0pt,yshift=10pt,ultra thick] (-0.6,-1) -- node [black,midway,yshift=0.6cm]{} (2.7,-1);
    	\node (A) at (1.3,-1.5)	 {$\{u^{(m,n)}:(m,n)\in \ind_{M+1}^1\}$};
        \draw[decorate,decoration={brace,mirror,amplitude=4mm},xshift=0pt,yshift=10pt,ultra thick] (3,0.5) -- node [black,midway,yshift=0.6cm]{} (14.7,0.5);
       \node (A) at (9,0)	 {$\{u^{(m,n)}:(m,n)\in \ind_{M+1}^2\}$};
	\end{tikzpicture}
	\caption
	{Red trapezoid: $\{u^{(m,n)}:(m,n)\in \ind_{M+1}^1\}$ with $M=6$. Blue trapezoid: $\{u^{(m,n)}:(m,n)\in \ind_{M+1}^2\}$ with $M=6$. Note that $\ind_{M+1}=\ind_{M+1}^1 \cup\ind_{M+1}^2$.}
	\label{fig:umn1}
\end{figure}
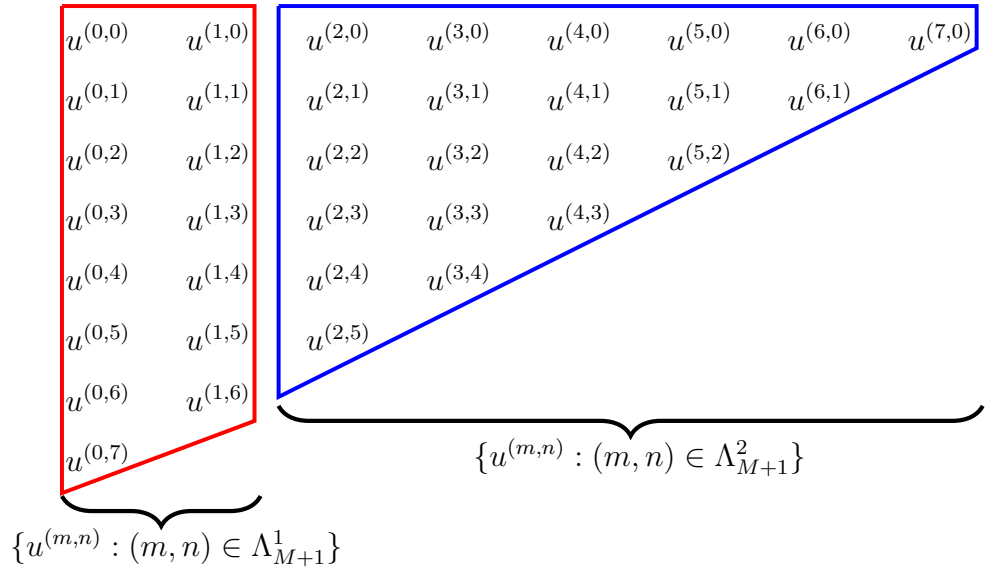

For $M=6$, the identities in \eqref{uderiv:relation} of \cref{lem:uderiv} for $u^{(m,n)}, (m,n)\in \ind_{7}^2$ can be explicitly given by
{\footnotesize{
\be \label{uderiv:M=6}
\begin{split}
    &u^{(2, 1)} = -f^{(0, 1)}-u^{(0, 3)},\quad
u^{(2, 2)} = -f^{(0, 2)}-u^{(0, 4)},\quad
u^{(2, 3)} = -f^{(0, 3)}-u^{(0, 5)},\quad
u^{(2, 4)} = -f^{(0, 4)}-u^{(0, 6)},\\
    & u^{(2, 5)} = -f^{(0, 5)}-u^{(0, 7)},\quad
u^{(3, 0)} = -f^{(1, 0)}-u^{(1, 2)},\quad
u^{(3, 1)} = -f^{(1, 1)}-u^{(1, 3)},\quad
u^{(3, 2)} = -f^{(1, 2)}-u^{(1, 4)},\\
		& u^{(3, 3)} = -f^{(1, 3)}-u^{(1, 5)},\quad  u^{(3, 4)} = -f^{(1, 4)}-u^{(1, 6)},\quad
u^{(4, 0)} = -f^{(2, 0)}+f^{(0, 2)}+u^{(0, 4)},\\
		& u^{(4, 1)} = -f^{(2, 1)}+f^{(0, 3)}+u^{(0, 5)},\quad
u^{(4, 2)} = -f^{(2, 2)}+f^{(0, 4)}+u^{(0, 6)},\quad u^{(4, 3)} = -f^{(2, 3)}+f^{(0, 5)}+u^{(0, 7)}, \\
		& u^{(5, 0)} = -f^{(3, 0)}+f^{(1, 2)}+u^{(1, 4)}, \quad
u^{(5, 1)} = -f^{(3, 1)}+f^{(1, 3)}+u^{(1, 5)}, \quad
u^{(5, 2)} = -f^{(3, 2)}+f^{(1, 4)}+u^{(1, 6)},\\
		& u^{(6, 0)} = -f^{(4, 0)}+f^{(2, 2)}-f^{(0, 4)}-u^{(0, 6)},\quad
u^{(6, 1)} = -f^{(4, 1)}+f^{(2, 3)}-f^{(0, 5)}-u^{(0, 7)},\\
		& u^{(7, 0)} = -f^{(5, 0)}+f^{(3, 2)}-f^{(1, 4)}-u^{(1, 6)}.
\end{split}
\ee
}}

Note that the cardinality of $\ind_{M+1}$ equals the sum of the cardinalities of $\ind_{M+1}^1$ and $\ind_{M-1}$.
The identities in \eqref{uderiv:relation} of \cref{lem:uderiv}
show that every
$u^{(m,n)}, (m,n)\in \ind_{M+1}$
can be written as a linear combination of
the quantities $u^{(m,n)}, (m,n)\in \ind_{M+1}^1$ and $f^{(m,n)}, (m,n)\in \ind_{M-1}$.
Conversely, by \eqref{umn:relation}, every $f^{(m,n)}, (m,n)\in \ind_{M-1}$ and every $u^{(m,n)}, (m,n)\in \ind_{M+1}^1$ can be trivially written as linear combinations of $u^{(m,n)}\in \ind_{M+1}$.
Because the source term $f$ is known, this can reduce the number of constraints on $u^{(m,n)}, (m,n)\in \ind_{M+1}$ for the function $u$ satisfying \eqref{Qeques1}.
Now using \eqref{uderiv:relation}, we can rewrite the approximation of $u(x+x_i^*,y+y_j^*)$ in \eqref{u:approx} as follows:
\begin{align*}
\sum_{(m,n)\in \ind_{M+1}} \frac{u^{(m,n)}}{m!n!} x^m y^{n}
&=\sum_{(m,n)\in \ind_{M+1}^1} \frac{u^{(m,n)}}{m!n!} x^m y^{n}+\sum_{(m,n)\in \ind_{M+1}^2} \frac{u^{(m,n)}}{m!n!} x^m y^{n}\\
 &=\sum_{(m,n)\in \ind_{M+1}^1}
u^{(m,n)} G_{m,n}(x,y) +\sum_{(m,n)\in \ind_{M-1}}
f^{(m,n)} H_{m,n}(x,y),
\end{align*}
where $G_{m,n}$ and $H_{m,n}$ are polynomials uniquely determined by the identities in \eqref{uderiv:relation}. Explicitly,
\be \label{Gmn}
G_{m,n}(x,y):=
\sum_{\ell=0}^{\lfloor \frac{n}{2}\rfloor}
(-1)^\ell \frac{x^{m+2\ell} y^{n-2\ell}}{(m+2\ell)!(n-2\ell)!},\qquad m\in \{0,1\}, n\in \NN
\ee
and
\be \label{Hmn}
H_{m,n}(x,y):=\sum_{\ell=1-\lfloor \frac{m}{2}\rfloor}^{1+\lfloor \frac{n}{2}\rfloor} \frac{(-1)^{\ell} x^{m+2\ell} y^{n-2\ell+2}}{(m+2\ell)!(n-2\ell+2)!},\qquad m,n\in \NN.
\ee
From \eqref{uderiv:relation} we observe that
$G_{m,n}$ is a homogeneous polynomial of total degree $m+n$ for all $(m,n)\in \ind_{M+1}^1$, while $H_{m,n}$ is a homogeneous polynomial of total degree $m+n+2$ for all $(m,n)\in \ind_{M-1}$.
For $M=6$, all the polynomials $G_{m,n}, (m,n)\in \ind_7^1$ and $H_{m,n},(m,n)\in \ind_5$ are explicitly given by
{\footnotesize{
	\be \label{Gxy}
	\begin{split}		
&G_{0, 0} = 1,\quad
G_{0, 1} = y,\quad
G_{0, 2} = \tfrac{1}{2}y^2-\tfrac{1}{2}x^2,\quad
G_{0, 3} = \tfrac{1}{6}y^3-\tfrac{1}{2}x^2y, \quad
G_{0, 4} = \tfrac{1}{24}y^4 +\tfrac{1}{24}x^4-\tfrac{1}{4}x^2y^2,\\
&G_{0, 5} = \tfrac{1}{120}y^5-\tfrac{1}{12}x^2y^3+\tfrac{1}{24}x^4y,\quad
G_{0, 6} = \tfrac{1}{720}y^6-\tfrac{1}{720}x^6 -\tfrac{1}{48}x^2y^4 +\tfrac{1}{48}x^4y^2,\quad\\
&G_{0, 7} = \tfrac{1}{5040}y^7+\tfrac{1}{144}x^4y^3-\tfrac{1}{720} x^6y
-\tfrac{1}{240}x^2y^5,\quad
G_{1, 0} = x, \quad
G_{1, 1} = xy,\quad
G_{1, 2} = \tfrac{1}{2}xy^2-\tfrac{1}{6}x^3,\\
&G_{1, 3} = \tfrac{1}{6}xy^3-\tfrac{1}{6}x^3y, \quad
G_{1, 4} = \tfrac{1}{24}xy^4+\tfrac{1}{120}x^5-\tfrac{1}{12}x^3y^2, \quad G_{1, 5} = \tfrac{1}{120}xy^5-\tfrac{1}{36}x^3y^3+\tfrac{1}{120} x^5y,\\
&G_{1, 6} = \tfrac{1}{720}xy^6-\tfrac{1}{5040}x^7+\tfrac{1}{240} x^5y^2-\tfrac{1}{144}x^3y^4,
\end{split}
\ee
}}
and
{\footnotesize{
	\be \label{Hxy}
	\begin{split}
&H_{0, 0} = -\tfrac{1}{2}x^2,\quad
H_{0, 1} = -\tfrac{1}{2}x^2y,\quad
H_{0, 2} = \tfrac{1}{24}x^4-\tfrac{1}{4}x^2y^2,\quad
H_{0, 3} = -\tfrac{1}{12}x^2y^3+\tfrac{1}{24}x^4y,\\
&H_{0, 4} = -\tfrac{1}{720}x^6-\tfrac{1}{48}x^2y^4+\tfrac{1}{48}
x^4y^2, \quad
H_{0, 5} = \tfrac{1}{144}x^4y^3-\tfrac{1}{720}x^6y
-\tfrac{1}{240}x^2y^5,\quad
H_{1, 0} = -\tfrac{1}{6}x^3,\quad  H_{1, 1} = -\tfrac{1}{6}x^3y,\\
&H_{1, 2} = \tfrac{1}{120}x^5-\frac{1}{12}x^3y^2, \quad
H_{1, 3} = \tfrac{1}{120}x^5y-\tfrac{1}{36}x^3y^3,\quad
H_{1, 4} = -\tfrac{1}{5040}x^7+\tfrac{1}{240}x^5y^2-\tfrac{1}{144}
x^3y^4,\quad
H_{2, 0}= -\tfrac{1}{24}x^4,\\
&H_{2, 1} = -\tfrac{1}{24}x^4y,\quad
H_{2, 2} = \tfrac{1}{720}x^6-\tfrac{1}{48}x^4y^2,\quad
H_{2, 3} = -\tfrac{1}{144}x^4y^3+\tfrac{1}{720}x^6y,\quad
H_{3, 0} = -\tfrac{1}{120}x^5,\\
&H_{3, 1} = -\tfrac{1}{120}x^5y,\quad
H_{3, 2} = \tfrac{1}{5040}x^7-\tfrac{1}{240}x^5y^2,\quad
H_{4, 0} = -\tfrac{1}{720}x^6,\quad  H_{4, 1} = -\tfrac{1}{720}x^6y,\quad  H_{5, 0} = -\tfrac{1}{5040}x^7.
\end{split}
\ee
}}
Hence, by \eqref{u:approx},
the solution $u$ to \eqref{Qeques1} near the base point $(x_i^*,y_j^*)$ can be approximated by
\be \label{u:approx:key}
u(x+x_i^*,y+y_j^*)
=
\sum_{(m,n)\in \ind_{M+1}^1}
u^{(m,n)} G_{m,n}(x,y) +\sum_{(m,n)\in \ind_{M-1}}
f^{(m,n)} H_{m,n}(x,y)+\bo(h^{M+2}),
\ee
for $x,y\in (-2h,2h)$. We shall use the above identity in \eqref{u:approx:key} for finding compact stencils achieving a desired accuracy order $M$.

\subsection{Regular points}

In this subsection, we discuss how to find a compact scheme centered at a regular point $(x_i,y_j)$, which has been well studied in the literature. The main purpose of this subsection is to outline the main ideas.
For simplicity, we just pick $(x_i,y_j)$ as the base point $(x_i^*,y_j^*)$, that is,
$(x_i^*,y_j^*)$ is defined in \eqref{base:pt} with $v_0=w_0=0$.
Recall that $M\in \N$ stands for the desired accuracy order.
For a compact stencil at a regular point $(x_i,y_j)$, we need to find stencil coefficients $\{C_{k,\ell}\}_{k,\ell=-1,0,1}$ and $\{C_{f,m,n}\}_{(m,n)\in \ind_{M-1}}$ satisfying
\be \label{stencil:regular1}
\sum_{k=-1}^1 \sum_{\ell=-1}^1
C_{k,\ell}(h) u(x_i+kh,y_j+\ell h)=
\sum_{(m,n)\in \ind_{M-1}} f^{(m,n)}C_{f,m,n}(h)
+\bo(h^{M+2}),\qquad h\to 0,
\ee
where $C_{k,\ell}$ and $C_{f,m,n}$ are to-be-determined polynomials of $h$ with degree less than $M+2$.
Plugging the approximation in \eqref{u:approx:key} into the above stencil in \eqref{stencil:regular1} to approximate $u(x_i+kh, y_j+\ell h)$,
the conditions in \eqref{stencil:regular1} become
\be \label{stencil:regular:2}
\sum_{(m,n)\in \ind_{M+1}^1}
u^{(m,n)} I_{m,n}(h)+
\sum_{(m,n)\in \ind_{M-1}} f^{(m,n)}
\left(J_{m,n}(h)-C_{f,m,n}(h)\right)
=\bo(h^{M+2}),
\ee
where
\be \label{IJ:regular}
I_{m,n}(h):=\sum_{k=-1}^1 \sum_{\ell=-1}^1 C_{k,\ell}(h) G_{m,n}(kh, \ell h),\quad
J_{m,n}(h):=\sum_{k=-1}^1 \sum_{\ell=-1}^1 C_{k,\ell}(h) H_{m,n}(kh, \ell h).
\ee
Since \eqref{stencil:regular:2} must hold for all unknowns $u^{(m,n)}, (m,n)\in \ind_{M+1}^1$ and $f^{(m,n)}, (m,n)\in \ind_{M-1}$, to find nontrivial stencil coefficients $C_{k,\ell}(h)$ for $k,\ell=-1,0,1$, solving \eqref{stencil:regular:2} is equivalent to solving
\be \label{stencil:regular:u}
I_{m,n}(h)=\bo(h^{M+2}) \quad h\to 0,\; \mbox{ for all }\; (m,n)\in \ind_{M+1}^1
\ee
and
\be \label{stencil:regular:f}
J_{m,n}(h)=C_{f,m,n}(h) +\bo(h^{M+2}),\qquad h\to 0, \; \mbox{ for all }\; (m,n)\in \ind_{M-1}.
\ee
Note that the solutions of
$\{C_{k,\ell}\}_{k,\ell=-1,0,1}$ and $\{C_{f,m,n}\}_{(m,n)\in \ind_{M-1}}$
to \eqref{stencil:regular:u} and \eqref{stencil:regular:f} are homogeneous in terms of its unknowns, that is, a solution multiplied with a given polynomial of $h$ to all coefficients is still a solution. Hence, we say that a solution for the coefficients in a compact stencil is nontrivial if $C_{k,\ell}(0)\ne 0$ for at least some $k,\ell=-1,0,1$.
Since $G_{m,n}$ is a homogeneous polynomial of degree $m+n$, we can write $G_{m,n}(kh,\ell h)=g_{m,n,k,\ell} h^{m+n}$ for some constants $g_{m,n,k,\ell}$.
Hence, \eqref{stencil:regular:u} becomes
\be \label{stencil:regular:u:2}
\sum_{k=-1}^1 \sum_{\ell=-1}^1 C_{k,\ell}(h) g_{m,n,k,\ell}=\bo(h^{M+2-m-n}),\quad h\to 0, \; \mbox{ for all }\; (m,n)\in \ind_{M+1}^1.
\ee
Because $M+2-m-n\ge 1$ for all $(m,n)\in \ind_{M+1}^1$, the identities in \eqref{stencil:regular:u:2} automatically imply
\be \label{stencil:regular:u:0}
\sum_{k=-1}^1 \sum_{\ell=-1}^1 C_{k,\ell}(0) g_{m,n,k,\ell}=0,\quad \mbox{ for all }\; (m,n)\in \ind_{M+1}^1.
\ee
By calculation, the maximum integer $M$ for the linear system in \eqref{stencil:regular:u:0} to have a nontrivial solution $\{C_{k,\ell}(0)\}_{k,\ell=-1,0,1}$ is $M=6$. More precisely, the rank of the matrix $(g_{m,n,k,\ell})_{(m,n)\in \ind_{M+1}^1, (k,\ell)\in \{-1,0,1\}^2}$ is nine for $M=7$ (hence \eqref{stencil:regular:u:0} has only the trivial solution for $M=7$) and its rank is $8$ for $M=6$. Therefore, for a compact stencil, the maximum accuracy order $M$ that we can achieve is $M=6$.
Moreover, up to a multiplicative constant, such a nontrivial solution $\{C_{k,\ell}(0)\}_{k,\ell=-1,0,1}$ to \eqref{stencil:regular:u:0} is uniquely given by
\be \label{C:M=6}
C_{0,0}=-20,\quad
C_{-1,0}=C_{1,0}=C_{0,-1}=C_{0,1}=4,\quad
C_{-1,-1}=C_{-1,1}=C_{1,-1}=C_{1,1}=1.
\ee
For a constant solution of $\{C_{k,\ell}(0)\}_{k,\ell=-1,0,1}$ satisfying \eqref{stencil:regular:u:0}, such a constant solution obviously satisfies also \eqref{stencil:regular:u:2} and therefore, it is a nontrivial solution to \eqref{stencil:regular:u}.

Since $H_{m,n}$ is a homogeneous polynomial of degree $m+n+2$, we can write $H_{m,n}(kh,\ell h)=h_{m,n,k,\ell} h^{m+n+2}$ for some constants $h_{m,n,k,\ell}$. Now plugging \eqref{C:M=6} into the definition of $J_{m,n}$,
we easily deduce from \eqref{stencil:regular:f} that all the coefficients $C_{f,m,n}$ satisfying \eqref{stencil:regular:f} are given by
\be \label{Cfmn:M=6}
C_{f,m,n}(h):=
\sum_{k=-1}^1 \sum_{\ell=-1}^1 C_{k,\ell}
H_{m,n}(h)=
\sum_{k=-1}^1 \sum_{\ell=-1}^1 C_{k,\ell} h_{m,n,k,\ell} h^{m+n+2}, \qquad (m,n)\in \ind_5.
\ee
Explicitly,
\[
C_{f,0,0}:=-6 h^2,\quad
C_{f,0,2}=C_{f,2,0}:=-\tfrac{1}{2} h^4,\quad
C_{f,0,4}=C_{f,4,0}:=-\tfrac{1}{60}h^6,
\quad C_{f,2,2}:=-\tfrac{1}{15}h^6
\]
and all other coefficients $C_{f,m,n}(h)=0$.
In summary, for a regular point $(x_i,y_j)$, we obtain the following theorem which is well known in the literature (e.g., see \cite{SDKS13,ZFH13,WZ09}).

\begin{theorem}\label{thm:regular}
Let a grid point $(x_i,y_j)$ be a regular point, i.e., either $d_{i,j}^+=\emptyset$ or $d_{i,j}^-=\emptyset$. Let $(u_{h})_{i,j}$ be the numerical approximated solution of the exact solution $u$ of the partial differential equation \eqref{Qeques1} at a regular point $(x_i, y_j)$. Then the scheme:
{\small{
\begin{equation}\label{stencil:regular}
	\begin{split}	 &(u_{h})_{i-1,j-1}+4(u_{h})_{i-1,j}+(u_{h})_{i-1,j+1}+4(u_{h})_{i,j-1}	 -20(u_{h})_{i,j}+4(u_{h})_{i,j+1}+(u_{h})_{i+1,j-1}+4(u_{h})_{i+1,j}\\
		&\qquad +(u_{h})_{i+1,j+1}
=-6h^2f^{(0, 0)}-\frac{1}{2}h^4(f^{(0, 2)}+f^{(2, 0)})-\frac{1}{60}h^6(f^{(0, 4)}+f^{(4, 0)})-\frac{1}{15}h^6f^{(2, 2)},
	\end{split}
\end{equation}
}}
achieves sixth order accuracy for $-\nabla ^2 u=f$, where $f^{(m,n)}:=\frac{\partial^{m+n} f}{ \partial^m x \partial^n y}(x_i,y_j)$.
Moreover,
\begin{enumerate}
\item[(i)] The compact finite difference scheme of order four can be  obtained from
\eqref{stencil:regular} by dropping the terms $\frac{1}{60}h^6(f^{(0, 4)}+f^{(4, 0)})$ and $\frac{1}{15}h^6f^{(2, 2)}$.
\item[(ii)] The maximum accuracy order $M$ for a compact finite difference scheme is $M=6$.
\end{enumerate}
\end{theorem}

\subsection{Irregular points{\label{IrrpointSection}}}

Let $(x_i,y_j)$ be an irregular point, that is, both $d_{i,j}^P\ne \emptyset$ and $d_{i,j}^N\ne \emptyset$.
In this subsection, we shall find a compact stencil at an irregular point $(x_i,y_j)$ for a given accuracy order $M$. The idea is essentially the same, although the technicalities are much more complicated.

Let $(x_i,y_j)$ be an irregular  point and we shall take a base point $(x^*_i,y^*_j)\in \Gamma \cap (x_i-h,x_i+h)\times (y_j-h,y_j+h)$ on the interface $\Gamma$ and inside $(x_i-h,x_i+h)\times (y_j-h,y_j+h)$.
That is, as in \eqref{base:pt},
\begin{equation} \label{base:pt:gamma}
x_i^*=x_i-v_0  \quad \mbox{and}\quad y_j^*=y_j-w_0  \quad \mbox{with}\quad
-h<v_0, w_0<h \quad \mbox{and}\quad (x_i^*,y_j^*)\in \Gamma.
\end{equation}
Let $u_{\pm}$ and $f_{\pm}$ represent the solution $u$ and source term $f$ in $\Op$ or $\Om$, respectively.
As in \eqref{ufmn},  we define
\[
u_{\pm}^{(m,n)}:=\frac{\partial^{m+n} u_{\pm}}{ \partial^m x \partial^n y}(x^*_i,y^*_j),\qquad f_{\pm}^{(m,n)}:=\frac{\partial^{m+n} f_{\pm}}{ \partial^m x \partial^n y}(x^*_i,y^*_j)
\]
and
\[
g^{(m,n)}:=\frac{\partial^{m+n} g}{ \partial^m x \partial^n y}(x^*_i,y^*_j),\qquad
g_1^{(m,n)}:=\frac{\partial^{m+n} g_1}{ \partial^m x \partial^n y}(x^*_i,y^*_j).
\]
Since the base point $(x_i^*,y_j^*)$ is now on the interface $\Gamma$,  the equation $-\nabla^2 u=f$ is no longer valid at the base point $(x_i^*,y_j^*)$. However, the curve $\Gamma$ is smooth and we assumed that the solution $u$ and $f$ are piecewise smooth. More precisely, $u_+$ and $f_+$ on $\Op$ can be extended into smooth functions in a neighborhood of $(x_i^*,y_j^*)$, while
$u_-$ and $f_-$ on $\Om$ can be extended into smooth functions in a neighborhood of $(x_i^*,y_j^*)$.
Therefore, \cref{lem:uderiv} still holds
for $u_\pm$ and $f_\pm$.
In other words, the identities in \eqref{uderiv:relation} hold by replacing $u$ and $f$ by $u_\pm$ and $f_\pm$, respectively. Consequently, the key identity in \eqref{u:approx:key} still holds by replacing $u$ and $f$ with $u_\pm$ and $f_\pm$, respectively.
Explicitly,
\be \label{u:approx:ir:key}
u_\pm (x+x_i^*,y+y_j^*)
=\sum_{(m,n)\in \ind_{M+1}^1}
u_\pm^{(m,n)} G_{m,n}(x,y) +\sum_{(m,n)\in \ind_{M-1}}
f_\pm ^{(m,n)} H_{m,n}(x,y)+\bo(h^{M+2}),
\ee
for $x,y\in (-2h,2h)$, where the index sets $\ind_{M+1}^1$ and $\ind_{M-1}$ are defined in \eqref{Sk1} and \eqref{Sk}, respectively,  while the polynomials $G_{m,n}(x,y)$ and $H_{m,n}(x,y)$ are defined in \eqref{Gmn} and \eqref{Hmn}, respectively.

For a compact stencil at an irregular point $(x_i,y_j)$ to achieve a given accuracy order $M$, we need to find nontrivial stencil coefficients satisfying
{\small{\be \label{stencil:irregular}
\begin{split}
\sum_{k=-1}^1 \sum_{\ell=-1}^1
&C_{k,\ell}(h) u(x_i+kh,y_j+\ell h)=
\sum_{(m,n)\in \ind_{M-1}} C_{f_{+},m,n}(h) f_+^{(m,n)}
+\sum_{(m,n)\in \ind_{M-1}} C_{f_{-},m,n}(h)f_-^{(m,n)}\\
&+\sum_{(m,n)\in \ind_{M+1}} C_{g_1,m,n}(h) g_1^{(m,n)}+
+\sum_{(m,n)\in \ind_{M}} C_{g,m,n}(h) g^{(m,n)}
+\bo(h^{M+2}),\qquad h\to 0,
\end{split}
\ee
}}
where $C_{k,\ell}, C_{f_{\pm},m,n}$,
$C_{g_1,m,n}$ and $C_{g,m,n}$
are to-be-determined polynomials of $h$ having degree less than $M+2$.
Because some indices $(k, \ell)$ may come from $d_{i,j}^+$ while others from $d_{i,j}^-$, we need to link information on $\Op$ and $\Om$ at the base point $(x_i^*,y_j^*)\in \Gamma$.
To do so, instead of using the level set function $\psi$ to describe the interface curve $\Gamma$, we shall now assume that we have a parametric equation for $\Gamma$ near the base point $(x_i^*,y_j^*)$.
We can easily obtain such a parametric equation by locally solving $\psi(x,y)=0$ near the base point $(x_i^*, y_j^*)$ for either $x$ or $y$. That is,
it suffices to consider one of the following two simple parametric representations of $\Gamma$:
\be\label{parametric:2 simple}
x=t+x_i^*, \quad y=r(t)+y_j^* \qquad \mbox{or}\quad
x=r(t)+x_i^*,\quad y=t+y_j^*, \quad \mbox{for}\;\; t\in (-\epsilon,\epsilon) \quad \mbox{with}\quad \epsilon>0,
\ee
for a smooth function $r$, since $\Gamma$ is assumed to be smooth.
Note that the parameter corresponding to the base point $(x_i^*, y_j^*)$ is $t=0$ with $r(0)=0$.
It is important to notice that we do not need to actually solve $\psi(x,y)=0$ to get the function $r$, because we only need the derivatives of $r(t)$ at $t=0$, which can be easily obtained from $\psi(x,y)=0$ through the Implicit Function Theorem.
To cover the above two cases of parametric equations in \eqref{parametric:2 simple} for $\Gamma$ together, we discuss the following general parametric equation for $\Gamma$:
\be \label{parametric}
x=r(t)+x_i^*,\quad y=s(t)+y_j^*,\quad
(r'(t))^2+(s'(t))^2>0
\quad \mbox{for}\;\; t\in (-\epsilon,\epsilon) \quad \mbox{with}\quad \epsilon>0.
\ee
Note that the parameter $t$ for the base point $(x_i^*,y_j^*)$ is $t=0$ and it is also important to notice that $r(0)=s(0)=0$.

Using the interface conditions in \eqref{Qeques1},
we now link the two sets
$\{u_-^{(m,n)}:(m,n)\in \ind_{M+1}^1\}$ and $\{u_+^{(m,n)}:(m,n)\in \ind_{M+1}^1\}$ through the following result, whose proof is given in \cref{sec:proof}.

\begin{theorem}\label{thm:interface}
Let $u$ be the solution to the Poisson interface problem in \eqref{Qeques1} and the base point $(x_i^*,y_j^*)\in \Gamma$, which is parameterized near $(x_i^*,y_j^*)$ by \eqref{parametric}.
Then
\be \label{tranmiss:cond}
\begin{split}
u_-^{(m',n')}&=u_+^{(m',n')}+\sum_{(m,n)\in \ind_{M-1}} \left(T^+_{m',n',m,n} f_+^{(m,n)}
+ T^-_{m',n',m,n} f_{-}^{(m,n)}\right)\\
&+\sum_{(m,n)\in \ind_{M+1}} T^{g_1}_{m',n',m,n} g_{1}^{(m,n)}
+\sum_{(m,n)\in \ind_{M}} T^{g}_{m',n',m,n} g^{(m,n)},\qquad \forall\; (m',n')\in \ind_{M+1}^1,
\end{split}
\ee
where all the transmission coefficients $T^\pm, T^{g_1}, T^g$ are uniquely determined by
$r^{(k)}(0)$ and $s^{(k)}(0)$ for $k=0,\ldots,M+1$ and can be easily obtained by recursively calculating $U^{(m',n')}:=u_+^{(m',n')}-u_-^{(m',n')}, (m',n')\in \ind_{M+1}^1$ through the recursive formulas given in \eqref{interface:u:0} and \eqref{transmissioncoef}.
\end{theorem}

Now we discuss how to find a compact stencil at an irregular point $(x_i,y_j)$ with the given accuracy order $M$. Since the set $\{-1,0,1\}^2$ is the disjoint union of $d_{i,j}^+$ and $d_{i,j}^-$, we can write
\[
\begin{split}
&\sum_{k=-1}^1 \sum_{\ell=-1}^1
C_{k,\ell}(h) u(x_i+kh,y_j+\ell h)\\
&=\sum_{(k,\ell)\in d_{i,j}^+}
C_{k,\ell}(h) u(x_i^*+v_0+kh,y_j^*+w_0+\ell h)
+\sum_{(k,\ell)\in d_{i,j}^-}
C_{k,\ell}(h) u(x_i^*+v_0+kh,y_j^*+w_0+\ell h).
\end{split}
\]
Using \eqref{u:approx:ir:key}, we have
\be \label{IJ+}
\sum_{(k,\ell)\in d_{i,j}^{\pm}}
C_{k,\ell}(h) u(x_i^*+v_0+kh,y_j^*+w_0+\ell h)
=\sum_{(m,n)\in \ind_{M+1}^1}
u_\pm^{(m,n)} I^\pm_{m,n}(h)+
\sum_{(m,n)\in \ind_{M-1}}
 f_\pm ^{(m,n)} J^{\pm,0}_{m,n}(h),
\ee
where
\[
I^\pm_{m,n}(h):=\sum_{(k,\ell)\in d_{i,j}^\pm}
C_{k,\ell}(h) G_{m,n}(v_0+kh,w_0+\ell h),
\quad
J^{\pm,0}_{m,n}(h):=\sum_{(k,\ell)\in d_{i,j}^\pm} C_{k,\ell}(h) H_{m,n}(v_0 +kh,w_0+\ell h).
\]
Now using the transmission coefficients in \cref{thm:interface}, we have
\begin{align*}
\sum_{(m',n')\in \ind_{M+1}^1}
u_-^{(m',n')} I^-_{m',n'}(h)
=&\sum_{(m',n')\in \ind_{M+1}^1} u_+^{(m',n')}
I^-_{m',n'}(h)+
\sum_{(m,n)\in \ind_{M-1}}
\left(f_+^{(m,n)} J^{+,T}_{m,n}(h)+
f_-^{(m,n)} J^{-,T}_{m,n}(h)\right)\\
&+\sum_{(m,n)\in \ind_{M+1}} g_1^{(m,n)}J^{g_1}_{m,n}(h)+
\sum_{(m,n)\in \ind_M} g^{(m,n)}J^{g}_{m,n}(h),
\end{align*}
where
\be
\label{JTno}
\begin{split}
&J^{\pm,T}_{m,n}(h):=
\sum_{(m',n')\in \ind_{M+1}^1} I^-_{m',n'}(h) T^\pm_{m',n',m,n},\\
&J^{g_1}_{m,n}(h):=
\sum_{(m',n')\in \ind_{M+1}^1} I^-_{m',n'}(h) T^{g_1}_{m',n',m,n},\quad
J^{g}_{m,n}(h):=
\sum_{(m',n')\in \ind_{M+1}^1} I^-_{m',n'}(h) T^g_{m',n',m,n}.
\end{split}
\ee
Putting everything together we have
\be \label{u:sum}
\begin{split}
\sum_{k=-1}^1 \sum_{\ell=-1}^1
&C_{k,\ell}(h) u(x_i+kh,y_j+\ell h)
=\sum_{(m,n)\in \ind_{M+1}^1}
u_+^{(m,n)} I_{m,n}(h)+
\sum_{(m,n)\in \ind_{M-1} } f_+^{(m,n)}J^+_{m,n}(h)\\
&+\sum_{(m,n)\in \ind_{M-1}} f_-^{(m,n)}J^-_{m,n}(h)
+\sum_{(m,n)\in \ind_{M+1}} g_1^{(m,n)}J^{g_1}_{m,n}(h)
+\sum_{(m,n)\in \ind_{M}} g^{(m,n)}J^{g}_{m,n}(h),
\end{split}
\ee
where
\be \label{IJ:ir}
I_{m,n}(h):=I^+_{m,n}(h)+I^-_{(m,n)}(h),
\quad J^{\pm}_{m,n}(h):=
J_{m,n}^{\pm,0}(h)+J^{\pm,T}_{m,n}(h).
\ee
Note that all the unknowns are $u_+^{(m,n)}$ for $(m,n)\in \ind_{M+1}^1$, $f_\pm^{(m,n)}$ for $(m,n)\in \ind_{M-1}$, $g_1^{(m,n)}$ for $(m,n)\in \ind_{M+1}$ and $g^{(m,n)}$ for $(m,n)\in \ind_M$.
Therefore, to find a compact stencil at an irregular point $(x_i,y_j)$ with a desired accuracy order $M$, the conditions in \eqref{stencil:irregular} are equivalent to
\begin{align}
&I_{m,n}(h)=\bo(h^{M+2}),  &&h\to 0, \; \mbox{ for all }\; (m,n)\in \ind_{M+1}^1, \label{stencil:regular:u:ir}\\
&J^\pm_{m,n}(h)=C_{f_\pm,m,n}(h)+\bo(h^{M+2}),
&&h\to 0,\; \mbox{ for all }\; (m,n)\in \ind_{M-1},\label{stencil:regular:f:ir}\\
&J^{g_1}_{m,n}(h)=C_{g_1,m,n}(h)+\bo(h^{M+2}),
&&h\to 0,\; \mbox{ for all }\; (m,n)\in \ind_{M+1},\label{stencil:regular:g1:ir}\\
&J^{g}_{m,n}(h)=C_{g,m,n}(h)+\bo(h^{M+2}),
&&h\to 0,\; \mbox{ for all }\; (m,n)\in \ind_{M}.\label{stencil:regular:g:ir}
\end{align}
Due to the relations in \eqref{tranmiss:cond} of \cref{thm:interface},
we observe that
the solution $\{C_{k,\ell}\}_{k,\ell=-1,0,1}$ in \eqref{C:M=6} is also a solution to
\eqref{stencil:regular:u:ir} with $M=6$.
Now similar to \eqref{Cfmn:M=6}, plugging the values of $\{C_{k,\ell}\}_{k,\ell=-1,0,1}$ in \eqref{C:M=6} into \eqref{stencil:regular:f:ir},
\eqref{stencil:regular:g1:ir} and \eqref{stencil:regular:g:ir} with $M=6$, we obtain all the other stencil coefficients given by
\begin{align}
&C_{f_\pm,m,n}(h):=J_{m,n}^\pm(h), \quad (m,n)\in \ind_{5},\label{coeff:f:ir}\\
&C_{g_1,m,n}(h):=J^{g_1}_{m,n}(h),\quad (m,n)\in \ind_{7}\quad
\mbox{and}\quad
C_{g,m,n}(h)=J^g_{m,n}(h),\quad (m,n)\in \ind_6.\label{coeff:g:ir}
\end{align}

In summary, we obtain the following theorem for compact stencils at irregular points.

\begin{theorem}\label{fluxtm2}
Let $(u_{h})_{i,j}$ be the numerical solution of \eqref{Qeques1} at an irregular point $(x_i, y_j)$. Pick a base point
$(x_i^*,y_j^*)$ as in \eqref{base:pt:gamma}.
Then
the following compact scheme centered at the irregular point $(x_i,y_j)$:
\begin{equation}\label{stencil:irregular1}
\begin{split}
    &(u_{h})_{i-1,j-1}+4(u_{h})_{i-1,j}+(u_{h})_{i-1,j+1}+4(u_{h})_{i,j-1}
    -20(u_{h})_{i,j}+4(u_{h})_{i,j+1}+(u_{h})_{i+1,j-1}\\
	 &\qquad +4(u_{h})_{i+1,j}+(u_{h})_{i+1,j+1}=\sum_{(m,n)\in \ind_{5} } f_+^{(m,n)}J^+_{m,n}(h)+\sum_{(m,n)\in \ind_{5}} f_-^{(m,n)}J^-_{m,n}(h)\\
	 &\hspace{4cm}+\sum_{(m,n)\in \ind_{7}} g_1^{(m,n)}J^{g_1}_{m,n}(h)
	 +\sum_{(m,n)\in \ind_{6}} g^{(m,n)}J^{g}_{m,n}(h),
\end{split}
\end{equation}
achieves sixth  order of accuracy,
where the quantities $J^\pm_{m,n}, (m,n)\in \ind_5$, $J^{g_1}_{m,n}, (m,n)\in \ind_7$ and $J^g_{m,n}, (m,n)\in \ind_6$ are given in  \eqref{IJ:ir} and \eqref{JTno}, respectively.
Moreover, the stencils for the accuracy order $M=3,4,5$ can be easily obtained from the stencil in \eqref{stencil:irregular1} by dropping $G_{m,n}$ with $m+n>M+1$ and $H_{m,n}$ with $m+n>M-1$.
\end{theorem}

\section{Numerical experiments}\label{sec:Numeri}

Let $\Omega=(l_1,l_2)\times(l_3,l_4)$ with
$l_4-l_3=N_0(l_2-l_1)$ for some positive integer $N_0$. For a given $J\in \NN$, we define
$h:=(l_2-l_1)/N_1$ with $N_1:=2^J$ and let
$x_i=l_1+ih$ and
$y_j=l_3+jh$ for $i=1,2,\dots,N_1-1$ and $j=1,2,\dots,N_2-1$ with $N_2:=N_0N_1$.
Let
$u(x,y)$ be the exact solution of \eqref{Qeques1} and $(u_{h})_{i,j}$ be the numerical solution at $(x_i, y_j)$ using the mesh size $h$.
We shall measure the consistency of our proposed scheme in the $l_2$ norm by the relative error
$\frac{\|u_{h}-u\|_{2}}{\|u\|_{2}}$ if the exact solution $u$ is available,
and
by the relative error $\frac{\|u_{h}-u_{h/2}\|_{2}}{\|u_{h/2}\|_{2}}$ if the exact solution is not known (and hence we use the next level numerical solution $u_{h/2}$ as a reference solution), where
\begin{align*}
&\|u_{h}-u\|_{2}^2:= \frac{1}{N}\sum_{i=1}^{N_1-1}\sum_{j=1}^{N_2-1} \left((u_h)_{i,j}-u(x_i,y_j)\right)^2, &&\|u\|_{2}^2:=\frac{1}{N} \sum_{i=1}^{N_1-1}\sum_{j=1}^{N_2-1} \left(u(x_i,y_j)\right)^2,\\
&\|u_{h}-u_{h/2}\|_{2}^2:= \frac{1}{N}\sum_{i=1}^{N_1-1}\sum_{j=1}^{N_2-1} \left((u_{h})_{i,j}-(u_{h/2})_{2i,2j}\right)^2, &&\|u_{h/2}\|_{2}^2:=\frac{1}{N}\sum_{i=1}^{N_1-1}\sum_{j=1}^{N_2-1} \left((u_{h/2})_{2i,2j}\right)^2,
\end{align*}
where $N:=(N_1-1)(N_2-1)$.
The errors can be also measured in the $l_\infty$ norm as follows:
\[
\|u_h-u\|_\infty
:=\max_{1\le i<N_1, 1\le j<N_2} \left|(u_h)_{i,j}-u(x_i,y_j)\right|,
\qquad
\|u_{h}-u_{h/2}\|_\infty:=\max_{1\le i<N_1,1\le j<N_2} \left|(u_{h})_{i,j}-(u_{h/2})_{2i,2j}\right|.
\]
In addition, $\kappa$ denotes the condition number of the coefficient matrix. According to \cref{thm:regular,fluxtm2}, \eqref{Qeques1} has the same coefficient matrix. So we only provide the values of $\kappa$ in \cref{table:QSp1}.

\subsection{Numerical examples with $u$ known and $\Gamma \cap \partial \Omega=\emptyset$}

In this subsection, we provide a few numerical experiments such that the exact solution $u$ of \eqref{Qeques1} is known and the interface curve $\Gamma$ does not touch the boundary of $\Omega$.

\begin{example}\label{ex1}
\normalfont
Let $\Omega=(-\pi,\pi)^2$ and
the interface curve be given by
$\Gamma:=\{(x,y)\in \Omega \; :\; \psi(x,y)=0\}$ with
$\psi (x,y)=x^2+y^2-2$. Note that $\Gamma \cap \partial \Omega=\emptyset$ and
the exact solution $u$ of \eqref{Qeques1} is given by
\[
u_{+}=u\chi_{\Op}=\sin(4x)(2-(x^2+y^2))^2,
\qquad u_{-}=u\chi_{\Om}=\cos(4y)(2-(x^2+y^2))^2+100.
\]
All the functions $f,g,g_0,g_1$ in \eqref{Qeques1} can be obtained by plugging the above exact solution into \eqref{Qeques1}. In particular,
$g_1=-100$ and $g=0$.
The numerical results are presented in \cref{table:QSp1} and \cref{fig:figure1}.	 
\end{example}

\begin{table}[htbp]
	\caption{Performance in \cref{ex1}  of the proposed sixth order compact finite difference scheme in \cref{thm:regular,fluxtm2} on uniform Cartesian meshes with $h=2^{-J}\times2\pi$. $\kappa$ is the condition number of the coefficient matrix.}
	\centering
	\setlength{\tabcolsep}{2mm}{
		 \begin{tabular}{c|c|c|c|c|c|c|c|c|c}
			\hline
			$J$
			& $\frac{\|u_{h}-u\|_{2}}{\|u\|_{2}}$ 
			
			&order & $\|u_{h}-u\|_{\infty}$
			
			&order &  $\frac{\|u_{h}-u_{h/2}\|_{2}}{\|u_{h/2}\|_{2}}$
			&order &  $\|u_{h}-u_{h/2}\|_{\infty}$
			
			&order &  $\kappa$ \\
			\hline
3   &3.65E+00   &0   &3.55E+02   &0   &3.08E+00   &0   &3.40E+02   &0   &3.14E+01\\
4   &1.25E-01   &4.868   &1.90E+01   &4.224   &1.24E-01   &4.631   &1.89E+01   &4.165   &1.26E+02\\
5   &6.60E-04   &7.566   &1.03E-01   &7.529   &6.56E-04   &7.565   &1.03E-01   &7.528   &5.03E+02\\
6   &3.38E-06   &7.610   &5.87E-04   &7.456   &3.35E-06   &7.613   &5.83E-04   &7.459   &2.01E+03\\
7   &2.55E-08   &7.048   &4.27E-06   &7.103   &2.53E-08   &7.050   &4.24E-06   &7.104   &8.05E+03\\
8   &2.40E-10   &6.733   &3.50E-08   &6.928   &3.58E-10   &6.142   &8.04E-08   &5.720   &3.22E+04\\
			\hline
	\end{tabular}}
	\label{table:QSp1}
\end{table}

\begin{example}\label{ex2}
\normalfont
Let $\Omega=(-\pi,\pi)^2$ and the interface curve be given by $\Gamma:=\{ (x,y)\in \Omega \; : \; \psi(x,y)=0\}$ with $\psi (x,y)=\frac{y^2}{2}+\frac{x^2}{1+x^2}-\frac{1}{2}$. Note that
$\Gamma \cap \partial \Omega=\emptyset$ and
the exact solution $u$ of \eqref{Qeques1} is given by
\[
u_{+}=u\chi_{\Op}=\sin(2x),
\qquad u_{-}=u\chi_{\Om}=\sin(2x)+3.
\]
All the functions $f,g,g_0,g_1$ in  \eqref{Qeques1} can be obtained by plugging the above exact solution into \eqref{Qeques1}. In particular,
$g_1=-3$ and $g=0$.
The numerical results are provided in \cref{table:QSp2} and \cref{fig:figure1}.
\end{example}

\begin{table}[htbp]
	\caption{Performance in \cref{ex2}  of the proposed sixth order compact finite difference scheme in \cref{thm:regular,fluxtm2}  on uniform Cartesian meshes with $h=2^{-J}\times2\pi$.}
	\centering
	\setlength{\tabcolsep}{3mm}{
		 \begin{tabular}{c|c|c|c|c|c|c|c|c}
			\hline
			$J$
& $\frac{\|u_{h}-u\|_{2}}{\|u\|_{2}}$ 

&order & $\|u_{h}-u\|_{\infty}$

&order &  $\frac{\|u_{h}-u_{h/2}\|_{2}}{\|u_{h/2}\|_{2}}$
&order &  $\|u_{h}-u_{h/2}\|_{\infty}$

&order \\
			\hline
3   &2.51E-02   &0   &7.00E-02   &0   &2.49E-02   &0   &6.96E-02   &0\\
4   &1.92E-04   &7.033   &6.63E-04   &6.721   &1.90E-04   &7.036   &6.59E-04   &6.722\\
5   &1.88E-06   &6.667   &7.48E-06   &6.471   &1.87E-06   &6.667   &7.42E-06   &6.474\\
6   &1.67E-08   &6.817   &6.65E-08   &6.813   &1.66E-08   &6.814   &6.60E-08   &6.811\\
7   &1.21E-10   &7.112   &4.94E-10   &7.074   &1.20E-10   &7.113   &4.90E-10   &7.074\\
8   &1.02E-12   &6.891   &4.36E-12   &6.822   &1.10E-12   &6.769   &5.06E-12   &6.598\\
			\hline
	\end{tabular}}
	\label{table:QSp2}
\end{table}

\begin{figure}[htbp]
	\centering
	 \begin{subfigure}[b]{0.3\textwidth}
		 \includegraphics[width=5.7cm,height=5.7cm]{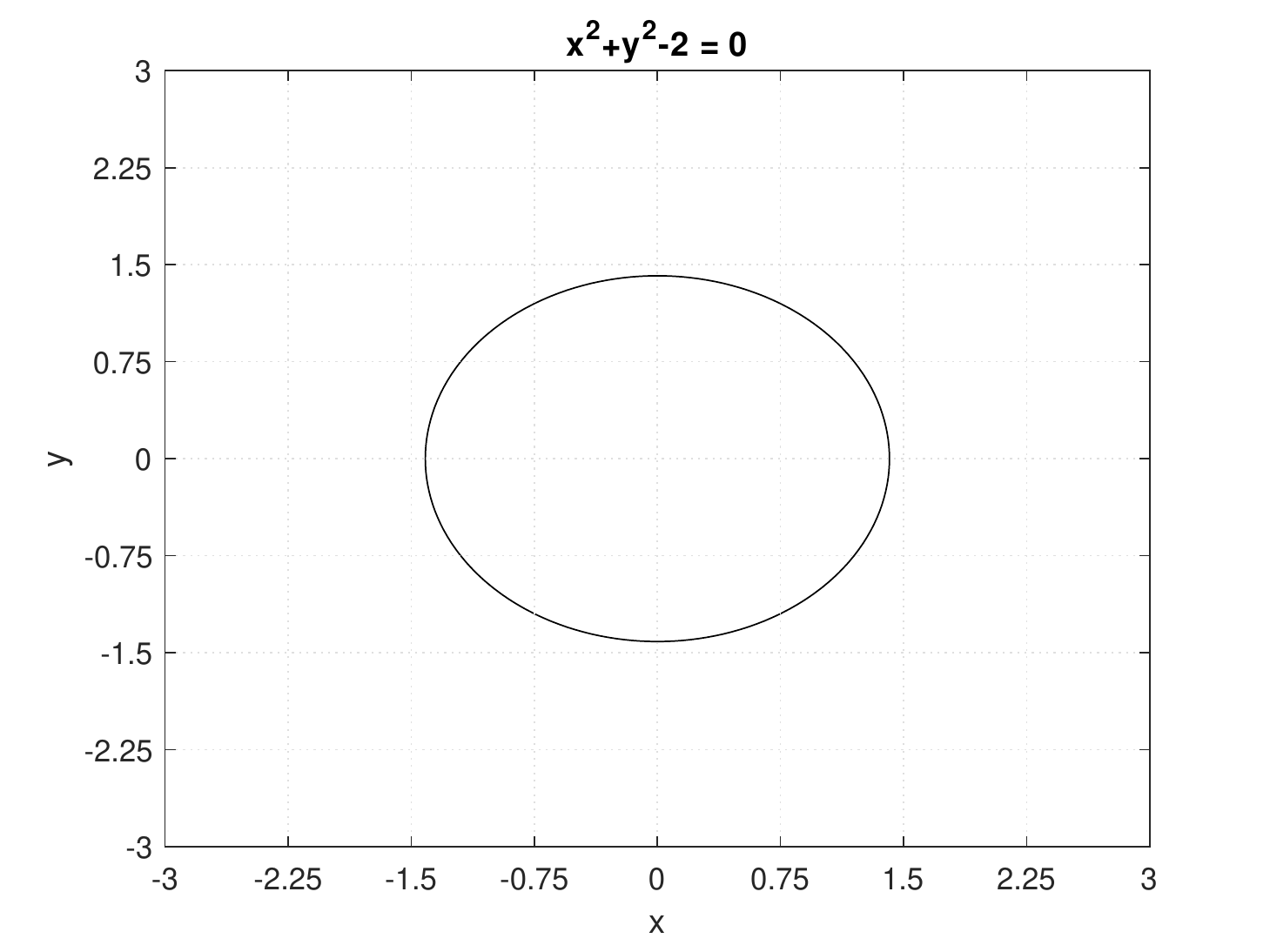}
	\end{subfigure}
	 \begin{subfigure}[b]{0.3\textwidth}
		 \includegraphics[width=5.7cm,height=5.7cm]{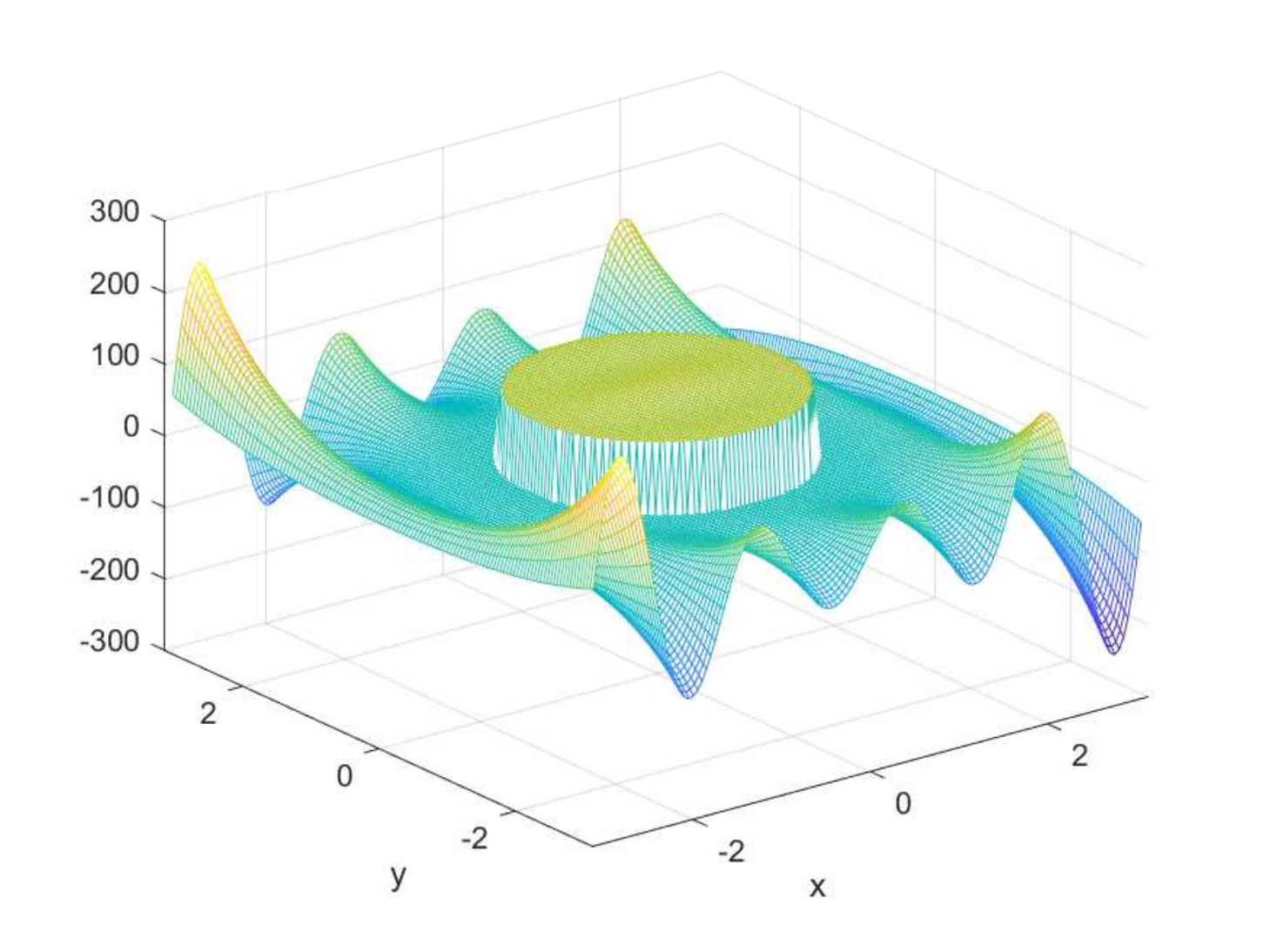}
	\end{subfigure}
	 \begin{subfigure}[b]{0.3\textwidth}
		 \includegraphics[width=5.7cm,height=5.7cm]{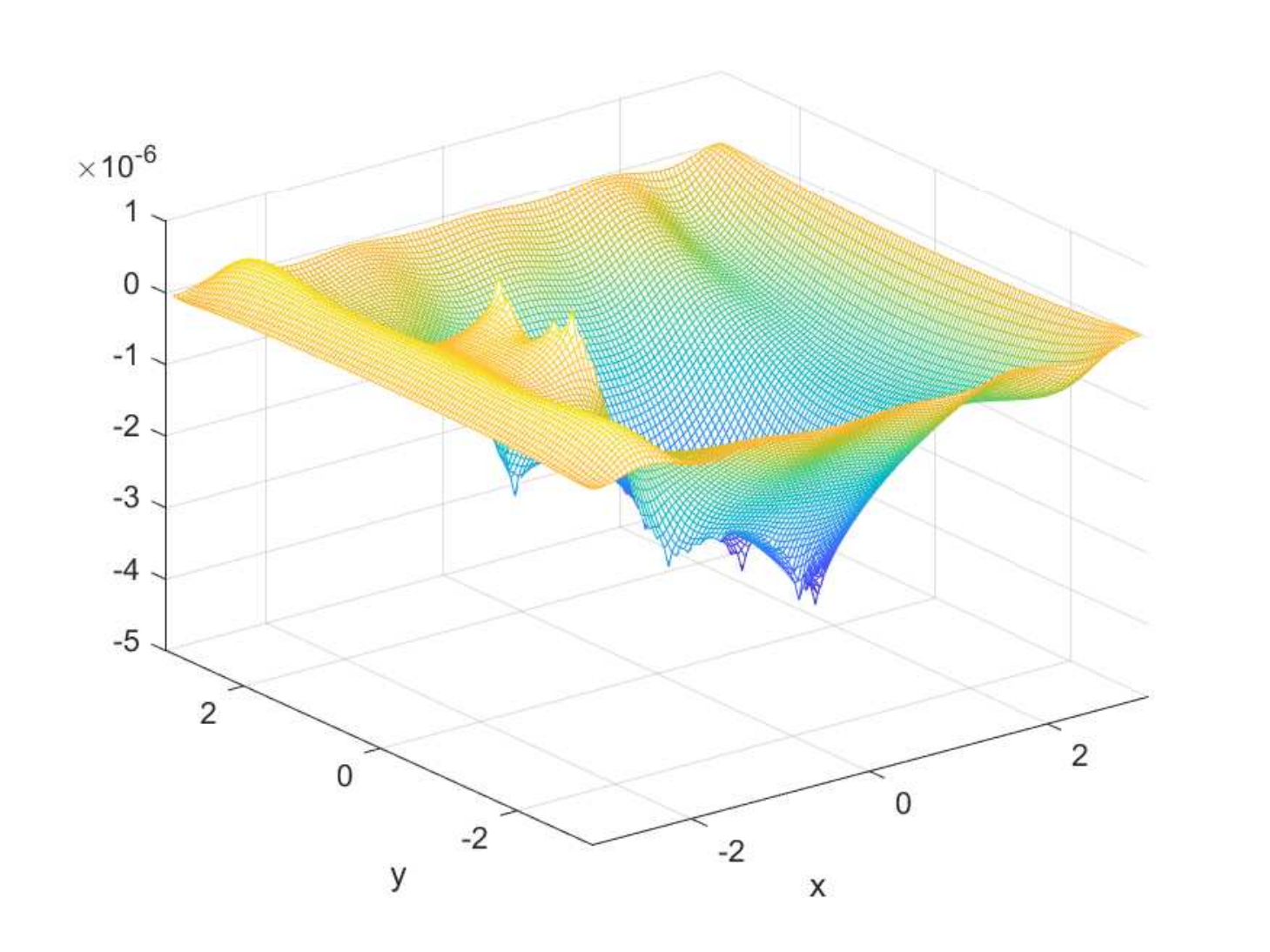}
	\end{subfigure}
	 \begin{subfigure}[b]{0.3\textwidth}
		 \includegraphics[width=5.7cm,height=5.7cm]{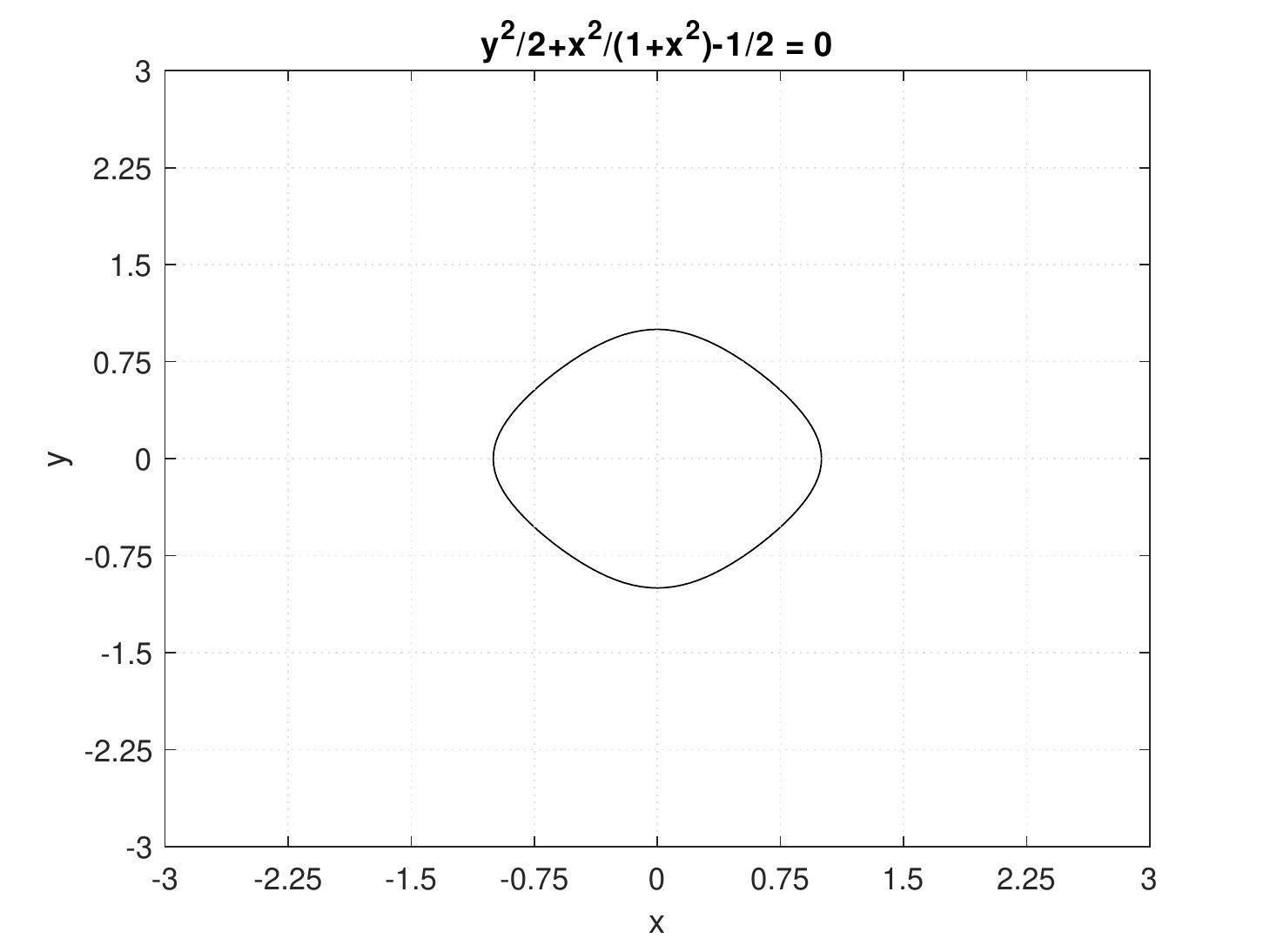}
	\end{subfigure}
	 \begin{subfigure}[b]{0.3\textwidth}
		 \includegraphics[width=5.7cm,height=5.7cm]{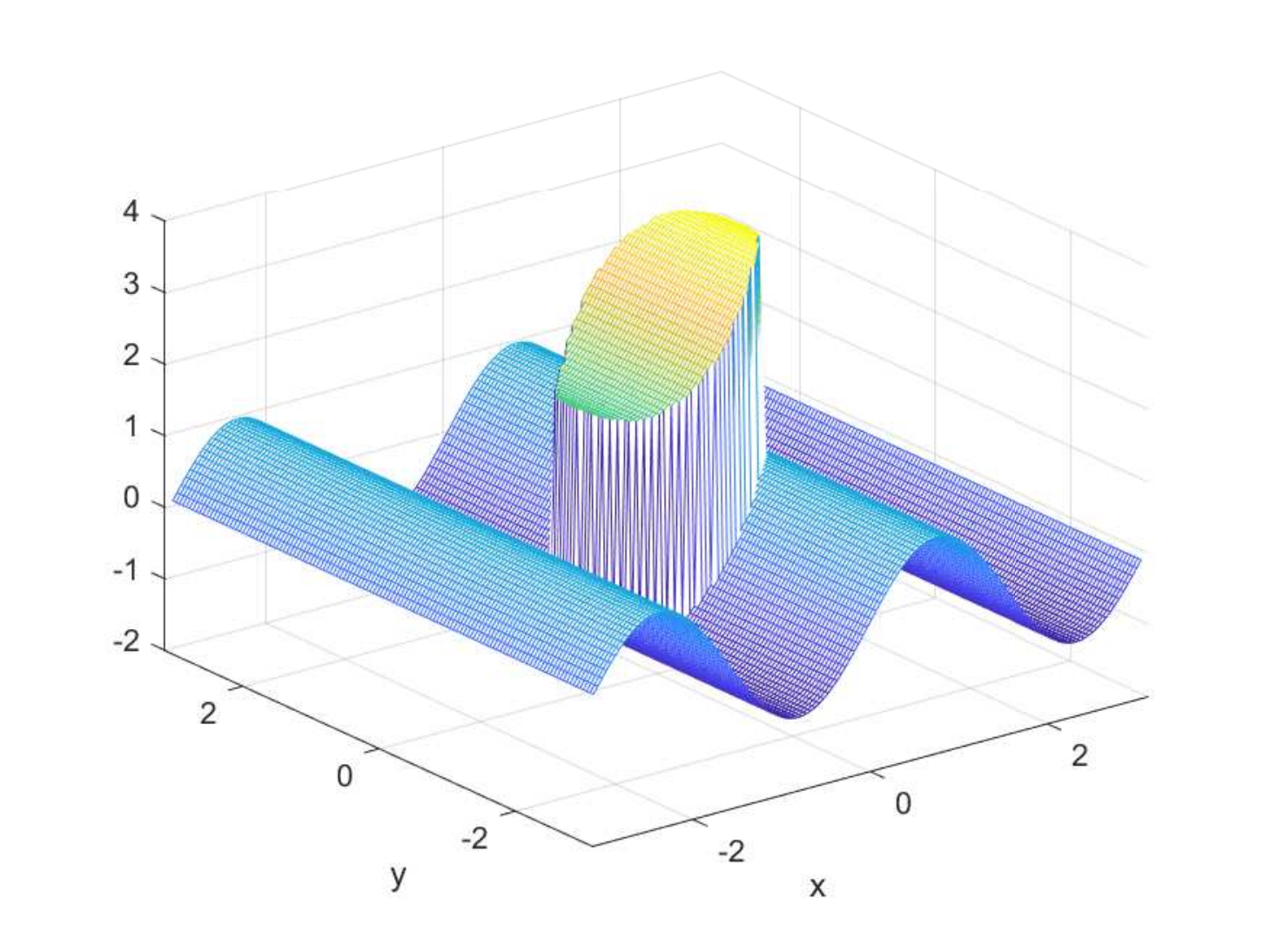}
	\end{subfigure}
	 \begin{subfigure}[b]{0.3\textwidth}
		 \includegraphics[width=5.7cm,height=5.7cm]{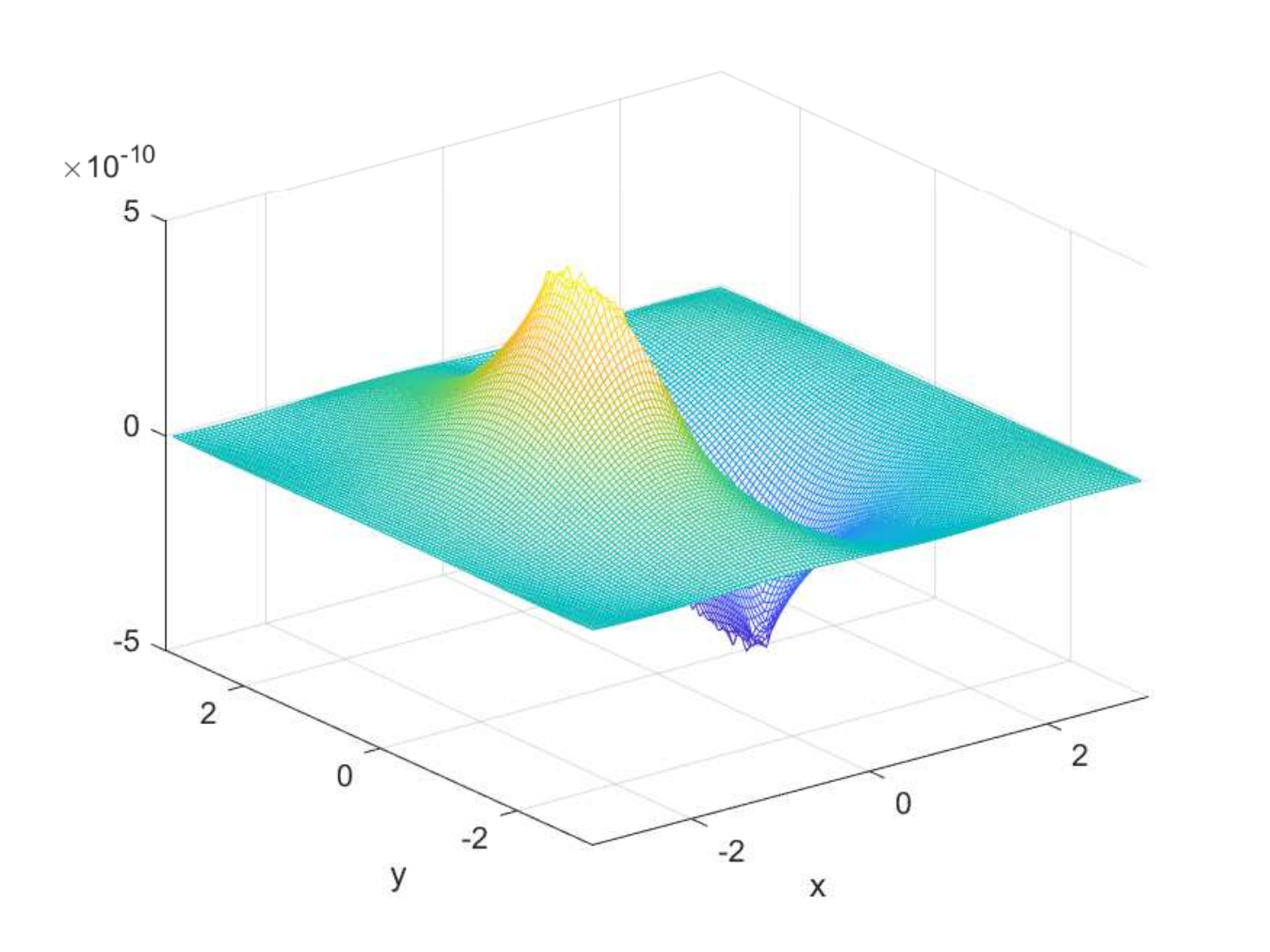}
	\end{subfigure}
	\caption
	{Top row for \cref{ex1}: the interface curve $\Gamma$ (left), the numerical solution $u_h$ (middle) and the error $u-u_h$ (right) with $h=2^{-7}\times 2\pi$. Bottom row for \cref{ex2}: the interface curve $\Gamma$ (left), the numerical solution $u_h$ (middle) and the error $u-u_h$ (right) with $h=2^{-7}\times 2\pi$.}
	\label{fig:figure1}
\end{figure}

\subsection{Numerical examples with $u$ known and $\Gamma \cap \partial \Omega\ne\emptyset$}

In this subsection, we provide a few numerical experiments such that the exact solution $u$ of \eqref{Qeques1} is known and the interface curve $\Gamma$ touches the boundary of $\Omega$.

\begin{example}\label{ex3}
\normalfont
Let $\Omega=(-\frac{3}{2}\pi,\frac{3}{2}\pi)^2$ and the interface curve be given by $\Gamma:=\{ (x,y)\in \Omega \; : \; \psi(x,y)=0\}$ with $\psi (x,y)=y-\cos(x)$. Note that
$\Gamma \cap \partial \Omega\ne\emptyset$ and
the exact solution $u$ of \eqref{Qeques1} is given by
\[
u_{+}=u\chi_{\Op}=-\sin(x)(y-\cos(x))^2,
\quad u_{-}=u\chi_{\Om}=-\cos(y)(y-\cos(x))^2-10.
\]
All the associated functions $f,g,g_0,g_1$ can be obtained by plugging the above exact solution into \eqref{Qeques1}. In particular,
$g_1=10$ and $g=0$.
The numerical results are provided in \cref{table:QSp3} and \cref{fig:figure2}.			 
\end{example}

\begin{table}[htbp]
	\caption{Performance in \cref{ex3}  of the proposed sixth order compact finite difference scheme  in \cref{thm:regular,fluxtm2}  on uniform Cartesian meshes with $h=2^{-J}\times 3\pi$.}
	\centering
	\setlength{\tabcolsep}{3mm}{
		 \begin{tabular}{c|c|c|c|c|c|c|c|c}
			\hline
			$J$
& $\frac{\|u_{h}-u\|_{2}}{\|u\|_{2}}$

&order & $\|u_{h}-u\|_{\infty}$

&order &  $\frac{\|u_{h}-u_{h/2}\|_{2}}{\|u_{h/2}\|_{2}}$

&order &  $\|u_{h}-u_{h/2}\|_{\infty}$

&order \\
			\hline
3   &1.72E-01   &0   &3.28E+00   &0   &1.68E-01   &0   &3.22E+00   &0\\
4   &3.78E-03   &5.508   &7.36E-02   &5.476   &3.77E-03   &5.483   &7.34E-02   &5.454\\
5   &1.28E-05   &8.206   &2.46E-04   &8.224   &1.26E-05   &8.224   &2.42E-04   &8.244\\
6   &1.97E-07   &6.024   &4.25E-06   &5.856   &1.96E-07   &6.009   &4.23E-06   &5.839\\
7   &1.03E-09   &7.577   &2.19E-08   &7.603   &1.02E-09   &7.586   &2.16E-08   &7.611\\
8   &1.17E-11   &6.462   &2.68E-10   &6.348   &1.26E-11   &6.341   &2.81E-10   &6.265\\
			\hline
	\end{tabular}}
	\label{table:QSp3}
\end{table}

\begin{example}\label{ex4}
\normalfont
Let $\Omega=(0,3)^2$ and the interface curve be given by $\Gamma:=\{ (x,y)\in \Omega \; : \; \psi(x,y)=0\}$ with  $\psi (x,y)=x^2/2+y^2/2-2$. Note that
$\Gamma \cap \partial \Omega\ne\emptyset$ and
the exact solution $u$ of \eqref{Qeques1} is given by
\[
u_{+}=u\chi_{\Op}=\sin(\pi x),
\quad u_{-}=u\chi_{\Om}=\sin(\pi x)+2.
\]
All the associated functions $f,g,g_0,g_1$ can be obtained by plugging the above exact solution into \eqref{Qeques1}. In particular,
$g_1=-2$ and $g=0$.
The numerical results are provided in \cref{table:QSp4} and \cref{fig:figure2}.			 
\end{example}

\begin{table}[htbp]
	\caption{Performance in \cref{ex4}  of the proposed  sixth order compact finite difference scheme in \cref{thm:regular,fluxtm2} on uniform Cartesian meshes with $h= 2^{-J}\times 3$.}
	\centering
	\setlength{\tabcolsep}{3mm}{
		 \begin{tabular}{c|c|c|c|c|c|c|c|c}
			\hline
			$J$
& $\frac{\|u_{h}-u\|_{2}}{\|u\|_{2}}$

&order & $\|u_{h}-u\|_{\infty}$

&order &  $\frac{\|u_{h}-u_{h/2}\|_{2}}{\|u_{h/2}\|_{2}}$

&order &  $\|u_{h}-u_{h/2}\|_{\infty}$

&order \\
			\hline
3   &2.96E-03   &0   &1.02E-02   &0   &2.93E-03   &0   &1.02E-02   &0\\
4   &2.93E-05   &6.655   &1.27E-04   &6.333   &2.91E-05   &6.652   &1.26E-04   &6.336\\
5   &2.16E-07   &7.086   &1.03E-06   &6.938   &2.14E-07   &7.086   &1.02E-06   &6.940\\
6   &1.66E-09   &7.025   &8.25E-09   &6.967   &1.65E-09   &7.025   &8.19E-09   &6.968\\
7   &1.28E-11   &7.013   &6.70E-11   &6.943   &1.30E-11   &6.989   &6.71E-11   &6.931\\
8   &2.65E-13   &5.598   &9.91E-13   &6.080   &9.54E-13   &3.763   &3.06E-12   &4.454\\
			\hline
	\end{tabular}}
	\label{table:QSp4}
\end{table}

\begin{figure}[htbp]
	\centering
	 \begin{subfigure}[b]{0.3\textwidth}
		 \includegraphics[width=5.7cm,height=5.7cm]{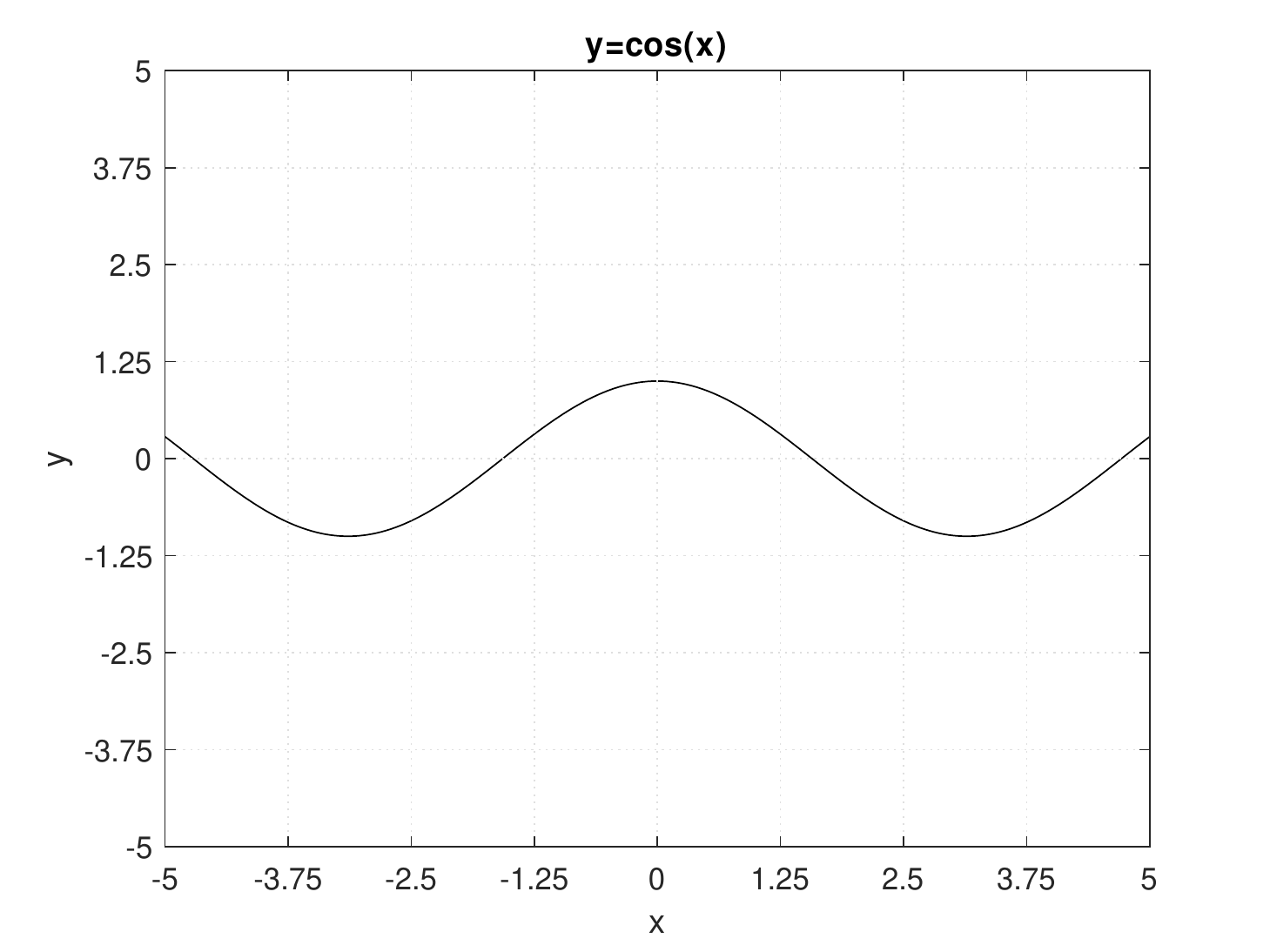}
	\end{subfigure}
	 \begin{subfigure}[b]{0.3\textwidth}
		 \includegraphics[width=5.7cm,height=5.7cm]{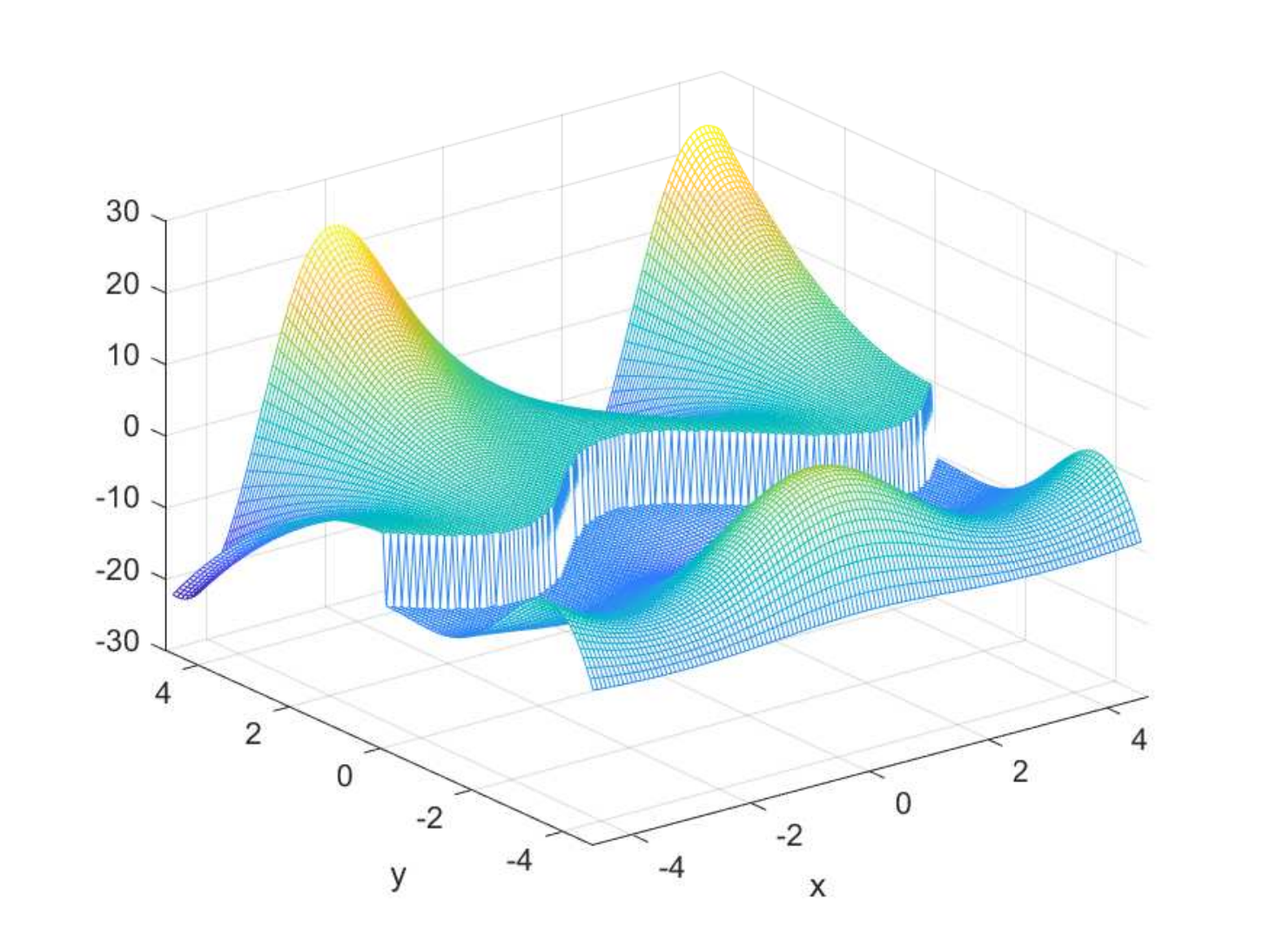}
	\end{subfigure}
	 \begin{subfigure}[b]{0.3\textwidth}
		 \includegraphics[width=5.7cm,height=5.7cm]{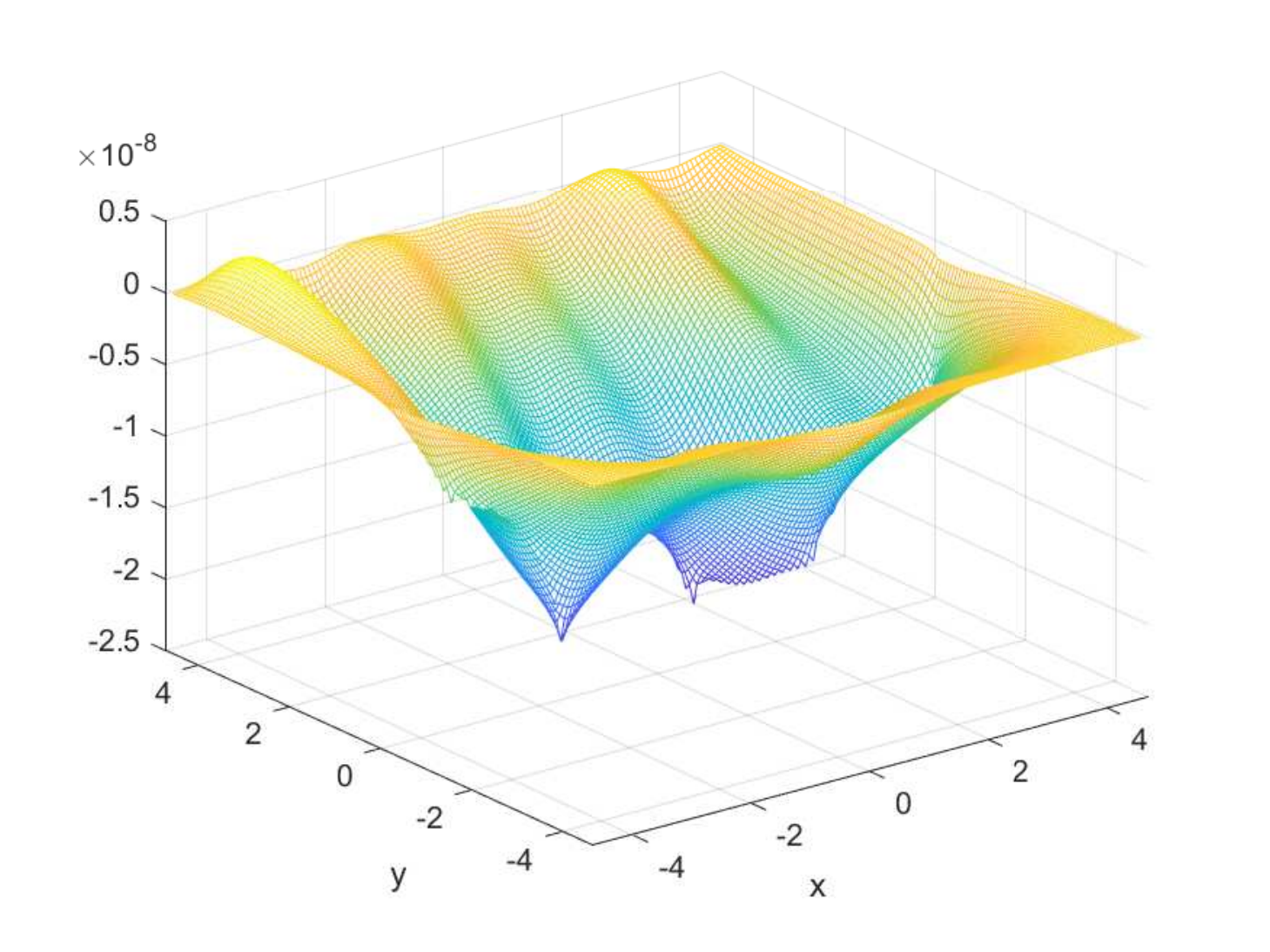}
	\end{subfigure}
	 \begin{subfigure}[b]{0.3\textwidth}
		 \includegraphics[width=5.7cm,height=5.7cm]{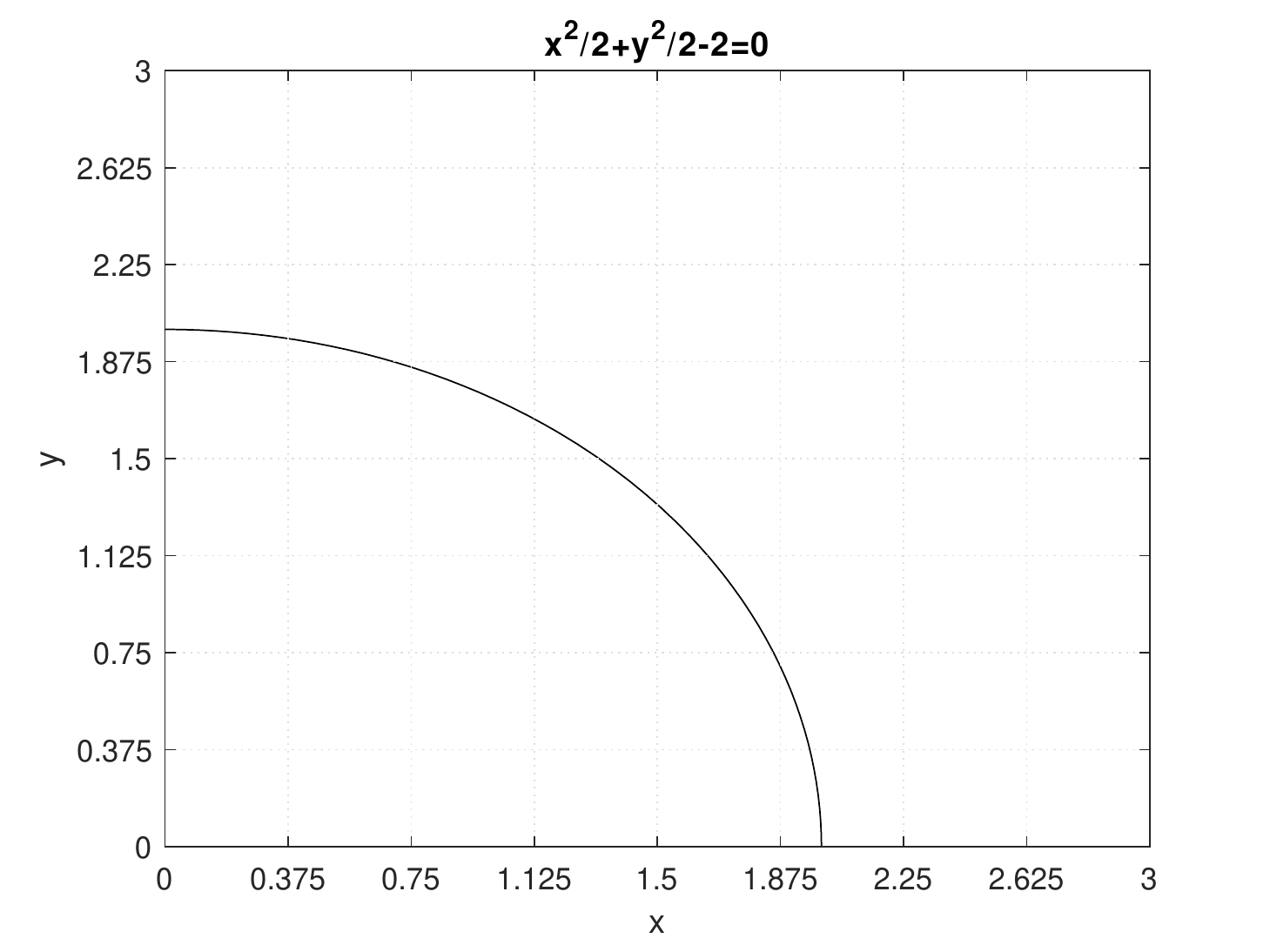}
	\end{subfigure}
	 \begin{subfigure}[b]{0.3\textwidth}
		 \includegraphics[width=5.7cm,height=5.7cm]{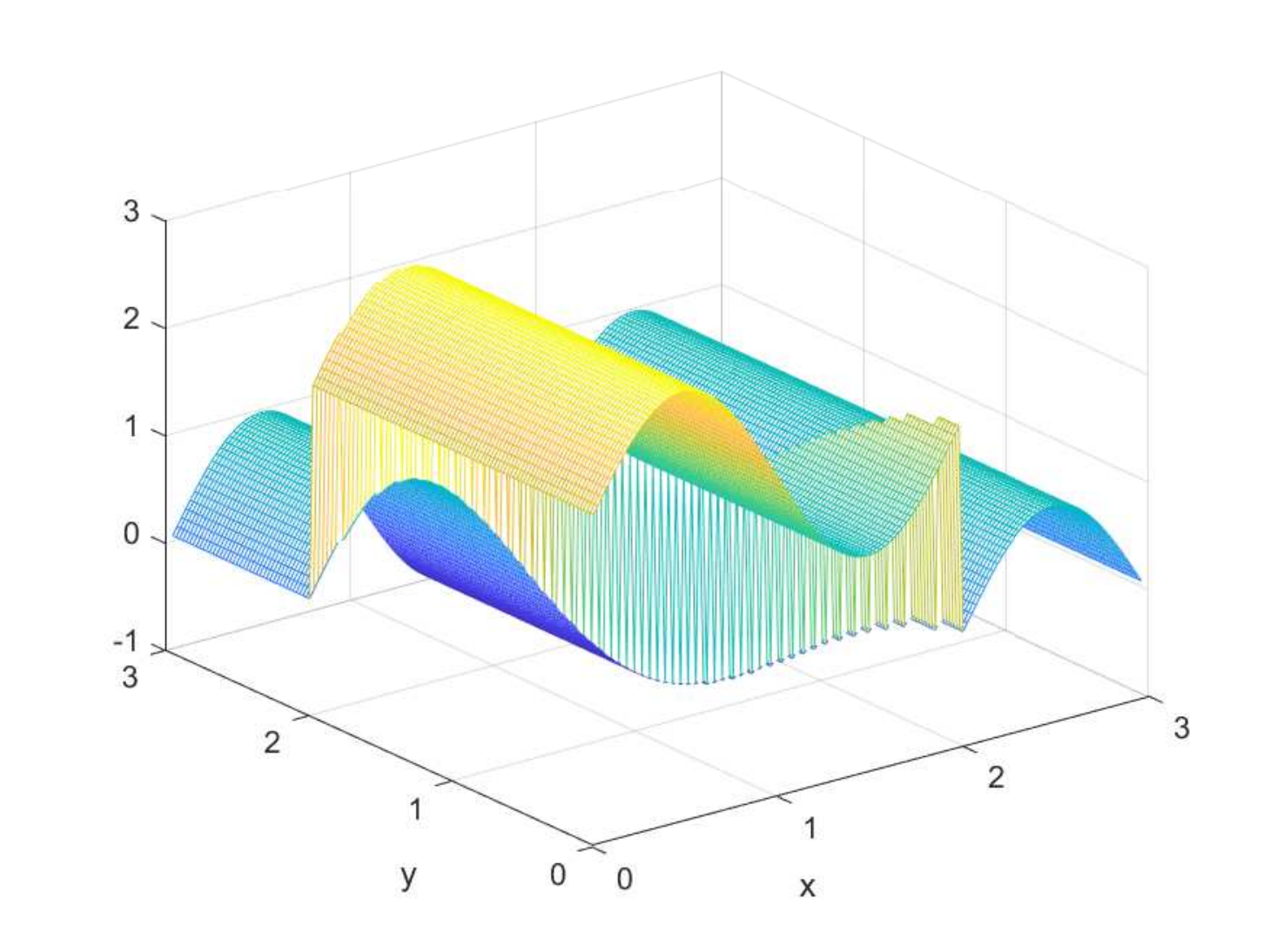}
	\end{subfigure}
	 \begin{subfigure}[b]{0.3\textwidth}
		 \includegraphics[width=5.7cm,height=5.7cm]{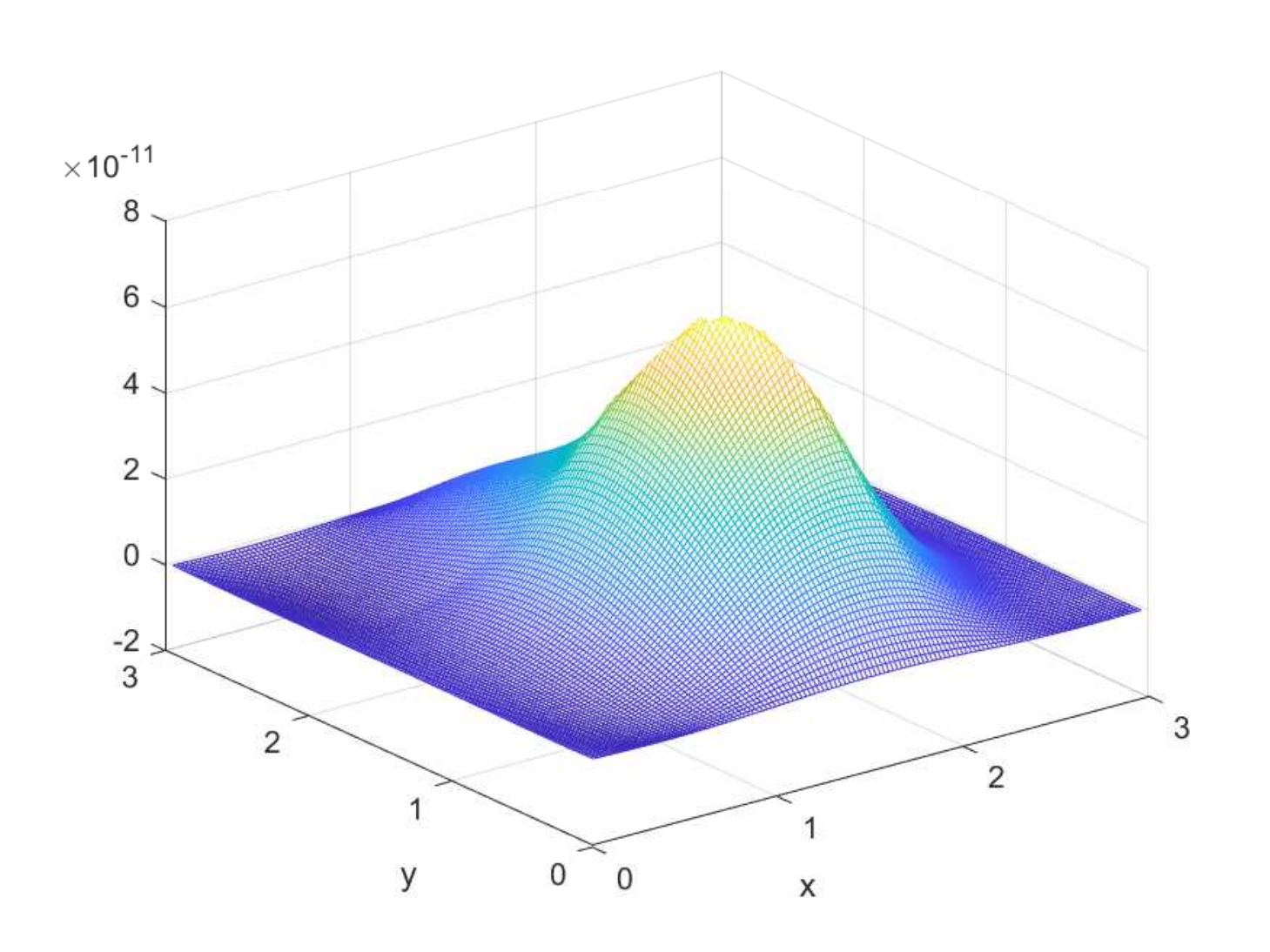}
	\end{subfigure}

	\caption
	{Top row for \cref{ex3}: the interface curve $\Gamma$ (left), the numerical solution $u_h$ (middle) and the error $u-u_h$ (right) with $h=2^{-7}\times 3\pi$. Bottom row for \cref{ex4}: the interface curve $\Gamma$ (left), the numerical solution $u_h$ (middle) and the error $u-u_h$ (right) with $h=2^{-7}\times 3$.}
	\label{fig:figure2}
\end{figure}

\subsection{Numerical examples with $u$ unknown and $\Gamma \cap \partial \Omega=\emptyset$}

In this subsection, we provide a few numerical experiments such that the exact solution $u$ of \eqref{Qeques1} is unknown and the interface curve $\Gamma$ does not touch the boundary of $\Omega$.

\begin{example}\label{ex5}
\normalfont
Let $\Omega=(-\pi,\pi)^2$ and
the interface curve be given by
$\Gamma:=\{(x,y)\in \Omega \; :\; \psi(x,y)=0\}$ with
$\psi (x,y)=\frac{x^2}{2}+\frac{y^2}{2}-1$. Note that $\Gamma \cap \partial \Omega=\emptyset$ and
the coefficients of \eqref{Qeques1} are given by
\begin{align*}
&f_{+}=f\chi_{\Op}=\sin(3x)\sin(3y),
\qquad f_{-}=f\chi_{\Om}=\cos(3x)\cos(3y),\\
&g_0=0, \qquad g_1=-\exp(x-y)\sin(x+y),
\qquad g=-\exp(x+y)\cos(x-y).
\end{align*}
The numerical results are provided in \cref{table:QSp5} and \cref{fig:figure3}.		 
\end{example}

\begin{example}\label{ex6}	
\normalfont
Let $\Omega=(-\pi,\pi)^2$ and
the interface curve be given by
$\Gamma:=\{(x,y)\in \Omega \; :\; \psi(x,y)=0\}$ with
$\psi (x,y)=\frac{y^2}{2}+\frac{x^2}{1+x^2}-\frac{1}{2}$. Note that $\Gamma \cap \partial \Omega=\emptyset$ and
the coefficients of \eqref{Qeques1} are given by
\begin{align*}
&f_{+}=f\chi_{\Op}=\sin(3x)\sin(2y),
\qquad f_{-}=f\chi_{\Om}=\cos(2x)\cos(2y),\\
&g_0=0,
\qquad g_1=-\sin(x)\sin(y),
\qquad g=-\cos(x)\sin(y).
\end{align*}
The numerical results are provided in \cref{table:QSp5} and \cref{fig:figure3}.	 
\end{example}

\begin{table}[htbp]
	\caption{Performance in \cref{ex5} and  \cref{ex6} of the proposed  sixth order compact finite difference scheme in \cref{thm:regular,fluxtm2} on uniform Cartesian meshes with the same mesh size $h=2^{-J}\times 2\pi$.}
	\centering
	\setlength{\tabcolsep}{2.5mm}{
		 \begin{tabular}{c|c|c|c|c|c|c|c|c}
           \hline
           \multicolumn{1}{c|}{} &
           \multicolumn{4}{c|}{\cref{ex5}} &
           \multicolumn{4}{c}{ \cref{ex6}} \\
           \cline{1-9}
$J$&   $\frac{\|u_{h}-u_{h/2}\|_2}{\|u_{h/2}\|_2}$    &order &   $\|u_{h}-u_{h/2}\|_\infty$    &order &   $\frac{\|u_{h}-u_{h/2}\|_2}{\|u_{h/2}\|_2}$    &order &   $\|u_{h}-u_{h/2}\|_{\infty}$    &order \\
			\hline
3   &4.42E-01   &0   &1.48E+00   &0   &5.26E+02   &0   &4.01E+02   &0\\
4   &4.67E-02   &3.245   &9.76E-02   &3.919   &6.79E-02   &12.919   &6.00E-02   &12.705\\
5   &5.01E-04   &6.541   &1.01E-03   &6.589   &8.20E-04   &6.372   &9.18E-04   &6.032\\
6   &4.34E-06   &6.850   &1.01E-05   &6.647   &1.31E-05   &5.974   &1.44E-05   &5.992\\
7   &5.95E-08   &6.190   &4.52E-07   &4.483   &2.09E-07   &5.967   &2.61E-07   &5.789\\
			\hline
	\end{tabular}}
	\label{table:QSp5}
\end{table}

\begin{figure}[htbp]
	\centering
	 \begin{subfigure}[b]{0.45\textwidth}
		 \includegraphics[width=8cm,height=8cm]{AA1.pdf}
	\end{subfigure}
	 \begin{subfigure}[b]{0.45\textwidth}
		 \includegraphics[width=8cm,height=8cm]{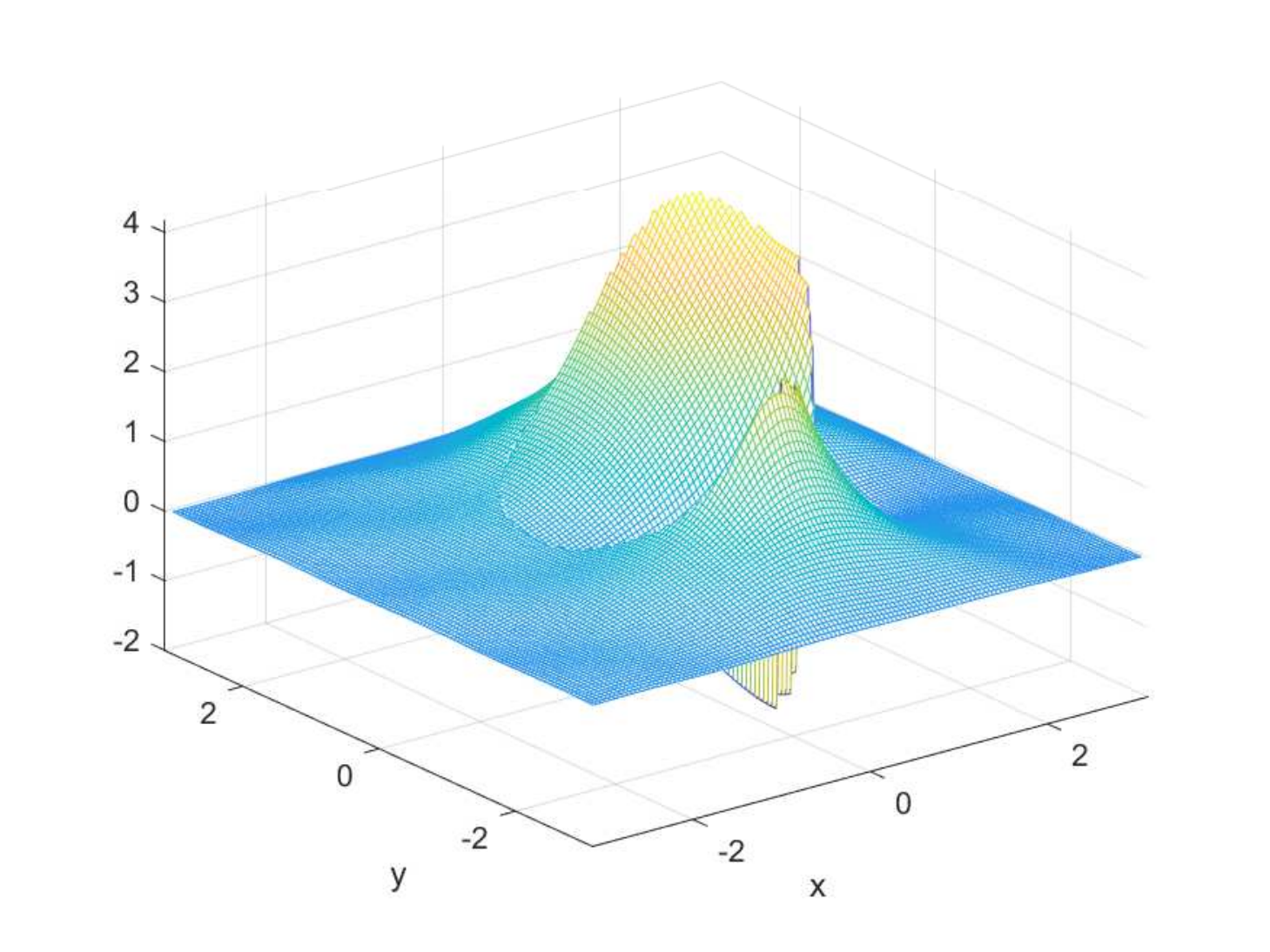}
	\end{subfigure}
	 \begin{subfigure}[b]{0.45\textwidth}
		 \includegraphics[width=8cm,height=8cm]{AA2.pdf}
	\end{subfigure}
	 \begin{subfigure}[b]{0.45\textwidth}
		 \includegraphics[width=8cm,height=8cm]{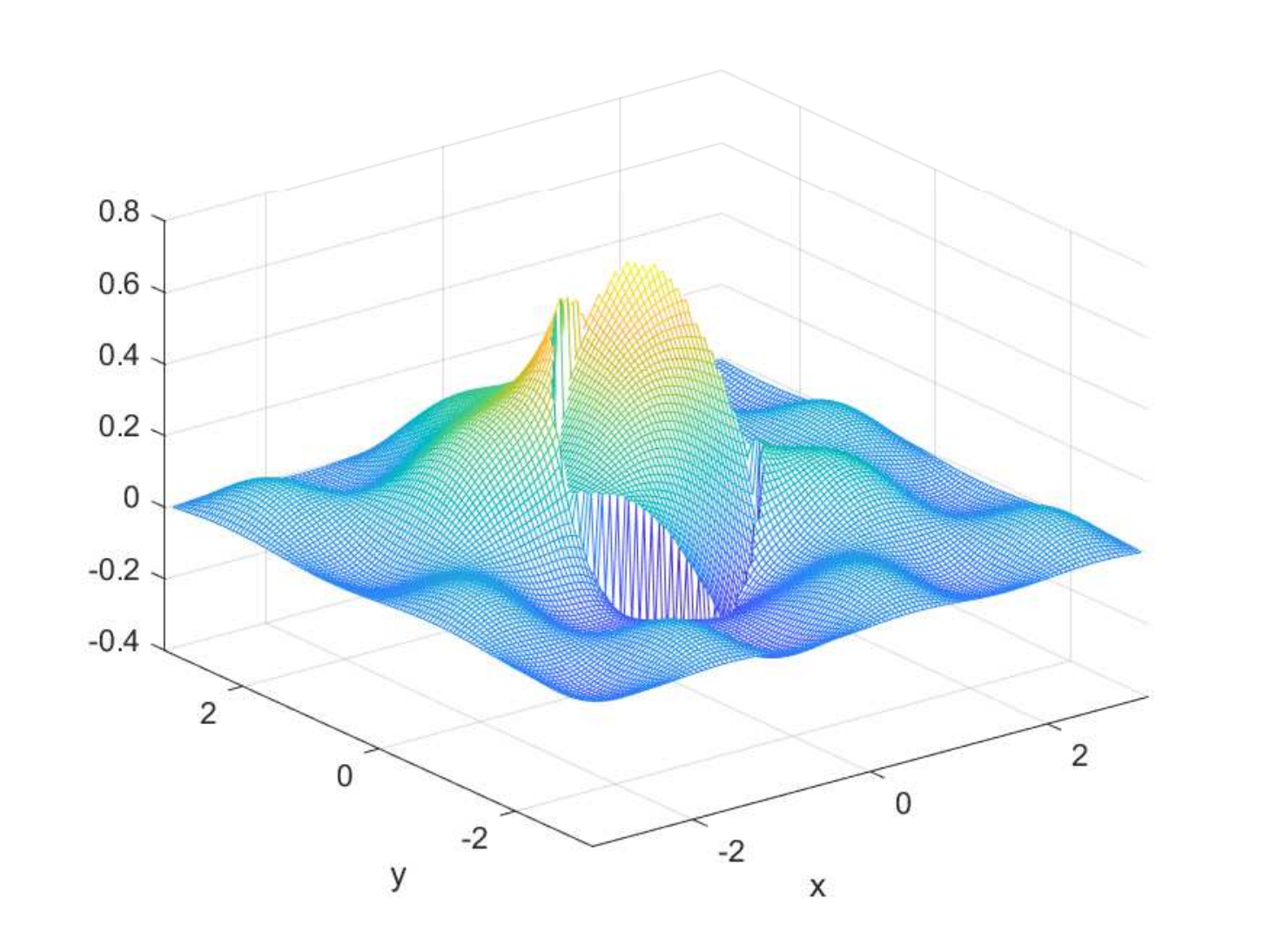}
	\end{subfigure}
	\caption
	{Top row for \cref{ex5}: the interface curve $\Gamma$ (left) and the numerical solution $u_h$ (right) with $h=2^{-7}\times 2\pi$. Bottom row for \cref{ex6}: the interface curve $\Gamma$ (left) and the numerical solution $u_h$ (right) with $h=2^{-7}\times 2\pi$.}
	\label{fig:figure3}
\end{figure}

\begin{example}\label{ex7}
\normalfont
Let $\Omega=(-\pi,\pi)^2$ and
the interface curve be given by
$\Gamma:=\{(x,y)\in \Omega \; :\; \psi(x,y)=0\}$ with
$\psi (x,y)=x^4+2y^4-2$. Note that $\Gamma \cap \partial \Omega=\emptyset$ and
the coefficients of \eqref{Qeques1} are given by
\begin{align*}
& f_{+}=f\chi_{\Op}=\sin(2x)\sin(2y),
\quad f_{-}=f\chi_{\Om}=\cos(2x-2y),\\
& g_0=0,
\quad g_1=-x^2,
\quad g=-y^2.
\end{align*}
The numerical results are provided in \cref{table:QSp6} and \cref{fig:figure4}.		 
\end{example}

\begin{example}\label{ex8}
\normalfont
Let $\Omega=(-\pi,\pi)^2$ and
the interface curve be given by
$\Gamma:=\{(x,y)\in \Omega \; :\; \psi(x,y)=0\}$ with
$\psi (x,y)=y^2-2x^2+x^4-1$. Note that $\Gamma \cap \partial \Omega=\emptyset$ and
the coefficients of \eqref{Qeques1} are given by
\begin{align*}
& f_{+}=f\chi_{\Op}=\sin(2x)\sin(3y),
\quad f_{-}=f\chi_{\Om}=\cos(2x)\sin(2y),\\
& g_0=0,
\quad g_1=0,
\quad g=-\exp(x-2y).
\end{align*}
Because $g_1=0$, the Poisson interface problem in \eqref{Qeques1} simply becomes $-\nabla^2 u=f-g\delta_\Gamma$ in $\Omega$ with the Dirichlet boundary condition $u|_{\partial \Omega}=g_0$.
The numerical results are provided in \cref{table:QSp6} and \cref{fig:figure4}.	 
\end{example}

\begin{table}[htbp]
	\caption{Performance in \cref{ex7} and  \cref{ex8} of the proposed sixth order compact finite difference scheme  in \cref{thm:regular,fluxtm2} on uniform Cartesian meshes with the same mesh size $h=2^{-J}\times 2\pi$.}
	\centering
	\setlength{\tabcolsep}{2.5mm}{
		 \begin{tabular}{c|c|c|c|c|c|c|c|c}
			\hline
			\multicolumn{1}{c|}{} &
			 \multicolumn{4}{c|}{\cref{ex7}} &
			 \multicolumn{4}{c}{\cref{ex8}} \\
			\cline{1-9}
$J$&   $\frac{\|u_{h}-u_{h/2}\|_2}{\|u_{h/2}\|_2}$    &order &   $\|u_{h}-u_{h/2}\|_\infty$    &order &   $\frac{\|u_{h}-u_{h/2}\|_2}{\|u_{h/2}\|_2}$    &order &   $\|u_{h}-u_{h/2}\|_{\infty}$    &order \\
			\hline
3   &5.05E+00   &0   &8.15E+00   &0   &2.88E+01   &0   &5.64E+02   &0\\
4   &5.56E-02   &6.502   &1.01E-01   &6.329   &4.25E-01   &6.083   &1.92E+01   &4.878\\
5   &1.39E-03   &5.328   &2.97E-03   &5.092   &2.41E-02   &4.144   &1.47E+00   &3.708\\
6   &2.78E-05   &5.637   &8.57E-05   &5.116   &1.41E-04   &7.413   &9.31E-03   &7.300\\
7   &1.55E-07   &7.485   &6.83E-07   &6.971   &8.88E-07   &7.313   &5.83E-05   &7.320\\
			\hline
	\end{tabular}}
	\label{table:QSp6}
\end{table}

\begin{figure}[htbp]
	\centering
	 \begin{subfigure}[b]{0.45\textwidth}
		 \includegraphics[width=8cm,height=8cm]{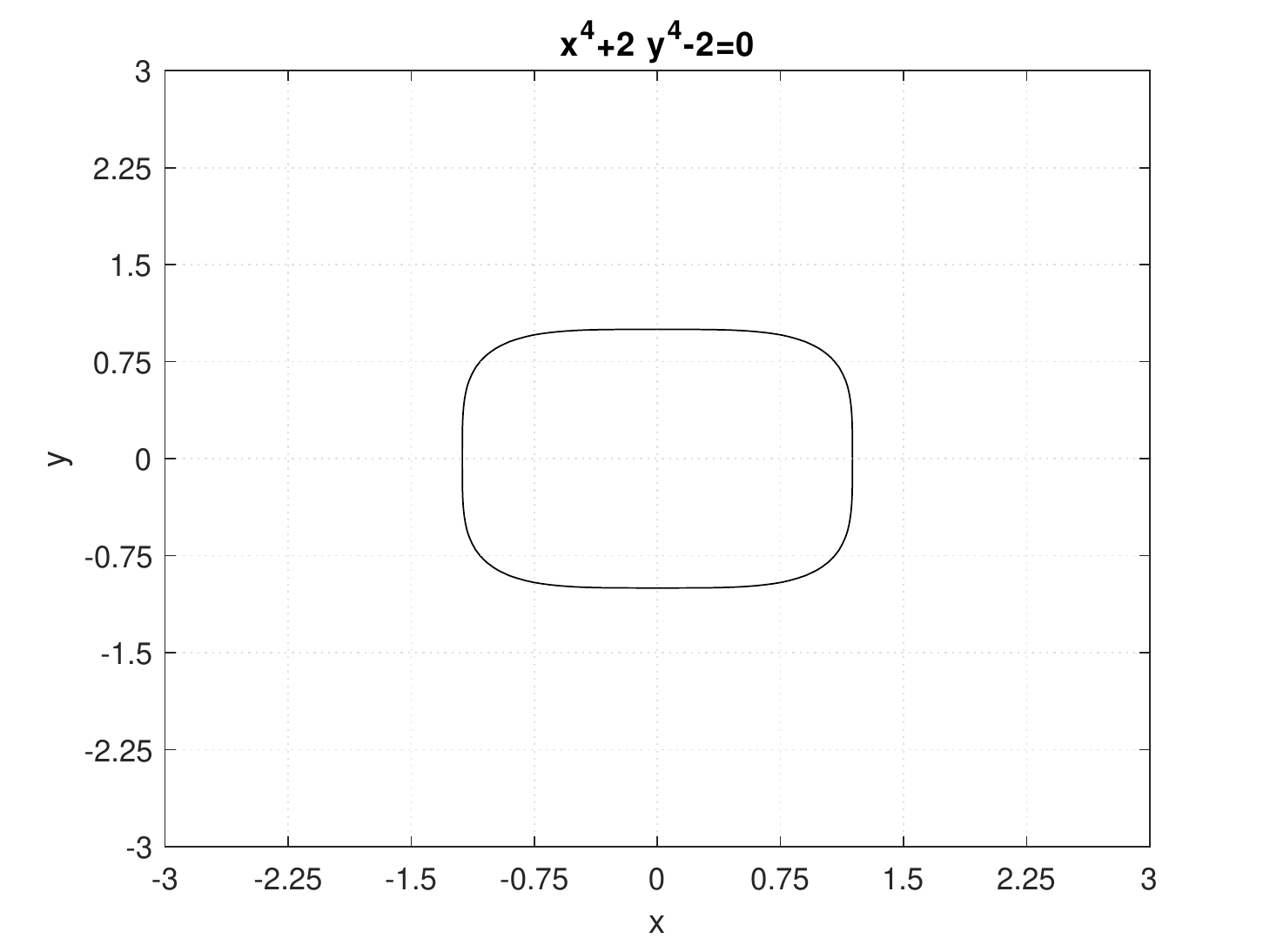}
	\end{subfigure}
	 \begin{subfigure}[b]{0.45\textwidth}
		 \includegraphics[width=8cm,height=8cm]{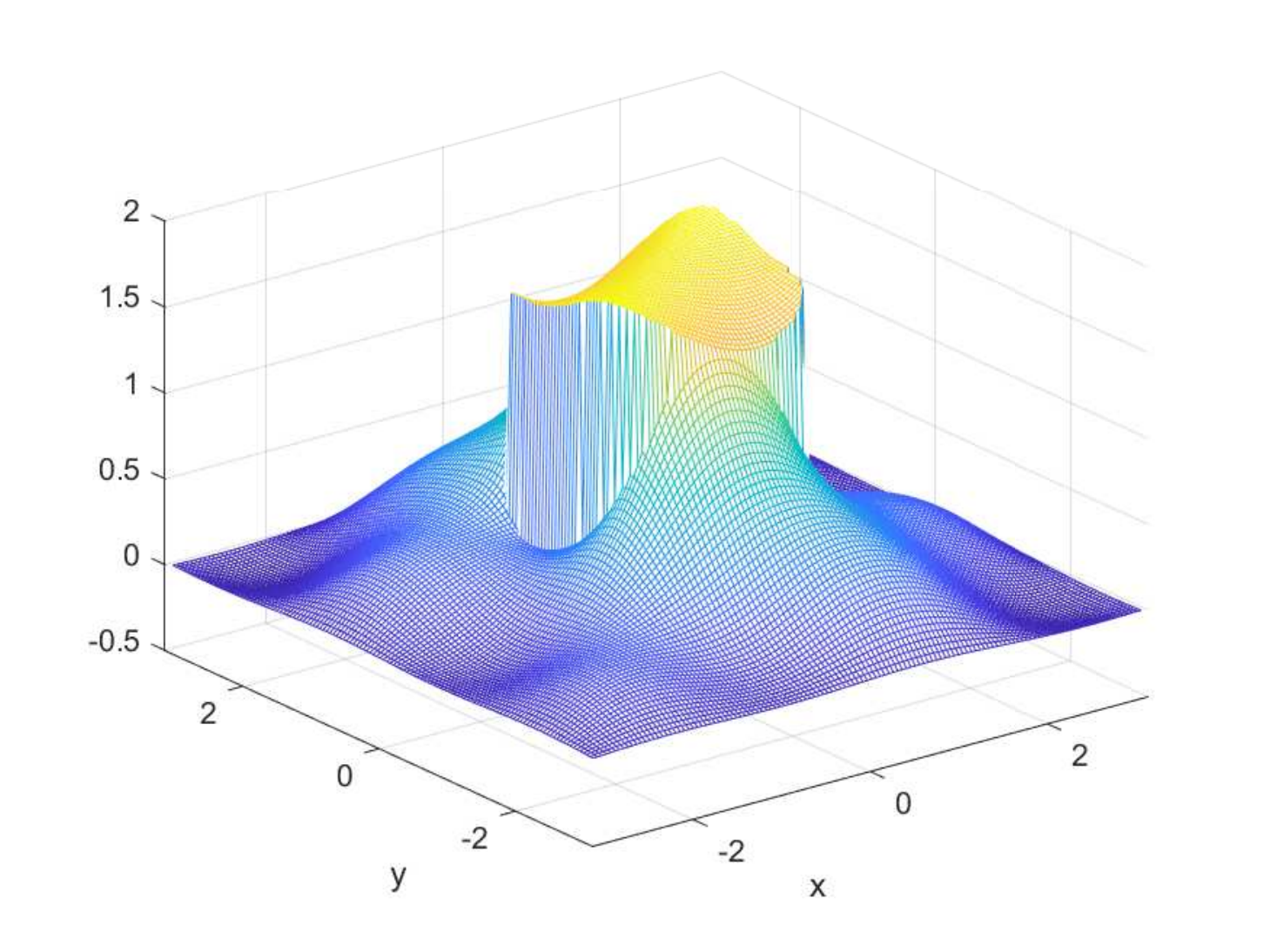}
	\end{subfigure}
	 \begin{subfigure}[b]{0.45\textwidth}
		 \includegraphics[width=8cm,height=8cm]{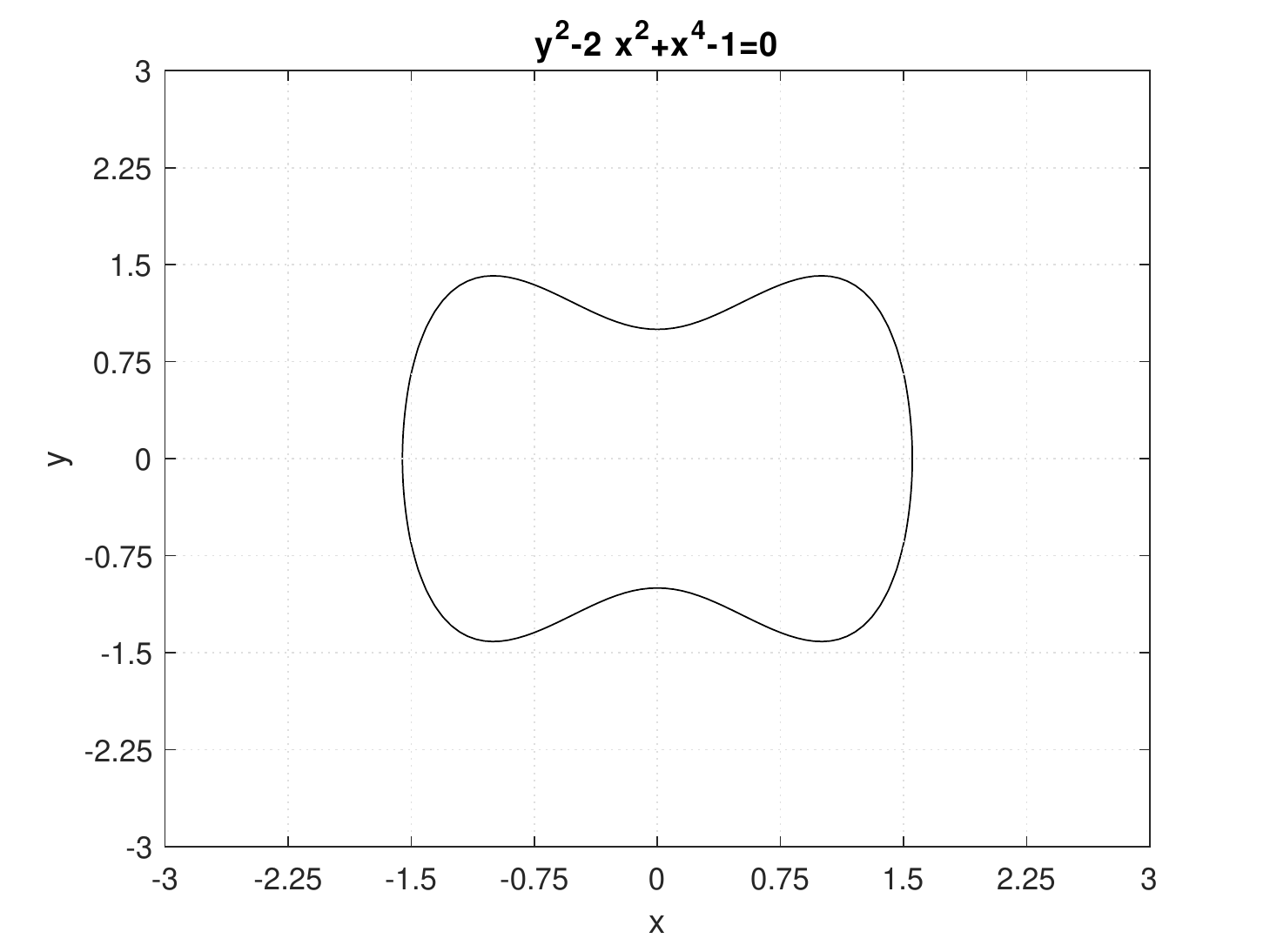}
	\end{subfigure}
	 \begin{subfigure}[b]{0.45\textwidth}
		 \includegraphics[width=8cm,height=8cm]{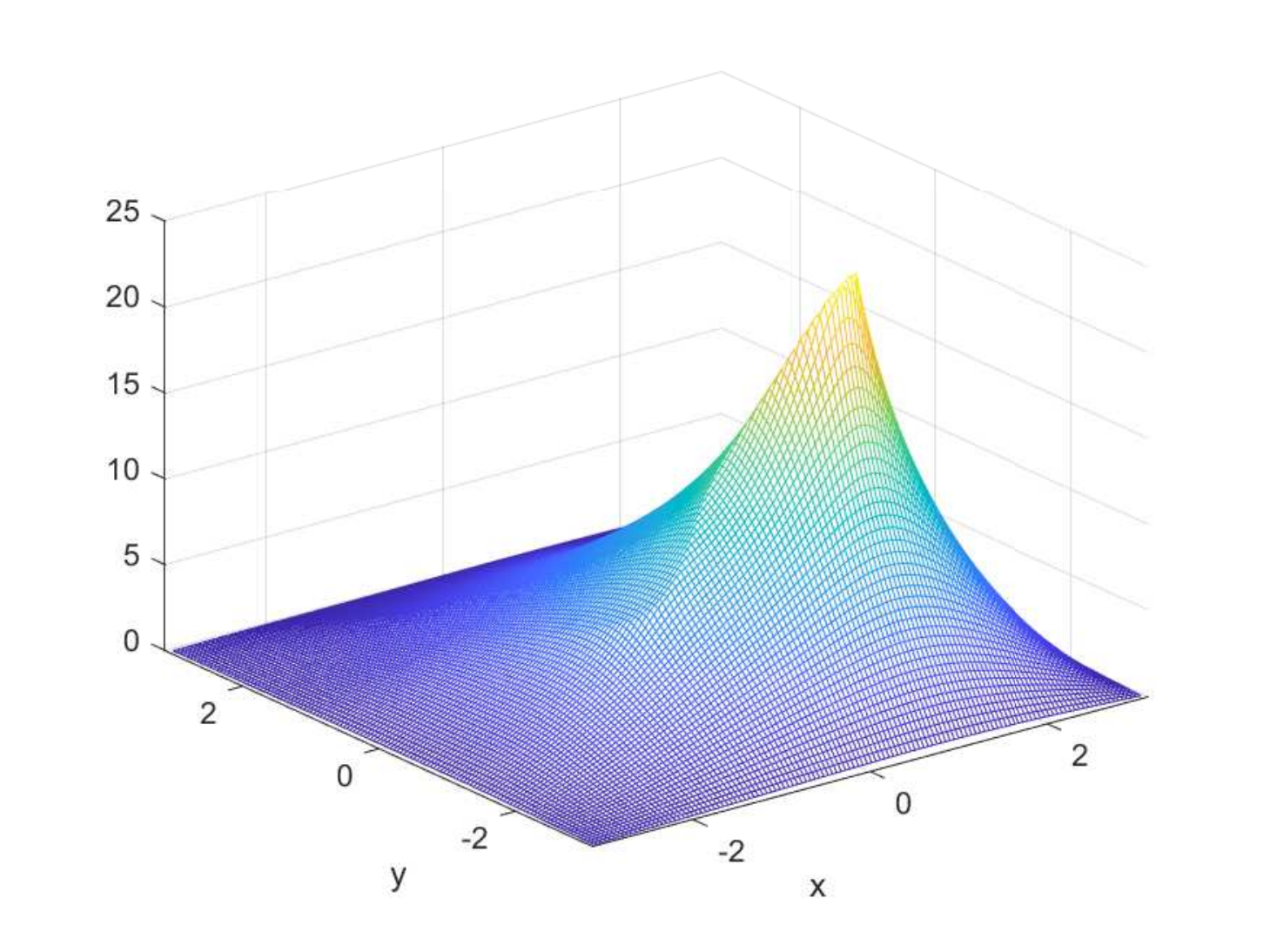}
	\end{subfigure}
	\caption
	{Top row for \cref{ex7}: the interface curve $\Gamma$ (left) and the numerical solution $u_h$ (right) with $h=2^{-7}\times 2\pi$. Bottom row for \cref{ex8}: the interface curve $\Gamma$ (left) and the numerical solution $u_h$ (right) with $h=2^{-7}\times 2\pi$.}
	\label{fig:figure4}
\end{figure}

\begin{example}\label{ex9}	
\normalfont
Let $\Omega=(-2,2)^2$ and
the interface curve be given by
$\Gamma:=\{(x,y)\in \Omega \; :\; \psi(x,y)=0\}$ with
$\psi (x,y)=2x^4+y^2-1/2$. Note that $\Gamma \cap \partial \Omega=\emptyset$ and
the coefficients of \eqref{Qeques1} are given by
\begin{align*}
&f_{+}=f\chi_{\Op}=\sin(2\pi x)\sin(2\pi y),
\qquad f_{-}=f\chi_{\Om}=\sin(\pi(x+2y)),\\
&g_0=0,
\quad g_1=-\sin(x+y),
\quad g=-\sin(2x)-\sin(2y).
\end{align*}
The numerical results are provided in \cref{table:QSp7} and \cref{fig:figure5}.		 
\end{example}

\subsection{Numerical examples with $u$ unknown and $\Gamma \cap \partial \Omega\ne\emptyset$}

In this subsection, we provide a few numerical experiments such that the exact solution $u$ of \eqref{Qeques1} is unknown and the interface curve $\Gamma$ touches the boundary of $\Omega$.

\begin{example}\label{ex10}
\normalfont
Let $\Omega=(-\pi,\pi)^2$ and
the interface curve be given by
$\Gamma:=\{(x,y)\in \Omega \; :\; \psi(x,y)=0\}$ with
$\psi (x,y)=y-\cos(x)$. Note that $\Gamma \cap \partial \Omega\ne\emptyset$ and
the coefficients of \eqref{Qeques1} are given by
\begin{align*}
& f_{+}=f\chi_{\Op}=-\sin(x)\sin(3y),
\quad f_{-}=f\chi_{\Om}=-\sin(2x)\sin(y),\\
& g_0=0,
\quad g_1=0,
\quad g=\sin(x).
\end{align*}
Because $g_1=0$, the Poisson interface problem in \eqref{Qeques1} simply becomes $-\nabla^2 u=f-g\delta_\Gamma$ in $\Omega$ with the Dirichlet boundary condition $u|_{\partial \Omega}=g_0$.
The numerical results are provided in \cref{table:QSp7} and \cref{fig:figure5}.
\end{example}

\begin{table}[htbp]
	\caption{Performance in \cref{ex9} and  \cref{ex10} of the proposed  sixth order compact finite difference scheme  in \cref{thm:regular,fluxtm2} on uniform Cartesian meshes with the mesh sizes $h=2^{-J}\times 4$ and $h=2^{-J}\times 2\pi$, respectively.}
	\centering
	\setlength{\tabcolsep}{2.5mm}{
		 \begin{tabular}{c|c|c|c|c|c|c|c|c}
			\hline
			\multicolumn{1}{c|}{} &
			 \multicolumn{4}{c|}{\cref{ex9}} &
			 \multicolumn{4}{c}{\cref{ex10}} \\
			\cline{1-9}
$J$&   $\frac{\|u_{h}-u_{h/2}\|_2}{\|u_{h/2}\|_2}$    &order &   $\|u_{h}-u_{h/2}\|_\infty$    &order &   $\frac{\|u_{h}-u_{h/2}\|_2}{\|u_{h/2}\|_2}$    &order &   $\|u_{h}-u_{h/2}\|_{\infty}$    &order \\
			\hline
3   &6.36E+00   &0   &5.46E+00   &0   &3.21E+00   &0   &2.71E+00   &0\\
4   &8.46E-02   &6.232   &6.56E-02   &6.379   &2.80E-02   &6.839   &2.37E-02   &6.840\\
5   &7.95E-04   &6.734   &8.20E-04   &6.322   &2.01E-04   &7.120   &1.83E-04   &7.017\\
6   &3.86E-06   &7.688   &5.70E-06   &7.169   &1.30E-06   &7.270   &1.23E-06   &7.211\\
7   &7.95E-08   &5.600   &8.84E-08   &6.011   &7.62E-09   &7.418   &7.61E-09   &7.340\\
			\hline
	\end{tabular}}
	\label{table:QSp7}
\end{table}

\begin{figure}[htbp]
	\centering
	 \begin{subfigure}[b]{0.45\textwidth}
		 \includegraphics[width=8cm,height=8cm]{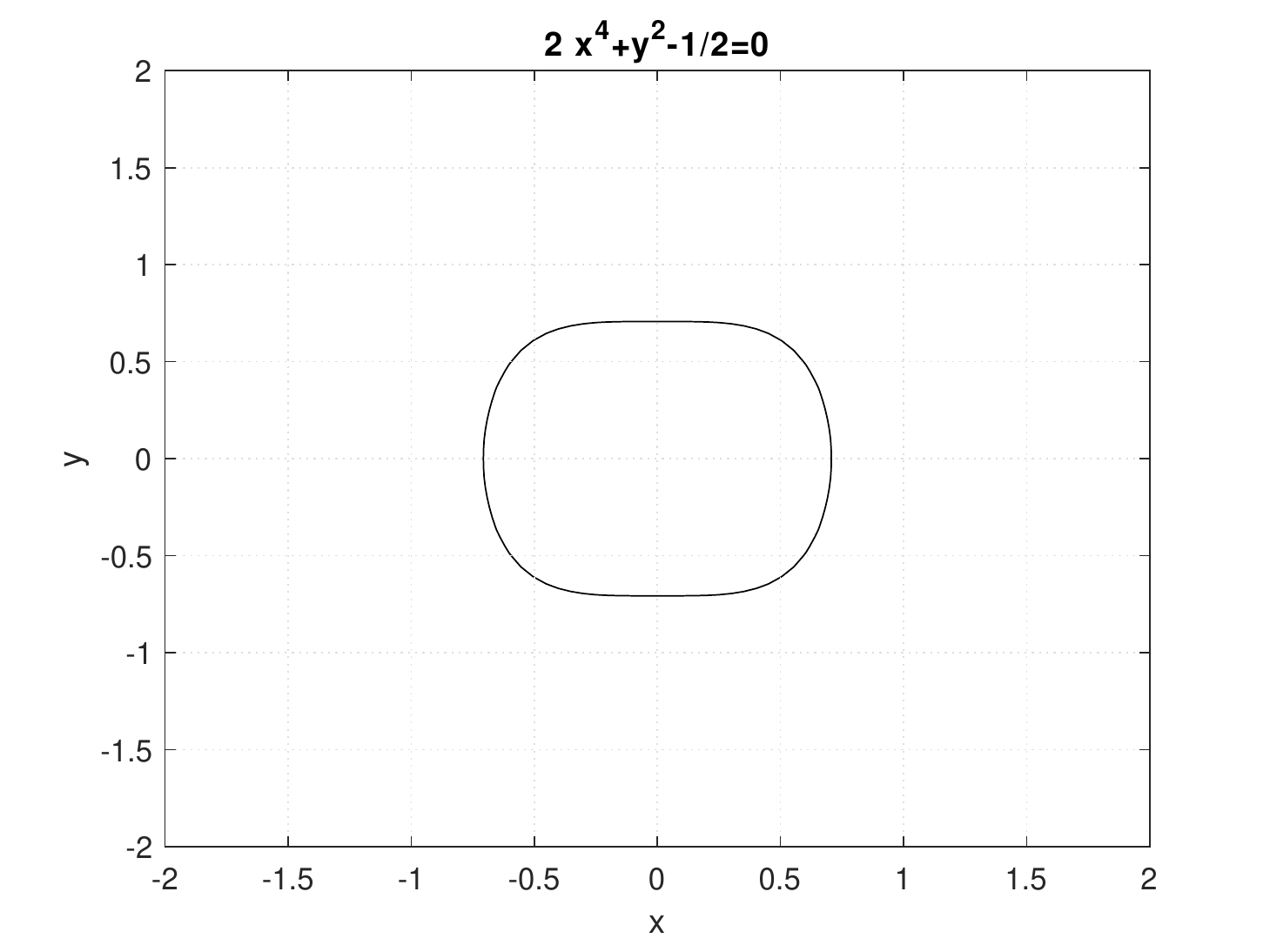}
	\end{subfigure}
	 \begin{subfigure}[b]{0.45\textwidth}
		 \includegraphics[width=8cm,height=8cm]{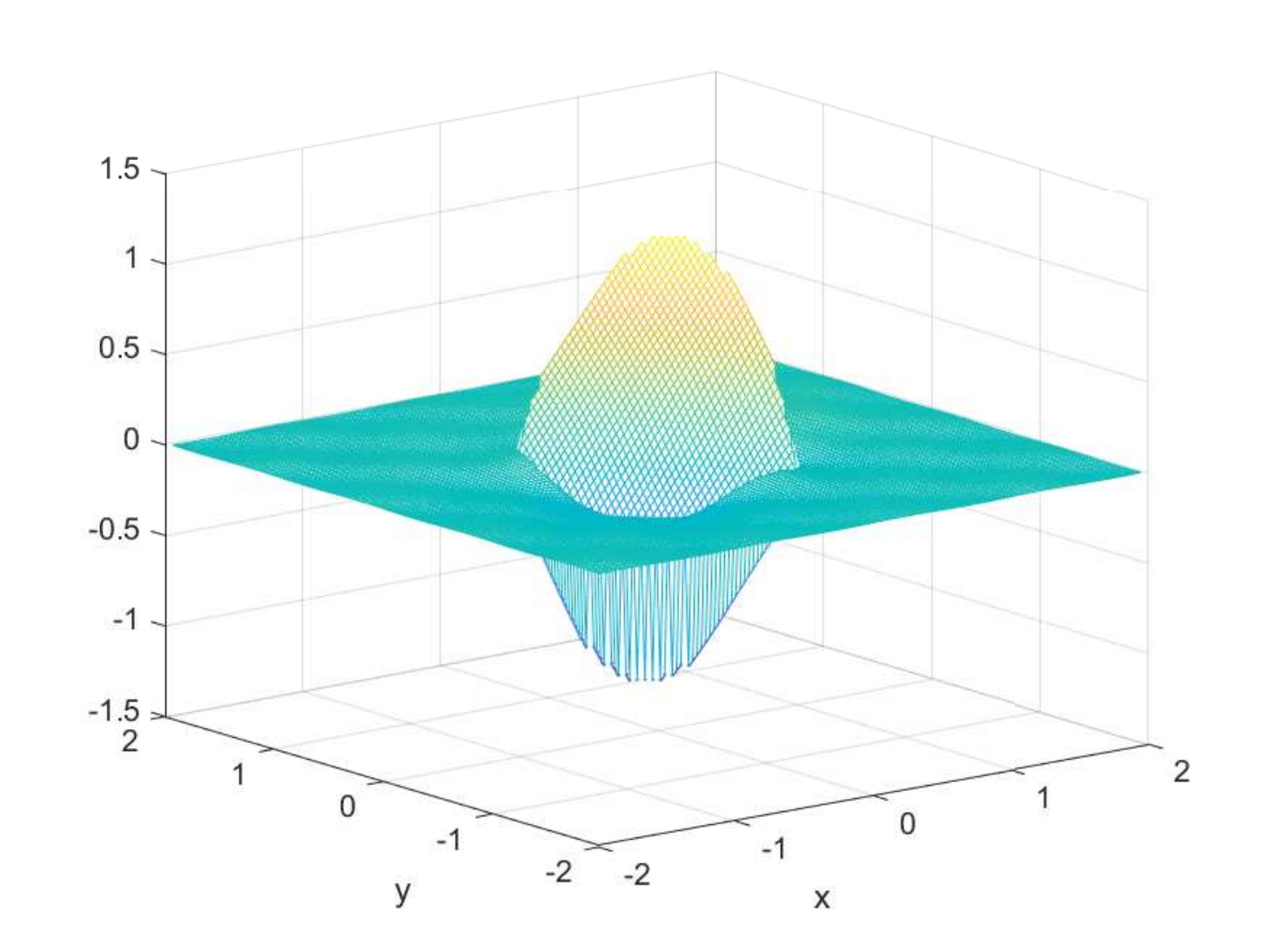}
	\end{subfigure}
	 \begin{subfigure}[b]{0.45\textwidth}
		 \includegraphics[width=8cm,height=8cm]{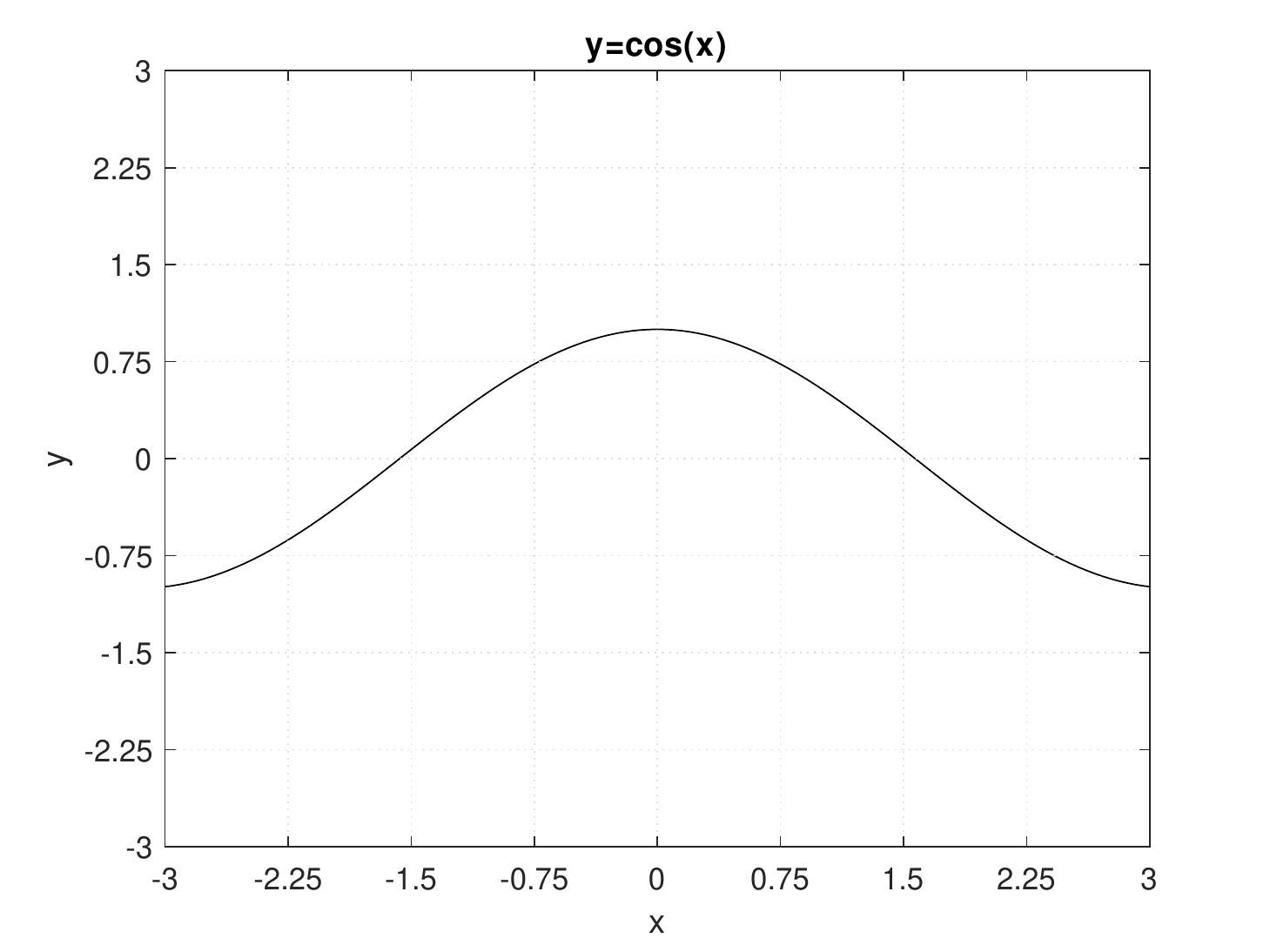}
	\end{subfigure}
	 \begin{subfigure}[b]{0.45\textwidth}
		 \includegraphics[width=8cm,height=8cm]{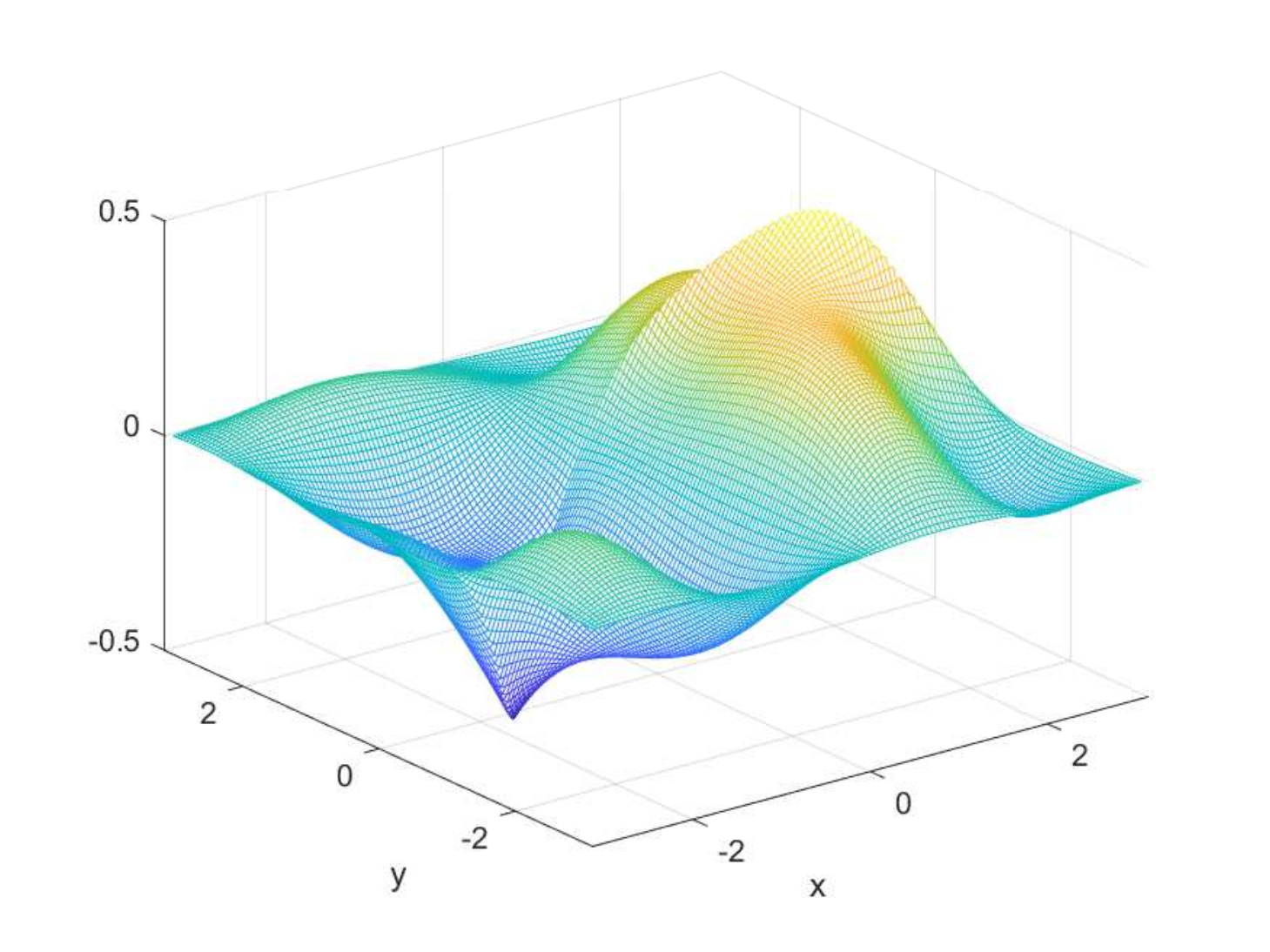}
	\end{subfigure}
	\caption
	{Top row for \cref{ex9}: the interface curve $\Gamma$ (left) and the numerical solution $u_h$ (right) with $h=2^{-7}\times 4$. Bottom row for \cref{ex10}: the interface curve $\Gamma$ (left) and the numerical solution $u_h$ (right) with $h=2^{-7}\times 2\pi$.}
	\label{fig:figure5}
\end{figure}

\begin{example}\label{ex11}
\normalfont
Let $\Omega=(0,3.5)^2$ and
the interface curve be given by
$\Gamma:=\{(x,y)\in \Omega \; :\; \psi(x,y)=0\}$ with
$\psi (x,y)=\frac{x^2}{2}+\frac{y^2}{2}-2$. Note that $\Gamma \cap \partial \Omega\ne\emptyset$ and
the coefficients of \eqref{Qeques1} are given by
\begin{align*}
&f_{+}=f\chi_{\Op}=\sin(\pi x)\sin(2\pi y),
\quad f_{-}=f\chi_{\Om}=\sin(2\pi x)\sin(\pi y),\\
& g_0=0,
\quad g_1=0,
\quad g=-\sin(2\pi x)\sin(2\pi y).
\end{align*}
Because $g_1=0$, the Poisson interface problem in \eqref{Qeques1} simply becomes $-\nabla^2 u=f-g\delta_\Gamma$ in $\Omega$ with the Dirichlet boundary condition $u|_{\partial \Omega}=g_0$.
The numerical results are provided in \cref{table:QSp8} and \cref{fig:figure6}.		 
\end{example}

\begin{example}\label{ex12}
\normalfont
Let $\Omega=(-\pi,\pi)^2$ and
the interface curve be given by
$\Gamma:=\{(x,y)\in \Omega \; :\; \psi(x,y)=0\}$ with
$\psi (x,y)=y-\sin(x)$. Note that $\Gamma \cap \partial \Omega\ne\emptyset$ and
the coefficients of \eqref{Qeques1} are given by
\begin{align*}
& f_{+}=f\chi_{\Op}=\sin(x)\sin(3y),
\quad f_{-}=f\chi_{\Om}=\sin(2x)\sin(y),\\
& g_0=0,
\quad g_1=0,
\quad g=-\sin(2x).
\end{align*}	
Because $g_1=0$, the Poisson interface problem in \eqref{Qeques1} simply becomes $-\nabla^2 u=f-g\delta_\Gamma$ in $\Omega$ with the Dirichlet boundary condition $u|_{\partial \Omega}=g_0$.
The numerical results are provided in \cref{table:QSp8} and \cref{fig:figure6}.		 
\end{example}

\begin{table}[htbp]
	\caption{Performance in \cref{ex11} and  \cref{ex12} of the proposed sixth order compact finite difference scheme  in \cref{thm:regular,fluxtm2} on uniform Cartesian meshes with $h=2^{-J}\times 3.5$ and $h=2^{-J}\times 2\pi$ respectively.}
	\centering
	\setlength{\tabcolsep}{2.5mm}{
		 \begin{tabular}{c|c|c|c|c|c|c|c|c}
			\hline
			\multicolumn{1}{c|}{} &
			 \multicolumn{4}{c|}{\cref{ex11}} &
			 \multicolumn{4}{c}{\cref{ex12}} \\
			\cline{1-9}
$J$&   $\frac{\|u_{h}-u_{h/2}\|_2}{\|u_{h/2}\|_2}$    &order &   $\|u_{h}-u_{h/2}\|_\infty$    &order &   $\frac{\|u_{h}-u_{h/2}\|_2}{\|u_{h/2}\|_2}$    &order &   $\|u_{h}-u_{h/2}\|_{\infty}$    &order \\
			\hline
3   &6.16E+00   &0   &3.94E-01   &0   &2.38E+00   &0   &9.54E-01   &0\\
4   &5.82E-02   &6.726   &4.65E-03   &6.406   &3.94E-03   &9.236   &2.36E-03   &8.662\\
5   &5.07E-04   &6.844   &3.43E-05   &7.085   &8.37E-05   &5.556   &8.19E-05   &4.846\\
6   &4.07E-06   &6.959   &3.13E-07   &6.773   &7.57E-06   &3.468   &1.57E-05   &2.382\\
7   &2.38E-07   &4.098   &6.00E-08   &2.383   &1.54E-06   &2.301   &6.39E-06   &1.299\\
			\hline
	\end{tabular}}
	\label{table:QSp8}
\end{table}

\begin{figure}[htbp]
	\centering
   \begin{subfigure}[b]{0.45\textwidth}
		 \includegraphics[width=8cm,height=8cm]{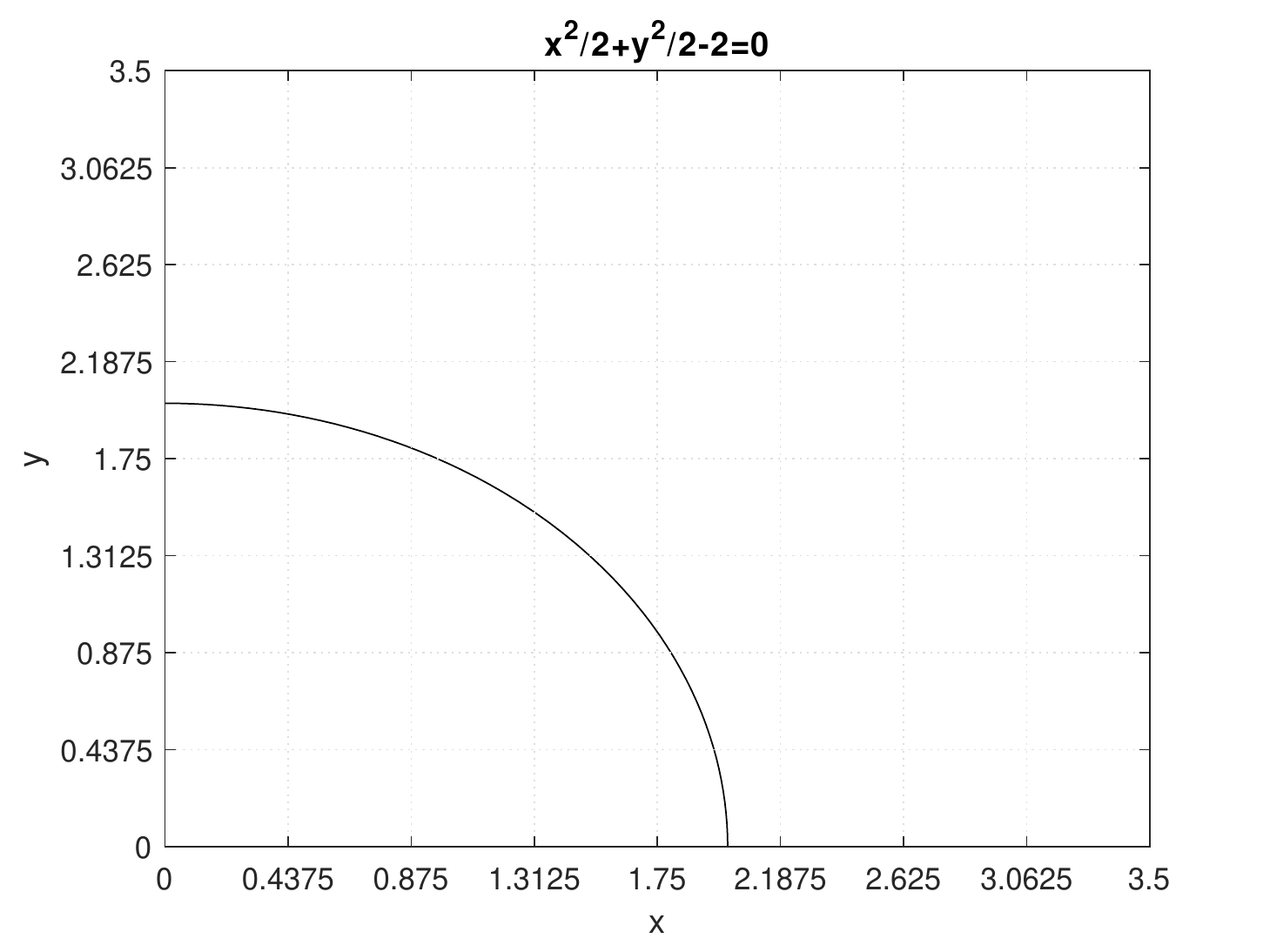}
	\end{subfigure}
	 \begin{subfigure}[b]{0.45\textwidth}
		 \includegraphics[width=8cm,height=8cm]{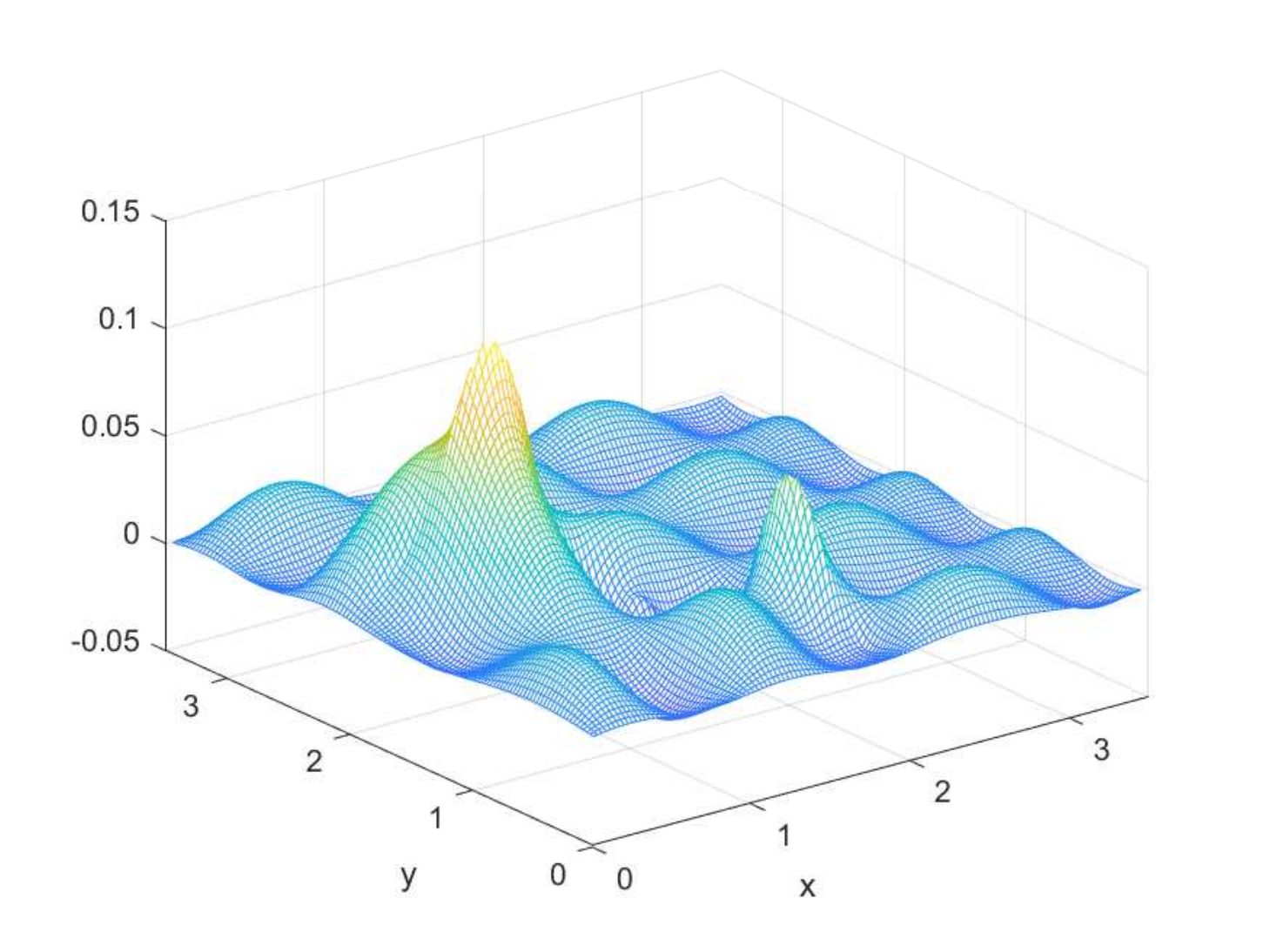}
	\end{subfigure}	
	 \begin{subfigure}[b]{0.45\textwidth}
		 \includegraphics[width=8cm,height=8cm]{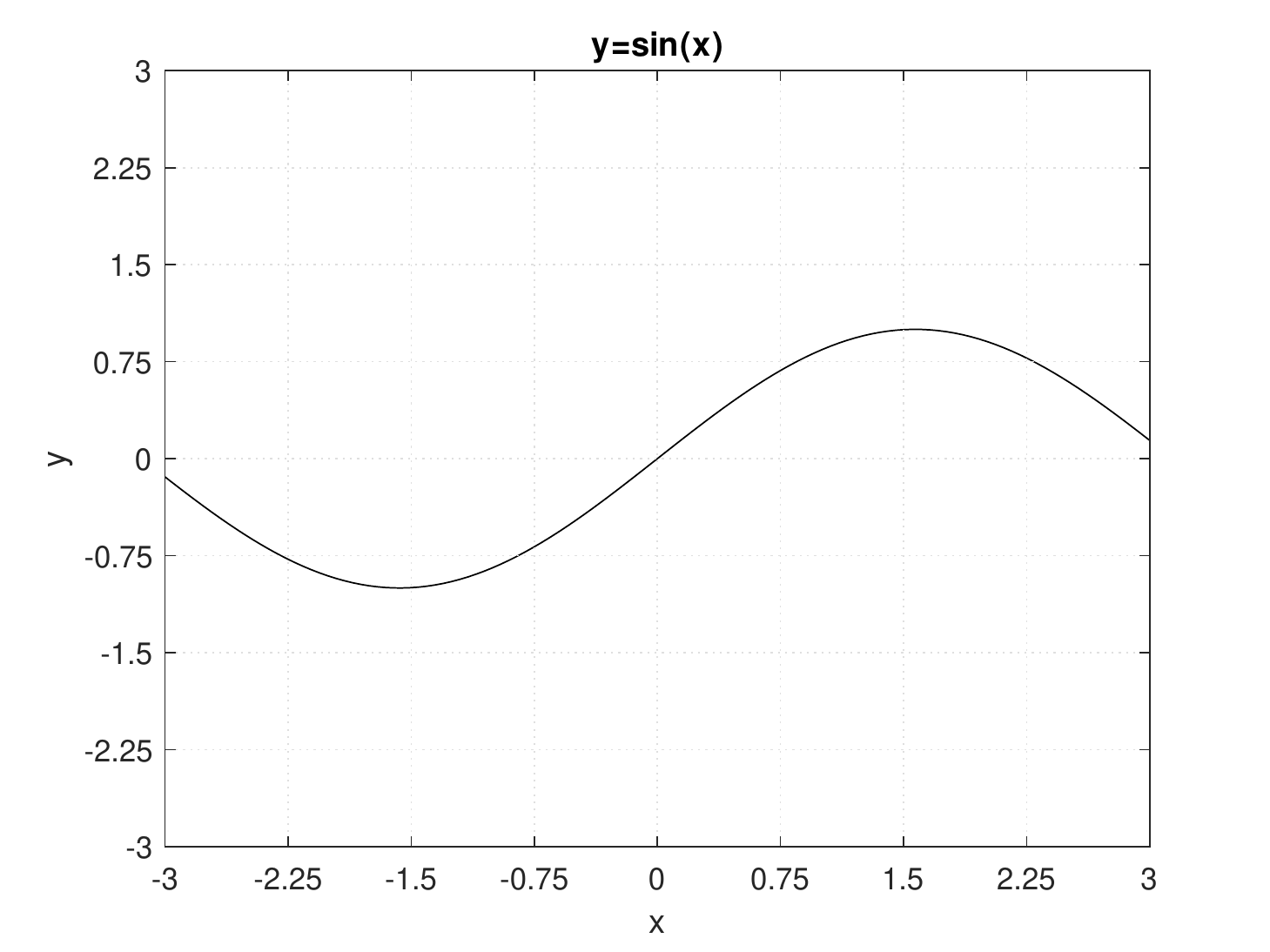}
	\end{subfigure}
	 \begin{subfigure}[b]{0.45\textwidth}
		 \includegraphics[width=8cm,height=8cm]{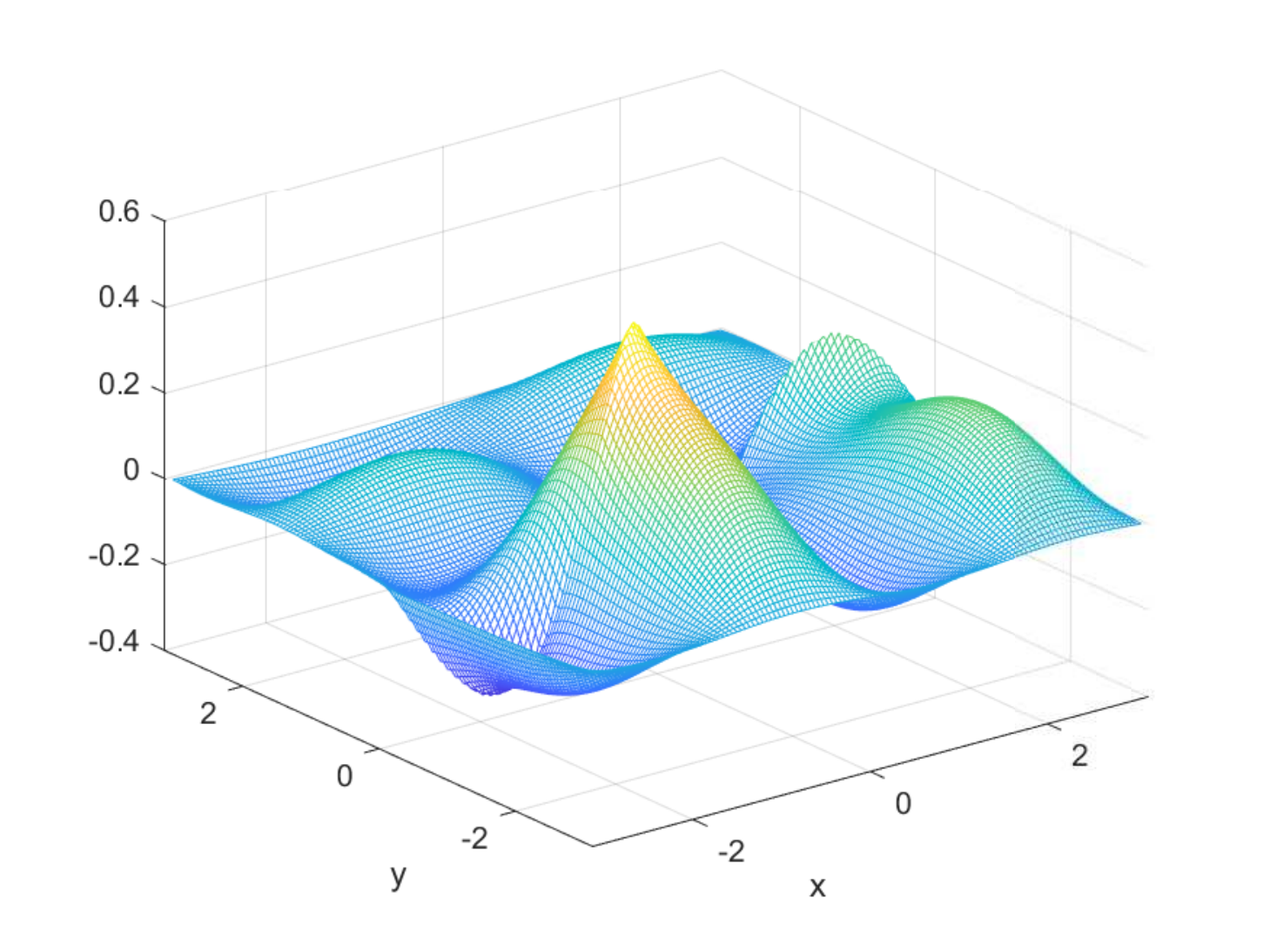}
	\end{subfigure}	
	\caption
	{Top row for \cref{ex11}: the interface curve $\Gamma$ (left) and the numerical solution $u_h$ (right) with $h=2^{-7}\times 3.5$. Bottom row for \cref{ex12}: the interface curve $\Gamma$ (left) and the numerical solution $u_h$ (right) with $h=2^{-7}\times 2\pi$.}
		\label{fig:figure6}
\end{figure}

\begin{remark} \normalfont
	For \cref{ex12}, the interface $\Gamma$ is $y=\sin(x)$ which is shown in \cref{fig:figure6}. Since $\Gamma \cap \partial \Omega\ne\emptyset$ and the angle between $\Gamma$ and $\partial \Omega$ is not $\pi/2$, the solution will contain a singular function which will affect the convergence rate which is shown in \cref{table:QSp8}.
\end{remark}

\section{Conclusion}\label{sec:Conclu}

To our best knowledge, so far there were no compact finite difference schemes available in the literature, that can achieve fifth or sixth order
for Poisson interface problems with singular source terms \eqref{Qeques1}.
Our contribution of this paper is that, we construct the sixth order compact finite difference schemes on uniform meshes for \eqref{Qeques1} with two non-homogeneous jump conditions and provide
explicit formulas for the coefficients of the linear equations. The explicit formulas are independent on how the interface curve partitions the nine points in a stencil,  so one can handle the $72$ different cases configurations of the nine-point stencil with respect to the interface. The matrix $A$ of the linear equations $Ax=b$, appearing after the discretization, is fixed for any source terms, two jump conditions and interface curves, and this allows for an easy
design of preconditioners if iterative methods are used for the solution of the linear system associated with interface problems.  It also allows to perform a single LU decomposition if a direct method is to be used, and solve the problem with multiple source terms/ interface conditions efficiently. This is particularly useful in case of moving boundary problems. Our numerical experiments confirm the flexibility and the sixth order accuracy in $l_2$ and $l_{\infty}$ norms of the proposed schemes.

\appendix

\section{Proof of \cref{thm:interface}}
\label{sec:proof}

\begin{proof}
Since the tangent vector at $t$ of the curve $\Gamma$ parameterized by \eqref{parametric} is given by $(x',y')=(r'(t),s'(t))$,
the unit normal vector $\nv(r(t)+x_i^*,s(t)+y_j^*)$ at the point $(r(t)+x_i^*,s(t)+y_j^*)$ pointing from $\Om$ to $\Op$ is given by one of
\[
\nv(r(t)+x_i^*,s(t)+y_j^*)=\pm\frac{(y',-x')}{\sqrt{(r'(t))^2+(s'(t))^2}}= \pm\frac{(s'(t),-r'(t))}{\sqrt{(r'(t))^2+(s'(t))^2}}.
\]
Let us firstly consider
\be \label{nvector}
\nv(r(t)+x_i^*,s(t)+y_j^*)=\frac{(s'(t),-r'(t))}{\sqrt{(r'(t))^2+(s'(t))^2}}.
\ee
Now we shall use the interface conditions in \eqref{Qeques1}. Plugging the parametric equation in \eqref{parametric} into the interface condition $[u]=g_1$ on $\Gamma$,
near the base point $(x_i^*,y_j^*)$ we have
\be \label{interface:u}
u_+(r(t)+x_i^*,s(t)+y_j^*)-u_-(r(t)+x_i^*,s(t)+y_j^*)=g_1(r(t)+x_i^*,s(t)+y_j^*),
\ee
for $t\in (-\epsilon,\epsilon)$.
Similarly, for flux, we have
\small{
\begin{align*}
(\nabla u_+)(r(t)+x_i^*,s(t)+y_j^*)\cdot \nv(r(t)+x_i^*&,s(t)+y_j^*)-(\nabla u_-)(r(t)+x_i^*,s(t)+y_j^*)\cdot \nv(r(t)+x_i^*,s(t)+y_j^*)\\
&=g(r(t)+x_i^*,s(t)+y_j^*),
\end{align*}
}
for $t\in (-\epsilon,\epsilon)$.
Using the unit norm vector in
\eqref{nvector}, the above relation becomes
\be \label{interface:flux}
\begin{split}
\big((\nabla u_+)(r(t)+x_i^*,s(t)+y_j^*)-
&(\nabla u_-)(r(t)+x_i^*,s(t)+y_j^*)\big) \cdot (s'(t),-r'(t))\\
&=g(r(t)+x_i^*,s(t)+y_j^*)\sqrt{(r'(t))^2+(s'(t))^2},
\end{split}
\ee
for $t\in (-\epsilon,\epsilon)$.
Since all involved functions in \eqref{interface:u} and \eqref{interface:flux} are assumed to be smooth,
to link the two sets $u_+^{(m,n)}$ and $u_-^{(m,n)}$ for $(m,n)\in \ind_{M+1}^1$,
we now take the Taylor approximation of the above functions near the base parameter $t=0$.
Using the identity in \eqref{u:approx:ir:key}, we have
\begin{align*}
&u_\pm (r(t)+x_i^*,s(t)+y_j^*)\\
&=\sum_{(m,n)\in \ind_{M+1}^1}
u_\pm^{(m,n)} G_{m,n}(r(t),s(t))
+\sum_{(m,n)\in \ind_{M-1}} f_\pm^{(m,n)}
H_{m,n}(r(t), s(t))+\bo(t^{M+2})\\
&=\sum_{p=0}^{M+1}
\left(\sum_{(m,n)\in \ind_{M+1}^1}
u_\pm^{(m,n)} g_{m,n,p}
+\sum_{(m,n)\in \ind_{M-1}} f_\pm^{(m,n)}
h_{m,n,p} \right) t^p+\bo(t^{M+2}),
\end{align*}
where the constants $g_{m,n,p}$ and $h_{m,n,p}$ only depend on $r^{(\ell)}(0)$ and $s^{(\ell)}(0)$ for $\ell=0,\ldots,M+1$, and are uniquely determined by
\[
G_{m,n}(r(t),s(t))-\sum_{p=0}^{M+1} g_{m,n,p} t^p=\bo(t^{M+2})
\quad\mbox{and}\quad
H_{m,n}(r(t),s(t))-\sum_{p=0}^{M+1} h_{m,n,p} t^p=\bo(t^{M+2}),\quad t\to0.
\]
More precisely,
\be\label{gmnhmn}
g_{m,n,p}:=\frac{1}{p!} \frac{d^p(G_{m,n}(r(t),s(t)))}{dt^p}\Big|_{t=0},\qquad h_{m,n,p}:=\frac{1}{p!}\frac{d^p(H_{m,n}(r(t),s(t)))}{dt^p}\Big|_{t=0},\quad p=0,\ldots,M+1.
\ee
Similarly, we have
\begin{align*}
g_1(r(t)+x_i^*,s(t)+y_j^*)
&=\sum_{(m,n)\in \ind_{M+1}} \frac{g_1^{(m,n)}}{m!n!} (r(t))^m (s(t))^n
+\bo(t^{M+2})\\
&=\sum_{p=0}^{M+1}\left(\sum_{(m,n)\in \ind_{M+1}} \frac{g_1^{(m,n)}}{m!n!} r_{m,n,p}\right) t^p+\bo(t^{M+2}),
\end{align*}
where the constants $r_{m,n,p}:=\frac{1}{p!}\frac{d^p ((r(t))^m (s(t))^n)}{d t^p}\Big|_{t=0}$ for $p=0,\ldots,M+1$, or equivalently,
\[
(r(t))^m (s(t))^n-\sum_{p=0}^{M+1} r_{m,n,p} t^p=\bo(t^{M+2}),\qquad t\to 0.
\]
Since $G_{m,n}$ is a homogeneous polynomial of degree $m+n$ and because $r(0)=s(0)=0$,
we must have $g_{m,n,p}=0$ for all $0\le p<m+n$ by \eqref{gmnhmn}. Define
\be \label{Umn}
U^{(m,n)}:=u_+^{(m,n)}-u_-^{(m,n)},\qquad (m,n)\in \ind_{M+1}^1.
\ee
Consequently, we deduce from \eqref{interface:u} that
\be \label{interface:u:01}
\sum_{(m,n)\in \ind_{M+1}^1}
U^{(m,n)}g_{m,n,p}
=\sum_{(m,n)\in \ind_{M+1}^1}\big(u_+^{(m,n)}-u_-^{(m,n)}\big)g_{m,n,p}
=F_p,\qquad p=0,\ldots,M+1,
\ee
where $F_0:=g_1^{(0,0)}$ and
\[
F_p:=\sum_{(m,n)\in \ind_{M-1}} \left(f_-^{(m,n)}-f_+^{(m,n)}\right)
h_{m,n,p}
+\sum_{(m,n)\in \ind_{M+1}}
\frac{g_1^{(m,n)}}{m!n!} r_{m,n,p}, \qquad p=1,\ldots,M+1.
\]
Note that $g_{0,0,0}=1$ and $g_{m,n,p}=0$ for all $0\le p< m+n$. We observe that the identities in \eqref{interface:u:01} can be equivalently rewritten as
%
\be \label{interface:u:0}
U^{(0,0)}
=F_0=g_1^{(0,0)},
\ee
and
\be \label{interface:u:1}
U^{(0,p)}g_{0,p,p}+U^{(1,p-1)}g_{1,p-1,p}
=F_p-
\sum_{(m,n)\in \ind_{M+1}^1, m+n<p}
U^{(m,n)} g_{m,n,p},\qquad p=1,\ldots,M+1.
\ee
On the other hand, we obtain from \eqref{u:approx:ir:key} that
\be \label{u:approx:ir:grad}
\nabla u_\pm (x+x_i^*,y+y_j^*)
=\sum_{(m,n)\in \ind_{M+1}^1}
u_\pm^{(m,n)} \nabla G_{m,n}(x,y) +\sum_{(m,n)\in \ind_{M-1}}
f_\pm ^{(m,n)} \nabla H_{m,n}(x,y)+\bo(h^{M+1}),
\ee
for $x,y\in (-2h,2h)$.
Using \eqref{u:approx:ir:grad} and a similar argument, we have
\small{
\begin{align*}
&\nabla u_\pm (r(t)+x_i^*, s(t)+y_j^*)\cdot(s'(t),-r'(t))\\
&\quad =\sum_{(m,n)\in \ind_{M+1}^1}
u_\pm^{(m,n)} \nabla G_{m,n}(r(t),s(t)) \cdot(s'(t),-r'(t))
+\sum_{(m,n)\in \ind_{M-1}}
f_\pm ^{(m,n)} \nabla H_{m,n}(r(t),s(t))\cdot(s'(t),-r'(t))\\
&\quad =
\sum_{p=0}^{M}
\left(
\sum_{(m,n)\in \ind_{M+1}^1}
u_\pm^{(m,n)} \tilde{g}_{m,n,p}
+\sum_{(m,n)\in \ind_{M-1}}
f_\pm ^{(m,n)} \tilde{h}_{m,n,p}
\right) t^p+\bo(t^{M+1}),
\end{align*}
}
where the constants $\tilde{g}_{m,n,p}$ and $\tilde{h}_{m,n,p}$ are uniquely determined by
\begin{align*}
&\nabla G_{m,n}(r(t),s(t)) \cdot(s'(t),-r'(t)) -\sum_{p=0}^{M} \tilde{g}_{m,n,p} t^p=\bo(t^{M+1}),\quad t\to 0,\\
&\nabla H_{m,n}(r(t),s(t)) \cdot(s'(t),-r'(t)) -\sum_{p=0}^{M} \tilde{h}_{m,n,p} t^p=\bo(t^{M+1}), \quad t\to 0.
\end{align*}
More precisely, for $p=0,\ldots,M$,
\be\label{gmnhmn2}
\tilde{g}_{m,n,p}:=\frac{1}{p!}\frac{d^p(\nabla G_{m,n}(r(t),s(t)) \cdot(s'(t),-r'(t)))}{dt^p}\Big|_{t=0},
\ee
\be\label{gmnhmn3}
\tilde{h}_{m,n,p}:=\frac{1}{p!}\frac{d^p(\nabla H_{m,n}(r(t),s(t)) \cdot(s'(t),-r'(t)))}{dt^p}\Big|_{t=0}.
\ee
Note that each entry of $\nabla G_{m,n}$ is a homogeneous polynomial of degree $m+n-1$. By $r(0)=s(0)=0$ and \eqref{gmnhmn2}, we observe that $\tilde{g}_{m,n,p}=0$ for all $0\le p<m+n-1$.
Similarly, we have
\begin{align*}
g(r(t)+x_i^*,s(t)+y_j^*)\sqrt{(r'(t))^2+(s'(t))^2}
&=\sum_{(m,n)\in \ind_{M}} \frac{g^{(m,n)}}{m!n!} (r(t))^m (s(t))^n \sqrt{(r'(t))^2+(s'(t))^2}
+\bo(t^{M+1})\\
&=\sum_{p=0}^{M}\left(\sum_{(m,n)\in \ind_M} \frac{g^{(m,n)}}{m!n!} \tilde{r}_{m,n,p}\right) t^p+\bo(t^{M+1}),
\end{align*}
as $t\to 0$,
where the constants $\tilde{r}_{m,n,p}$ for $p=0,\ldots,M$ are uniquely determined by
\[
(r(t))^m (s(t))^n\sqrt{(r'(t))^2+(s'(t))^2} -\sum_{p=0}^{M} \tilde{r}_{m,n,p} t^p=\bo(t^{M+1}),\qquad t\to 0.
\]
Consequently, \eqref{interface:flux} implies that for all
$p=0,\ldots,M$,
\be \label{interface:flux:0}
\sum_{(m,n)\in \ind_{M+1}^1}
U^{(m,n)}\tilde{g}_{m,n,p}
=
\sum_{(m,n)\in \ind_{M+1}^1}
\left(u_+^{(m,n)}-u_-^{(m,n)}\right) \tilde{g}_{m,n,p}=G_p,\qquad p=0,\ldots,M,
\ee
where
\[
G_p:=\sum_{(m,n)\in \ind_{M-1}} \left(f_-^{(m,n)}-f_+^{(m,n)}\right)
\tilde{h}_{m,n,p}
+\sum_{(m,n)\in \ind_{M}}
\frac{g^{(m,n)}}{m!n!} \tilde{r}_{m,n,p}.
\]
Note that $\tilde{g}_{0,0,0}=0$ and $\tilde{g}_{m,n,p}=0$ for all $0\le p< m+n-1$. We observe that the identities in \eqref{interface:flux:0} can be equivalently rewritten as
\be \label{interface:flux:1}
U^{(0,p)}\tilde{g}_{0,p,p-1}
+U^{(1,p-1)} \tilde{g}_{1,p-1,p-1}
=G_{p-1}-
\sum_{(m,n)\in \ind_{M+1}^1, m+n<p}
U^{(m,n)}\tilde{g}_{m,n,p-1},\quad p=1,\ldots,M+1.
\ee
Using our assumption $(r'(0))^2+(s'(0))^2>0$ in \eqref{parametric},
we now claim that
\be \label{nonzero}
g_{0,p,p}\tilde{g}_{1,p-1,p-1}- g_{1,p-1,p}
\tilde{g}_{0,p,p-1}>0,\qquad \forall\; p=1,\ldots,M.
\ee
%
Since the polynomial $G_{m,n}$ in \eqref{Gmn} is a homogeneous polynomial of degree $m+n$, we observe
\be \label{gtg1}
g_{m,n,m+n}= G_{m,n}(r'(0),s'(0)), \qquad (m,n)\in \ind_{M+1}^1.
\ee
From the definition of $G_{m,n}(x,y)$ in \eqref{Gmn}, we particularly have
\be \label{gmnp1}
g_{0,p,p}=\sum_{\ell=0}^{\lfloor \frac{p}{2}\rfloor}
(-1)^\ell \frac{(r'(0))^{2\ell} (s'(0))^{p-2\ell}}{(2\ell)!(p-2\ell)!} \qquad \mbox{and}  \qquad
g_{1,p-1,p}=\sum_{\ell=0}^{\lfloor \frac{p-1}{2}\rfloor}
(-1)^\ell \frac{(r'(0))^{1+2\ell} (s'(0))^{p-1-2\ell}}{(1+2\ell)!(p-1-2\ell)!}.
\ee
Clearly,
\be \label{formulforp1}
\Big\lfloor\frac{p}{2}\Big\rfloor=\begin{cases}
	\lfloor\frac{p-1}{2}\rfloor+1, &\text{if $p$ is even},\\
	\lfloor\frac{p-1}{2}\rfloor, &\text{if $p$ is odd},
\end{cases}
\qquad
\mbox{and} \qquad 2\Big\lfloor\frac{p-1}{2}\Big\rfloor+1=p, \text{ if $p$ is odd}.
\ee
Similarly, we also have
\be \label{gtg2}
\tilde{g}_{m,n,m+n-1}= \nabla G_{m,n}(r'(0),s'(0))\cdot (s'(0),-r'(0)),\qquad (m,n)\in \ind_{M+1}^1.
\ee
From the definition of $G_{m,n}(x,y)$ in \eqref{Gmn}, we deduce that
\be\label{gtildemn}
\begin{split}
	\tilde{g}_{0,p,p-1}&= \nabla G_{0,p}(r'(0),s'(0))\cdot (s'(0),-r'(0))\\
	&=\sum_{\ell=1}^{ \lfloor\frac{p}{2}\rfloor}(-1)^\ell \frac{(r'(0))^{2\ell-1} (s'(0))^{p+1-2\ell}}{(2\ell-1)!(p-2\ell)!}
	 -\sum_{\ell=0}^{\lfloor\frac{p-1}{2}\rfloor}(-1)^\ell \frac{(r'(0))^{2\ell+1} (s'(0))^{p-2\ell-1}}{(2\ell)!(p-2\ell-1)!}\\
	 &=-\sum_{\ell=0}^{\lfloor\frac{p}{2}\rfloor-1}(-1)^{\ell} \frac{(r'(0))^{2\ell+1} (s'(0))^{p-2\ell-1}}{(2\ell+1)!(p-2\ell-2)!}
	 -\sum_{\ell=0}^{\lfloor\frac{p-1}{2}\rfloor}(-1)^\ell \frac{(r'(0))^{2\ell+1} (s'(0))^{p-2\ell-1}}{(2\ell)!(p-2\ell-1)!}.
\end{split}
\ee
By \eqref{gmnp1}, \eqref{formulforp1} and \eqref{gtildemn}, we conclude that
\be\label{tildeggg1}
\tilde{g}_{0,p,p-1}=-p\sum_{\ell=0}^{\lfloor \frac{p-1}{2}\rfloor}
(-1)^\ell \frac{(r'(0))^{2\ell+1} (s'(0))^{p-2\ell-1}}{(2\ell+1)!(p-2\ell-1)!}
=-pg_{1,p-1,p}.
\ee
Similarly,
\be\label{gtildemn2}
\begin{split}
	\tilde{g}_{1,p-1,p-1}&= \nabla G_{1,p-1}(r'(0),s'(0))\cdot (s'(0),-r'(0))\\
	&=\sum_{\ell=0}^{ \lfloor\frac{p-1}{2}\rfloor}(-1)^\ell \frac{(r'(0))^{2\ell} (s'(0))^{p-2\ell}}{(2\ell)!(p-1-2\ell)!}-
\sum_{\ell=0}^{ \lfloor\frac{p}{2}\rfloor-1}(-1)^\ell \frac{(r'(0))^{2\ell+2} (s'(0))^{p-2\ell-2}}{(2\ell+1)!(p-2\ell-2)!}\\
	&=\sum_{\ell=0}^{ \lfloor\frac{p-1}{2}\rfloor}(-1)^\ell \frac{(r'(0))^{2\ell} (s'(0))^{p-2\ell}}{(2\ell)!(p-1-2\ell)!}+
\sum_{\ell=1}^{ \lfloor\frac{p}{2}\rfloor}(-1)^{\ell} \frac{(r'(0))^{2\ell} (s'(0))^{p-2\ell}}{(2\ell-1)!(p-2\ell)!}.
\end{split}
\ee
By \eqref{gmnp1}, \eqref{formulforp1} and \eqref{gtildemn2}, we deduce that
\be\label{tildeggg2}
\tilde{g}_{1,p-1,p-1}=p\sum_{\ell=0}^{ \lfloor\frac{p}{2}\rfloor}(-1)^{\ell} \frac{(r'(0))^{2\ell} (s'(0))^{p-2\ell}}{(2\ell)!(p-2\ell)!}\\
=pg_{0,p,p}.
\ee
By \eqref{tildeggg1} and  \eqref{tildeggg2},
\be \label{detg}
g_{0,p,p}\tilde{g}_{1,p-1,p-1}- g_{1,p-1,p}
\tilde{g}_{0,p,p-1}=p(g_{0,p,p})^2+p(g_{1,p-1,p})^2,\qquad \forall p=1,\ldots,M+1.
\ee
Let
\be	\label{WWq1} W:=(p!)^2(g_{0,p,p})^2+(p!)^2(g_{1,p-1,p})^2, \quad a:=(r'(0))^2 \quad \mbox{and} \quad b:=(s'(0))^2.
\ee
Then
\be\label{WWq2}
\begin{split}
	W&=\Bigg( \sum_{\ell=0}^{\lfloor \frac{p}{2}\rfloor}(-1)^\ell  {p\choose 2\ell} a^{\ell} b^{p/2-\ell}\Bigg)^2
	+\Bigg( \sum_{\ell=0}^{\lfloor \frac{p-1}{2}\rfloor}(-1)^\ell  {p\choose 2\ell+1}  a^{1/2+\ell} b^{(p-1)/2-\ell}\Bigg)^2\\
	&= \sum_{i=0}^{\lfloor \frac{p}{2}\rfloor} \sum_{j=0}^{\lfloor \frac{p}{2}\rfloor} (-1)^{i+j}  {p\choose 2i}{p\choose 2j} a^{i+j} b^{p-i-j}+\sum_{i=0}^{\lfloor \frac{p-1}{2}\rfloor}\sum_{j=0}^{\lfloor \frac{p-1}{2}\rfloor}(-1)^{i+j}  {p\choose 2i+1}{p\choose 2j+1}  a^{1+i+j} b^{p-1-i-j}\\
	&= \sum_{\ell=0}^{2\lfloor \frac{p}{2}\rfloor} \sum_{i=\max(0,\ell-\lfloor \frac{p}{2}\rfloor)}^{\min(\ell,\lfloor \frac{p}{2}\rfloor)} (-1)^{\ell}  {p\choose 2i}{p\choose 2(\ell-i)} a^{\ell} b^{p-\ell}\\
	&\qquad -\sum_{\ell=0}^{2\lfloor \frac{p-1}{2}\rfloor+1}\sum_{i=\max(0,\ell-\lfloor \frac{p+1}{2}\rfloor)}^{\min(\ell-1,\lfloor \frac{p-1}{2}\rfloor)}(-1)^{\ell}  {p\choose 2i+1}{p\choose 2(\ell-i)-1}  a^{\ell} b^{p-\ell}.
\end{split}
\ee
Let us consider the first case: $p$ is even and $\ell \le \lfloor \frac{p}{2}\rfloor$. Then
\be
\begin{split}\label{WWq3}
	&\sum_{i=\max(0,\ell-\lfloor \frac{p}{2}\rfloor)}^{\min(\ell,\lfloor \frac{p}{2}\rfloor)} (-1)^{\ell}  {p\choose 2i}{p\choose 2(\ell-i)} -\sum_{i=\max(0,\ell-\lfloor \frac{p+1}{2}\rfloor)}^{\min(\ell-1,\lfloor \frac{p-1}{2}\rfloor)}(-1)^{\ell}  {p\choose 2i+1}{p\choose 2(\ell-i)-1} \\
	&=\sum_{i=0}^{\ell} (-1)^{\ell}  {p\choose 2i}{p\choose 2(\ell-i)} -\sum_{i=0}^{\ell-1}(-1)^{\ell}  {p\choose 2i+1}{p\choose 2(\ell-i)-1} \\
	&=(-1)^{\ell} \sum_{i=0,2,4}^{2\ell}  {p\choose i}{p\choose 2\ell-i} -(-1)^{\ell}\sum_{i=1,3,5}^{2\ell-1}  {p\choose i}{p\choose 2\ell-i}=(-1)^{\ell}\sum_{i=0}^{2\ell} (-1)^{i}  {p\choose i}{p\choose 2\ell-i}\\
    &=(-1)^{\ell}\sum_{i=0}^{2\ell} \mbox{coeff}((1-x)^{p},x^{i})\mbox{coeff}((1+x)^{p},x^{2\ell-i})=(-1)^{\ell}\mbox{coeff}((1-x^2)^{p},x^{2\ell})={p\choose \ell},
\end{split}
\ee
where the coeff$(f(x), x^n)$ function extracts the coefficient of $x^n$ in the polynomial $f(x)$.
Similarly, we can prove that other cases can obtain the same result. By \eqref{WWq1}, \eqref{WWq2} and \eqref{WWq3},
\be\label{WWq4}
	W=\sum_{\ell=0}^{p} {p\choose \ell}a^{\ell}b^{p-\ell}=(a+b)^p=\Big((r'(0))^2+(s'(0))^2\Big)^p.
\ee
According to \eqref{detg}, \eqref{WWq1}, \eqref{WWq4} and \eqref{parametric},
\be
g_{0,p,p}\tilde{g}_{1,p-1,p-1}- g_{1,p-1,p}
\tilde{g}_{0,p,p-1} > 0,\qquad \forall p=1,\ldots,M+1.
\ee
%
Consequently, the associated $2\times 2$ coefficient matrix in the linear system in \eqref{interface:u:1} and \eqref{interface:flux:1} is invertible and its inverse is given by
\[
Q_p:=\frac{1}{g_{0,p,p}\tilde{g}_{1,p-1,p-1}- g_{1,p-1,p}
\tilde{g}_{0,p,p-1}}
\left[ \begin{matrix} \tilde{g}_{1,p-1,p-1}
&-g_{1,p-1,p}\\
-\tilde{g}_{0,p,p-1} &g_{0,p,p}\end{matrix}\right].
\]
Hence, the linear equations in \eqref{interface:u:1} and \eqref{interface:flux:1} must have a unique solution $\{U^{(0,p)}, U^{(1,p-1)}\}_{p=1,\ldots M}$,
which can be recursively computed
from $p=1$ to $p=M$ by $U^{(0,0)}=g_1^{(0,0)}$ due to \eqref{interface:u:0} and
\be \label{transmissioncoef}
\left[ \begin{matrix}
U^{(0,p)}\\
U^{(1,p-1)}\end{matrix}\right]
=Q_p
\left[ \begin{matrix}
F_p\\
G_{p-1}
\end{matrix}\right]
-\sum_{n=1}^{p-1}
Q_p \left[
\begin{matrix}U^{(0,n)}g_{0,n,p}+U^{(1,n-1)}g_{1,n-1,p}\\
U^{(0,n)}\tilde{g}_{0,n,p-1}
+U^{(1,n-1)}\tilde{g}_{1,n-1,p-1}
\end{matrix}\right],\qquad p=1,\ldots,M+1.
\ee
Note that for $p=1$, the above summation $\sum_{n=1}^{p-1}$ is empty.

If the normal vector $\nv$ in \eqref{nvector} gives the direction from $\Op$ to $\Om$, then we only need to add a negative sign to all $\tilde{r}_{m,n,p}$.
Since $U^{(m,n)}=u_+^{(m,n)}-u_-^{(m,n)}$,
the identities in \eqref{transmissioncoef} and \eqref{interface:u:0} prove all the claims.
\end{proof}


\begin{thebibliography}{99}

	
	\bibitem{BAGMO96} I.~Babu\v{s}ka, B.~Andersson, B.~Guo, J.~M.~Melenk and H.~S.~Oh,
	Finite element method for solving problems with
	singular solutions.
	\emph{J. Comput. Appl. Math.} \textbf{74} (1996), no. 1-2, 51-70.
	
		\bibitem{Blumen85} 
{M.~Blumenfeld,
	The regularity of interface-problems on corner-regions.	\emph{Lecture Notes in Math.} \textbf{1121} (1985), 38-54.}
	
	
		\bibitem{BG15} G.~Brandstetter and S.~Govindjee,
	A high-order immersed boundary discontinuous-Galerkin method
	for Poisson’s equation with discontinuous coefficients
	and singular sources. \emph{Int. J. Numer. Meth. Engng.} \textbf{101} (2015), no. 11, 847-869.
	

	
	\bibitem{CaKi01} Z.~Cai and S.~Kim,
	A finite element method using singular functions for the Poisson equation: corner singularities.
	\emph{SIAM J. Numer. Anal.} \textbf{39} (2001), no. 1,  286-299.
	
	\bibitem{CY93}
 J.~M.~Chadam and H.~M.~Yin, A diffusion equation with localized chemical reactions. \emph{Proc. Edinburgh Math. Soc.} \textbf{37} (1993), 101-118.	
	

	\bibitem{CFL19}
X.~Chen, X.~Feng and Z.~Li,
A direct method for accurate solution and gradient
computations for elliptic interface problems. \emph{Numer. Algorithms.} \textbf{80} (2019), 709-740.

	\bibitem{DFL20}
B.~Dong, X.~Feng and Z.~Li,
An FE-FD method for anisotropic elliptic interface problems. \emph{SIAM J. Sci. Comput.} \textbf{42} (2020), no. 4, B1041-B1066.
	
	
	\bibitem{EG17} R.~Egan and F.~Gibou,
	Geometric discretization of the multidimensional Dirac delta
	distribution-Application to the Poisson equation with
	singular source terms. \emph{J. Comput. Phys.} \textbf{346} (2017), 71-90.
	
\bibitem{Guermond} A.~Ern and J.-L.~Guermond,
	Theory and practice of finite elements. \emph{Vol. 159. Springer Science \& Business Media}, 2013.


	
		\bibitem{FLQ11}
	X.~Feng, Z.~Li and Z.~Qaio, High order compact finite difference schemes for the Helmholtz equation with discontinuous coefficients. \emph{J. Comput. Math.} \textbf{29} (2011), no. 3, 324-340.
		
	
	\bibitem{FGW73} G.~J.~Fix, S.~Gulati and G.~I.~Wakoff,
	On the use of singular functions with finite element approximations.
	\emph{J. Comput. Phys.} \textbf{13} (1973), 209-228.
	

	\bibitem{GO94} B.~Guo and H.~S.~Oh,
The h--p version of the finite element method for problems with interfaces.
\emph{Int. J. Numer. Methods. Eng.} \textbf{37} (1994), no. 10, 1741-1762.


  \bibitem{HW97} T.~Y.~Hou and X.~Wu, A multiscale finite element method for elliptic problems in composite materials and porous media.
\emph{J. Comput. Phys.} \textbf{134} (1997), 169-189.
	
	
	\bibitem{JCL18}
	H.~Ji, J.~Chen and Z.~Li, A high-order source removal finite element method for a class
	of elliptic interface problems. \emph{Appl. Numer. Math.} \textbf{130} (2018), 112-130.
	


	\bibitem{KV07} J.~D.~Kandilarov and L.~G.~Vulkov,
The immersed interface method for two-dimensional heat-diffusion
equations with singular own sources. \emph{Appl. Numer. Math.} \textbf{57} (2007), no. 5-7,486-497.

\bibitem{Kell71} R.~B.~Kellogg,
Singularities in interface problems.
\emph{Numerical Solution of Partial Differential Equations-\uppercase\expandafter{\romannumeral2}} (1971), 351-400.

\bibitem{Kell75}  R.~B.~Kellogg,
On the Poisson equation with intersecting interfaces.
\emph{Appl. Anal.} \textbf{4} (1975), no. 2,  101-129.

\bibitem{Kell72} R.~B.~Kellogg,
Higher order singularities for interface problems.
\emph{The Mathematical Foundations of the Finite Element Method with Applications to Partial Differential Equations} (1972), 589-602.



\bibitem{KCPK07} S.~Kim, Z.~Cai, J.~Pyo and S.~Kong,
A finite element method
using singular functions: interface problems.
\emph{Hokkaido Math. J.} \textbf{36} (2007), no. 4, 815-836.


\bibitem{KiKo06} S.~Kim and S.~Kong,
A finite element method dealing the singular points with a cut-off function.
\emph{J. Appl. Math. Comput.} \textbf{21} (2006), no. 1-2,  141-152.


			\bibitem{LeLi94}
	R.~J.~ Leveque and Z.~Li,
	The Immersed interface method for elliptic equations with discontinuous coefficients and singular sources. \emph{SIAM J. Numer. Anal.} \textbf{31} (1994), no. 4, 1019-1044.
	
			\bibitem{LeLi97}
R.~J.~ Leveque and Z.~Li,
Immersed interface methods for stokes flow with elastic boundaries or surface tension. \emph{SIAM J. Sci. Comput.} \textbf{18} (1997), no. 3, 709-735.

	
	
	\bibitem{LiIto06}
Z.~Li and K.~Ito, The immersed interface method: numerical solutions of PDEs involving interfaces and irregular domains. \emph{Society for Industrial and Applied Mathematics}. 2006.	
	

	\bibitem{Li98}
	Z.~Li,
	A fast iterative algorithm for elliptic interface problems. \emph{SIAM J. Numer. Anal.} \textbf{35} (1998), no. 1, 230-254.
	


	\bibitem{Li96}
	Z.~Li,
	A note on immersed interface method for
	three-dimensional elliptic equations. \emph{Comput. Math. with Appl.} \textbf{31} (1996), no. 3, 9-17.
	


\bibitem{Nic90} S.~Nicaise,
Polygonal interface problems: higher regularity
results.
\emph{Commun. Partial. Differ. Equ.} \textbf{15} (1990), no. 10, 1475-1508.


\bibitem{NS94} S.~Nicaise and  A.~M.~Sandig,
General interface problems-\uppercase\expandafter{\romannumeral1}.
\emph{Math. Methods Appl. Sci.} \textbf{17} (1994), 395-429.

\bibitem{Petz01} M.~Petzoldt,
Regularity results for Laplace interface problems in two dimensions.
\emph{J. Anal. Appl.} \textbf{20} (2001), no. 2,  431-455.

\bibitem{Petz02} M.~Petzoldt,
A posteriori error estimators for elliptic equations with
discontinuous coefficients.
\emph{Adv. Comput. Math.} \textbf{16} (2002), no. 1, 47-75.

\bibitem{PE11} D.~A.~D.~Pietro and A.~Ern, Mathematical aspects of discontinuous Galerkin methods. \emph{Springer Science and Business Media.} 2011.	

\bibitem{Samarskii} A.~A.~Samarskii, V.~L.~Makarov, and R.~D.~Lazarov, Difference schemes for differential equations with generalized solutions,
	Vysshaya Shkola, (Russian), Moscow (1987)

		\bibitem{SDKS13} S.~O.~Settle, C.~C.~Douglas, I.~Kim, and D.~Sheen,
On the derivation of highest-order compact finite difference schemes for the one- and two-dimensional Poisson equation with Dirichlet boundary conditions.
\emph{SIAM J. Numer. Anal.} \textbf{51} (2013), no. 4,  2470-2490.

\bibitem{TE04} A.~K.~Tornberg and B.~Engquist,
Numerical approximations of singular source terms
in differential equations. \emph{J. Comput. Phys.} \textbf{200} (2004), no. 2, 462-488.

\bibitem{Towers10} J.~D.~Towers,
Finite difference methods for discretizing singular source terms in a Poisson interface problem. \emph{Contemp. Math.} \textbf{526} (2010), 359-389.
	

\bibitem{Vazqu07} J.~L.~Vazquez, The Porous medium equation: mathematical theory. \emph{Clarendon Press.} 2007. p15.
	
		\bibitem{WZ09} Y.~Wang and J.~Zhang,
Sixth order compact scheme combined with multigrid method
and extrapolation technique for 2D Poisson equation.
\emph{J. Comput. Phys.} \textbf{228} (2009), no. 1,  137-146.
	
\bibitem{WB00} A.~Wiegmann and K.~P.~Bube,
The explicit-jump immersed interface method: finite difference methods for PDEs with piecewise smooth solutions.
\emph{SIAM J. Numer. Anal.} \textbf{37} (2000), no. 3,  827-862.


\bibitem{ZFH13} S.~Zhai, X.~Feng and Y.~He,
A family of fourth-order and sixth-order compact difference schemes for the three-dimensional Poisson equation.
\emph{J. Sci. Comput.} \textbf{54} (2013),  97-120.

	\bibitem{ZZFW06} Y.~C.~Zhou, S.~Zhao, M.~Feig and G.~W.~Wei,
High order matched interface and boundary method for elliptic
equations with discontinuous coefficients and singular sources. \emph{J. Comput. Phys.} \textbf{213} (2006), no. 1, 1-30.

\bibitem{P02} C.~S.~Peskin, The immersed boundary method. \emph{Acta Numerica} (2002), 479-517.

\end{thebibliography}
\end{document}